\documentclass{article}
\usepackage[utf8]{inputenc}
\usepackage{amsfonts,amssymb,amsmath,amsthm}
\usepackage[a4paper, total={6in,9in}]{geometry}

\usepackage[utf8]{inputenc}
\usepackage{graphicx}
\usepackage{mathtools}

\usepackage{subcaption}
\usepackage{tikz}
\usepackage{enumerate}
\usepackage{cases}
\usepackage{multicol}
\usepackage{sidecap}
\usepackage{float}
\usepackage{xcolor}
\definecolor{blue2}{rgb}{0.67, 0.9, 0.93}
\usepackage[color=blue2]{todonotes}
\usepackage{cases}
\usepackage{extarrows}
\usepackage{soul}

\usepackage[normalem]{ulem}
\usepackage[hidelinks=true]{hyperref}
\usepackage{setspace}
\numberwithin{equation}{section}
\newtheorem{theorem}{Theorem}[section]
\newtheorem{lemma}[theorem]{Lemma}
\newtheorem{proposition}[theorem]{Proposition}

\theoremstyle{definition}

\newtheorem{example}[theorem]{Example}
\newtheorem{corollary}[theorem]{Corollary}

\newcommand{\N}{{\mathbb N}}
\newcommand{\R}{{\mathbb R}}

\newcommand{\Z}{\mathbb{Z}}

\usepackage{xspace}
\newcommand{\Hg}{$\mathrm{(Hg)}$\xspace}
\newcommand{\Hgone}{$\mathrm{(Hg1)}$\xspace}
\newcommand{\Hgtwo}{$\mathrm{(Hg2)}$\xspace}
\newcommand{\Hd}{$\mathrm{(Hd)}$\xspace}

\usepackage{authblk}
\title{\sc Propagation reversal for bistable differential equations on trees}
\author[1]{Hermen Jan Hupkes \thanks{\tt hhupkes@math.leidenuniv.nl}}
\author[1]{Mia Juki\'c \thanks{\tt m.jukic@math.leidenuniv.nl}}
\author[2]{Petr Stehl\'{\i}k\thanks{\tt pstehlik@kma.zcu.cz}}
\author[2]{Vladim\'{\i}r \v{S}v\'{\i}gler \thanks{corresponding author, \tt sviglerv@kma.zcu.cz}}
\affil[1]{\small Mathematisch Instituut, Universiteit Leiden, P.O. Box 9512, 2300 RA Leiden, The Netherlands}
\affil[2]{\small Department of Mathematics and NTIS, Faculty of Applied Sciences, University of West Bohemia,\authorcr Univerzitn\'\i~8, 306 14 Plze\v{n}\\ Czech Republic}

\begin{document}

\maketitle
%\todo[color=yellow]{Consider \textbf{Diffusion-driven} propagation reversal...}
\begin{abstract}
We study traveling wave solutions to bistable differential equations on infinite $k$-ary trees. %These graphs directly generalize classical square infinite lattices and our results  complement those for bistable lattice equations on $\Z$. 
These graphs  generalize the notion of  classical square infinite lattices and our results  complement those for bistable lattice equations on $\Z$. 
    Using comparison principles and explicit lower and upper solutions, we show that wave-solutions are pinned for small diffusion parameters.  Upon increasing the diffusion, the wave starts to travel with non-zero speed, in a direction that depends on the detuning parameter. However, once the diffusion is sufficiently strong, the wave propagates in a single direction up the tree irrespective of the detuning parameter. %This phenomenon does not happen for the counterpart LDEs posed on the integer lattice. 
    In particular, our results imply that changes to the diffusion parameter can lead to a reversal of the propagation direction. 
    %We study traveling wave solutions to bistable differential equations on infinite $k$-ary trees. These graphs are direct generalization of classical infinite lattices and our results thus complement those for bistable lattice equations. Using comparison principles and constructing lower and upper solutions we show that the waves are pinned for small diffusion parameters. Once the diffusion increases the waves travel in both direction, depending on the bistability. However, once the diffusion exceeds a further threshold, this traditional result does not hold and the waves propagate in a single direction. In particular, our results imply that there exists bistabilities for which the mere change of diffusion can reverse the propagation direction.
\end{abstract}



% 34A33  	Lattice differential equations
% 34K31  	Lattice functional-differential equations
% 37L60  	Lattice dynamics
% 39A12     discrete version of topics in analysis
% 65M22     Solutions of discretized equations

\section{Introduction}
In this paper we consider traveling wave solutions to the scalar  bistable reaction-diffusion-advection lattice differential equation (LDE)
\begin{equation}\label{eqn:intro:kLDE}
\begin{aligned}
        \dot{u}_i & = d(ku_{i+1} - (k+1)u_i + u_{i-1} ) + g(u_i;a) \\
        & = d(u_{i+1} - 2u_i + u_{i-1}) + d(k-1)(u_{i+i} - u_i) +  g(u_i;a), \quad i\in \Z.
\end{aligned}
\end{equation}
%that arises in the study of the so-called layer solutions posed on infinite
%trees.
%$k$-ary trees $\mathcal{T}_k$ with $k>1$. An infinite $k$-ary tree $\mathcal{T}_k$ is an undirected graph in which each node has one parent and exactly $k$ children, see Fig.~\ref{fig:intro:tree}. 
Here  $d>0$ is a diffusion parameter and
the function  $g(u;a)$ is a bistable nonlinearity, such as the cubic
\begin{equation}\label{eqn:intro:cubic}
    g(u;a) = u(1 -u)(u -a), \quad a\in (0,1).
\end{equation} 
As we explain below, the advection parameter $k>0$ can be interpreted as the branch factor of an infinite tree
when it is integer valued.

We focus on the traveling front solutions of the form
\begin{equation}\label{eqn:intro:wave:ansatz:LDE}
    u_i(t) = \Phi(i-ct), \qquad \Phi(-\infty) = 0, \qquad \Phi(\infty) = 1.
\end{equation}
Our primary concern is how the diffusion strength $d>0$, the branch factor $k>0$ and the detuning parameter $a$ influence the sign of the wave-speed $c$. 
Fixing a value of $k > 1$ for convenience, our main results can be summed up into the following three points (illustrated in Fig.~\ref{fig:3_ad_regions}(c)):
\begin{enumerate}[(i)]
    \item  For any  sufficiently small $d>0$, wave pinning occurs in the sense that  $c=0$ for a nonempty range of parameters $a$ (Proposition~\ref{prop:monotonic:existence}). 
    \item As we increase $d$, % while keeping $k$ fixed,
    we have $c<0$ for  all $a\approx 0 $ and    $c>0$ for  all $a\approx 1$ (Theorem~\ref{prop:main_results:c_neg:d_small}). 
    \item For all %$k>1$ and 
    $a\in(0,1)$ %there exists a function $d^*(a,k)$ such that $c<0$  for all $d>d^*(a,k)$ 
    we have $c < 0$ whenever $d$ is sufficiently large (Theorem~\ref{prop:main_results:c_neg:d_large}).
\end{enumerate}
Consequently, these results show that for $a\approx 1$ we can reverse the speed of the wave from $c>0$ to $c<0$ by increasing the diffusion parameter $d$;
see Fig.~\ref{fig:3_ad_regions}(c), Ex.~\ref{ex:propagation:reversal} and Fig.~\ref{fig:ex:propagation:reversal} for illustration.

\paragraph{Layer solutions on $\mathcal{T}_k$}
%In this paper we consider a direct generalization of equation~\eqref{eqn:intro:lattice-LDE} and focus on 
Our primary motivation to study~\eqref{eqn:intro:kLDE} is to further our understanding of 
 reaction-diffusion equations on infinite $k$-ary trees, see Fig.~\ref{fig:intro:tree}.
 \begin{figure}[t]
    \centering
    \begin{subfigure}{0.45\textwidth}
        \centering
        \includegraphics[width=\textwidth]{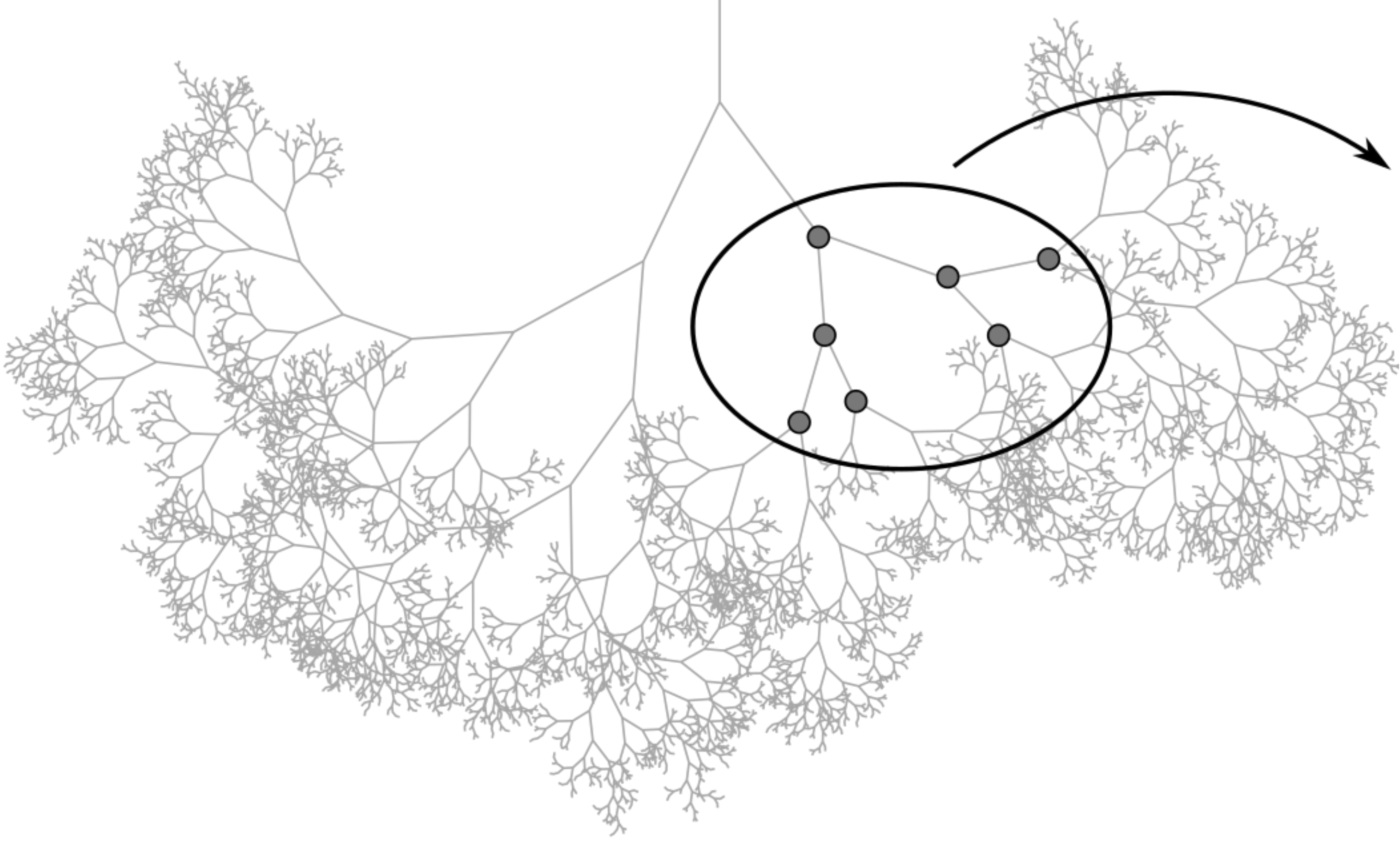}    
    \end{subfigure}
    \begin{subfigure}{0.45\textwidth}
        \centering
        \includegraphics[width=\textwidth]{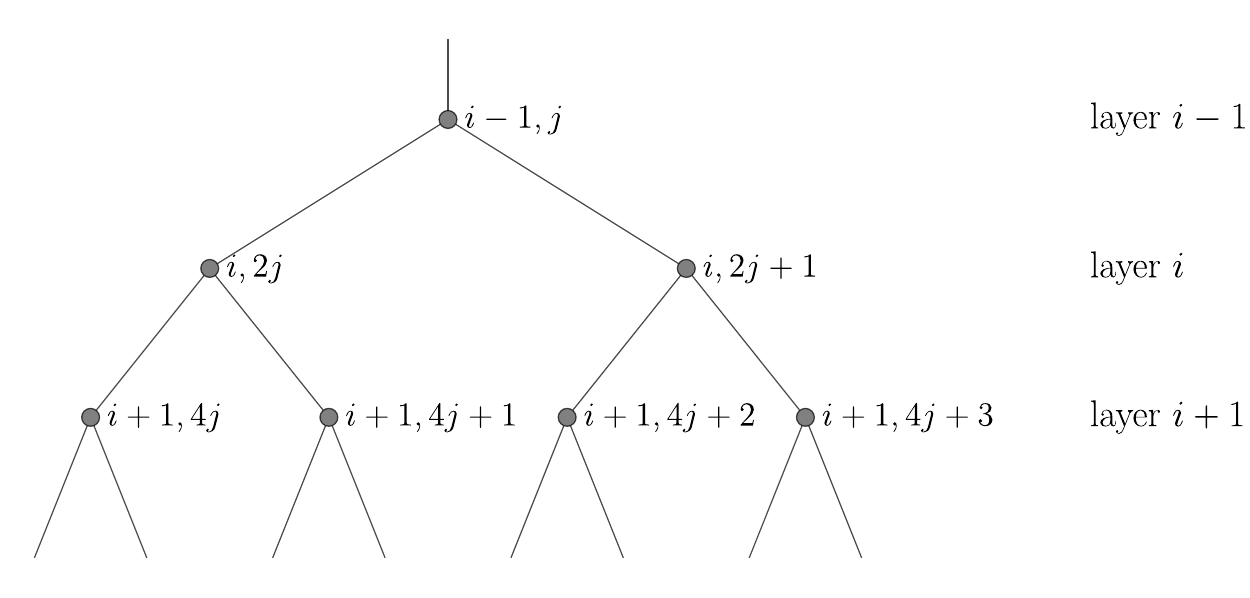}    
    \end{subfigure}
    
    \caption{The infinite binary tree  $\mathcal{G}=\mathcal{T}_2$ with a sketch of the associated labelling scheme.
    % Here we show an infinite, unrooted $k$-ary tree for $k=2$. The tree has infinitely many layers in both directions.
    }
    \label{fig:intro:tree}
\end{figure}
%
%An infinite $k$-ary tree is an 
Such trees are (undirected) graphs $\mathcal{T}_k=(V,E)$, $k\in\mathbb{N}$ in which the set of vertices is given by $V=\Z\times\N_0$ and the neighbourhood $\mathcal{N}(i,j)$ of each node $(i,j)$ consists of its parent node (in the $(i-1)$-th layer) and $k$ children (in the  $(i+1)$-th layer). 
We can explicitly characterize the set of edges $E$ as
\begin{equation*}
     \qquad E = \left\{\big((i,j), (i+1, kj + l)\big): \ i\in \Z, \ j\in \N_0, \ l\in \{0,\ldots, k-1\} \right\}.
\end{equation*}
Note in particular that $\mathcal{T}_1$ reduces to independent copies of $\Z$ with nearest-neighbour edges. 
% The standard one-dimensional lattice $\Z$ is isomorphic with the $k$-ary tree with $k=1$, i.e., $\mathcal{T}_1 \simeq \mathbb{Z}$. 

Let us now consider 
the bistable reaction-diffusion system
%\begin{equation}\label{eqn:intro:trees}
%    \dot{u}_{i,j}(t) = d \sum_{(i', j')\in \mathcal{N}(i,j)} (u_{i', j'}(t) - u_{i,j}(t))  + g(u_{i,j}(t);a),
%\end{equation}
\begin{equation}\label{eqn:intro:trees}
    \dot{u}_{i,j}(t) = d  [\mathcal{L}_k u(t)]_{i,j}  + g(u_{i,j}(t);a), \qquad (i,j)\in V
\end{equation}
posed on the tree $\mathcal{T}_k$,
in which the operator 
\begin{equation*}
    [\mathcal{L}_k u]_{i,j} = \sum_{(i', j')\in \mathcal{N}(i,j)} (u_{i', j'} - u_{i,j})
\end{equation*}
is commonly referred to as the graph Laplacian. 
We restrict our attention to so-called \textit{layer solutions},
%, i.e., solutions which are constant in each layer of $\mathcal{G}=\mathcal{T}_k$ so that the equality
which satisfy the equality
\begin{equation*}
    u_{i,j} (t) = u_i(t)
\end{equation*}
for all $(i,j)\in \Z\times\N_0$ and $t\in \R$. This substitution reduces the dynamics of  \eqref{eqn:intro:trees} to that of \eqref{eqn:intro:kLDE}.
In particular, the traveling fronts \eqref{eqn:intro:wave:ansatz:LDE} can be seen as layered invasion waves for the graph system \eqref{eqn:intro:trees}.
% the generalization of LDE~\eqref{eqn:intro:lattice-LDE}
% \begin{equation}\label{eqn:intro:kLDE}
%     \dot{u}_i(t) = d\left(ku_{i+1}(t) - (k+1)u_i(t) + u_{i-1}(t)\right) + g(u_i(t);a).
% \end{equation}
From this point of view it appears natural to take $k\in \N$ in \eqref{eqn:intro:kLDE},  but for our analysis it turns out to be worthwhile to also allow this parameter to be real. %consider real values  $k>1$, see \eqref{eqn:intro:kLDE}, or even $k>0$. 

In the following paragraphs we  motivate our approach and briefly summarize the existing literature from the point of view of our results.
\paragraph{Propagation through continuous media} 
%\todo[inline]{+[hjh: Give a bit more context first, i.e., taking $k=1$, writing $d = h^{-2}$
%and $x = i h $, our main equation formally reduces to... etc ]}
Taking $k=1$,  LDE~\eqref{eqn:intro:kLDE} can be considered as a spatially discrete approximation of the 
 classical bistable partial differential equation 
\begin{equation}\label{eqn:intro:PDE}
    u_t(x,t)= \nu u_{xx}(x,t)+g\big(u(x,t);a\big),\quad x\in\mathbb{R}, \quad t>0.
\end{equation}
Indeed, replacing the second derivative $u_{xx}$ with the central difference scheme results in the system
\begin{equation}
\label{eqn:intro:lattice-LDE}
%\label{eqn:intro:LDE:as:discretization}
    \dot{u}_i(t) = %\lim_{h\to 0} 
    d \big(u_{i+1}(t) - 2u_i(t) + u_{i-1}(t)\big) + g\big(u_i(t);a\big), \qquad i\in \Z,
\end{equation}
where $u_i(t) \sim u(ih, t)$ and $d h^2 = \nu$.
%$x_i = ih$. 
The bistable PDE \eqref{eqn:intro:PDE} %has first been studied in the
has been used to model %connection to
the spread of genetic traits \cite{Aronson1975nonlinear, fisher1937wave}, where it is often referred to as the \textit{heterozygote inferior}  case. It has  also been proposed as a basic model for the propagation of electrical signals through unmyelinated nerve fibres, also known as the `reduced' Fitzhugh-Nagumo equation  \cite{bell1981some, Nagumo1962active}. In general, \eqref{eqn:intro:PDE}  has played  a prototypical role 
during the development of the theory of traveling waves that
connect  two stable states of the underlying nonlinearity \cite{Fife1977}.

 Using  phase-plane analysis \cite{fife2013mathematical},  one can show that there exists a traveling wave solution 
\[
u(x,t) = \Phi(x- \sigma t), \qquad \Phi(-\infty) = 0, \qquad \Phi(+\infty) = 1
\]
of \eqref{eqn:intro:PDE} with
\[
\mathrm{sign} (\sigma )= - \mathrm{sign} \left( \int_0^1 g(u;a) \mathrm{d}u \right).
\]
This traveling wave satisfies the second order ODE
\begin{equation}
    \label{eq:int:trv:wave:ode}
    - \sigma \Phi'(\xi)  = \nu \Phi''(\xi) + g(\Phi(\xi);a).
\end{equation}
In the case of the cubic nonlinearity \eqref{eqn:intro:cubic}, there even exists an explicit solution formula for the speed $\sigma$, namely
\begin{equation}
\label{eq:int:trv:wave:speed:ode}
\sigma=\sqrt{2\nu}\left(a-\frac{1}{2} \right).
\end{equation}
From this equation it follows that  $\sigma=0$ if and only if $a=1/2$. The fact that we have  $\sigma=0$ only at one value of the bistable parameter $a$ is one of the fundamental differences between spatially continuous and discrete differential equations. 

\begin{figure}[t]
        \centering
        \begin{subfigure}{.32\textwidth}
                \centering
                \includegraphics[width=\textwidth]{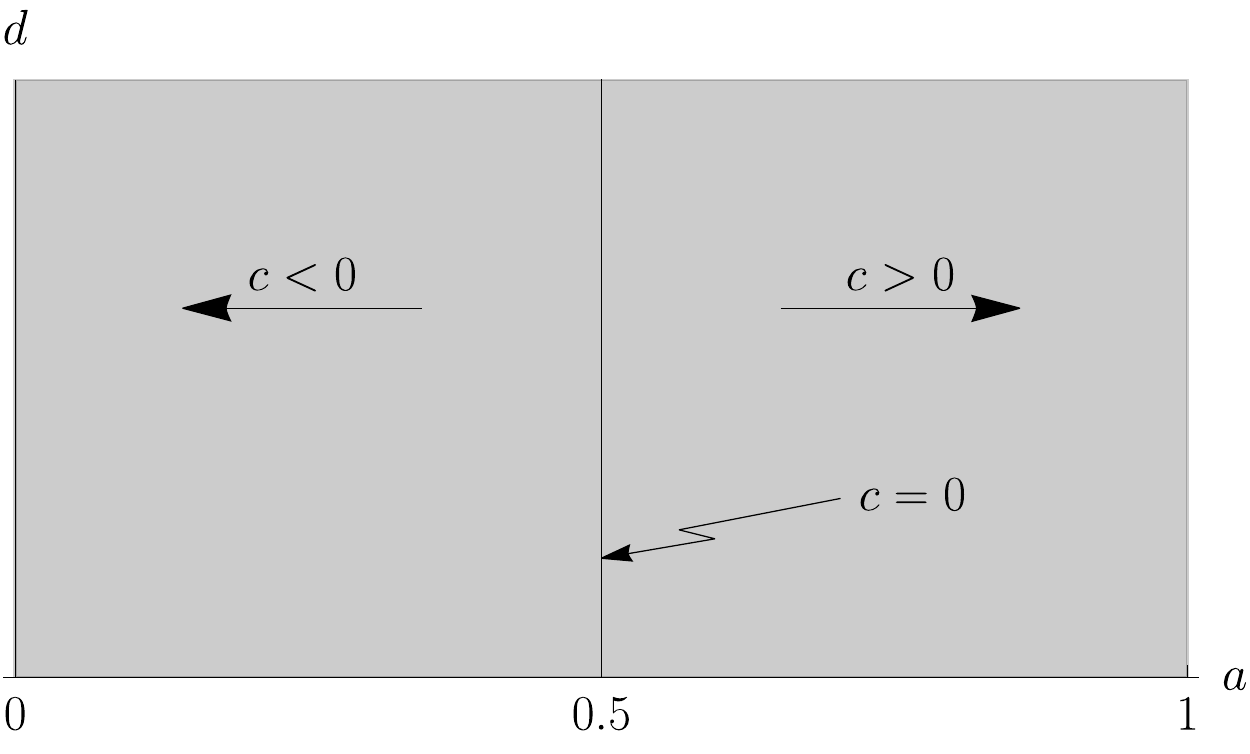}
                \caption{PDE~\eqref{eqn:intro:PDE}}
        \end{subfigure}
        \hfill
        \begin{subfigure}{.32\textwidth}
                \centering
                \includegraphics[width=\textwidth]{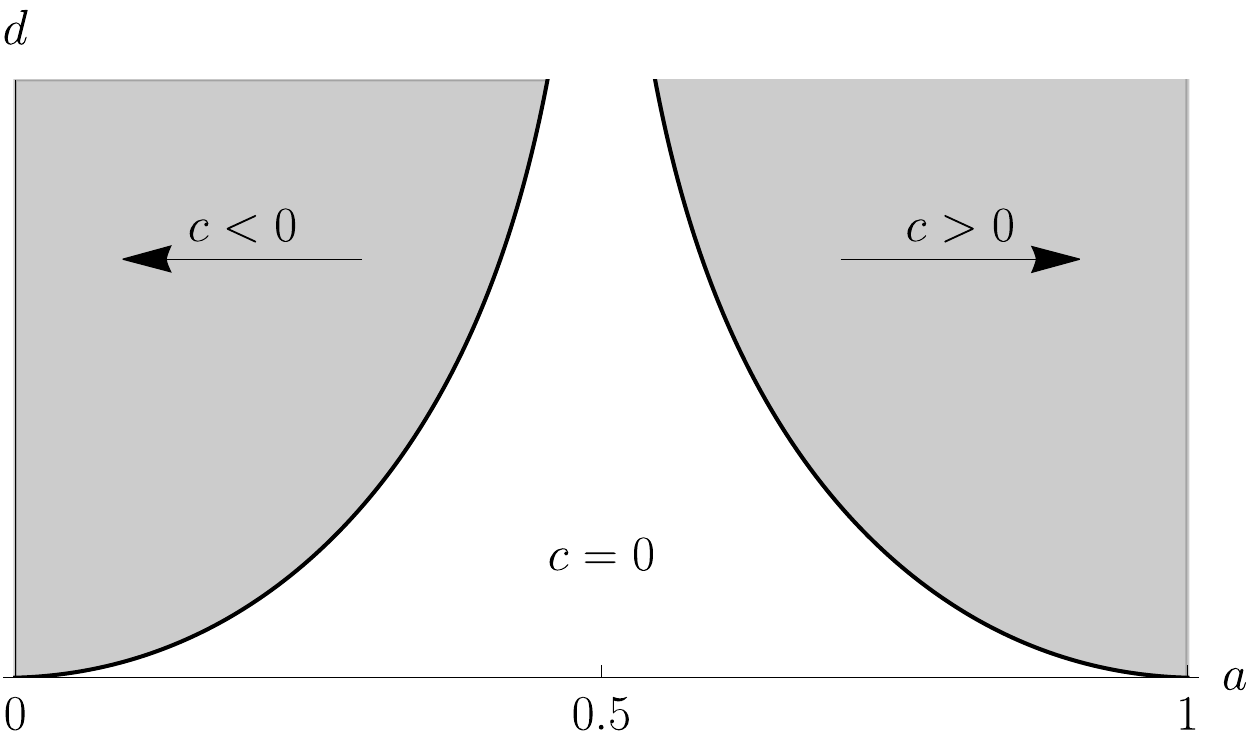}
                \caption{LDE~\eqref{eqn:intro:kLDE}, $k=1$}
        \end{subfigure}
        \hfill
        \begin{subfigure}{.32\textwidth}
                \centering
                \includegraphics[width=\textwidth]{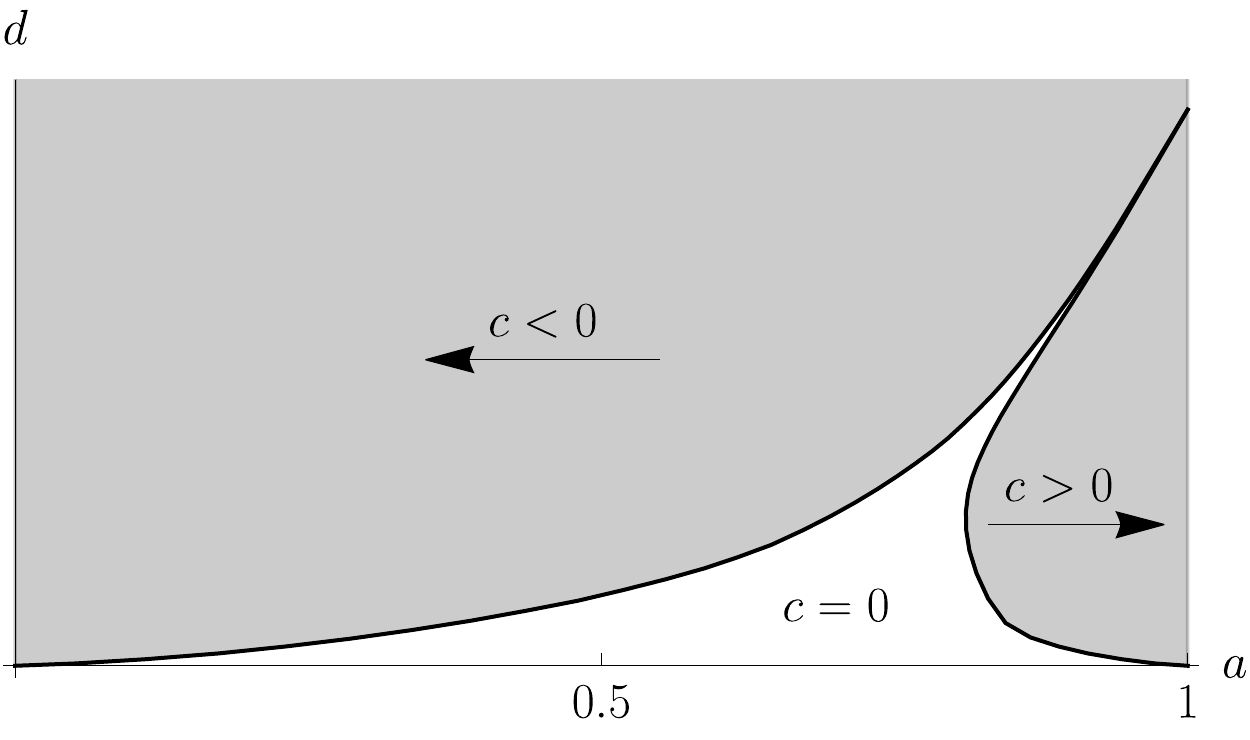}
                \caption{LDE~\eqref{eqn:intro:kLDE}, $k>1$}
        \end{subfigure}
        \caption{The dependence of the wave speed $c$ on the detuning parameter $a$ and the diffusion $d$. The left and middle panels show the well-known results for the PDE~\eqref{eqn:intro:PDE} and the LDE~\eqref{eqn:intro:kLDE} with $k=1$, i.e.,  \eqref{eqn:intro:lattice-LDE}. The right panel  summarizes our results for the LDE~\eqref{eqn:intro:kLDE} with $k>1$.}
    \label{fig:3_ad_regions}
\end{figure}

%\todo{The following paragraph was before called `Bistable lattice differential equation'}
\paragraph{Propagation through regular lattices}
Lattice differential equations are a natural modelling tool when the underlying spatial domain has a discrete structure. Crystals \cite{Cahn1960}, patchy landscapes \cite{Slavik2020, Stehlik2017} and myelinated neurons \cite{Ranvier1878} are all examples of such domains. One can find extensive lists of %one-dimensional 
models and application areas in \cite{hupkes2018traveling,keener1987propagation}.

Formally, equation~\eqref{eqn:intro:kLDE} is a generalization of the classic bistable LDE
\eqref{eqn:intro:lattice-LDE},
%\begin{equation}\label{eqn:intro:lattice-LDE}
%    \dot{u}_i(t) = d\big(u_{i+1}(t) - 2u_i(t) + u_{i-1}(t)\big) + g(u_i(t);a), \qquad i\in \Z,
%\end{equation}
%The considerable interest into equation~\eqref{eqn:intro:lattice-LDE} arises from applications  with underlying discrete spatial structure, 
%Solutions to this equation exhibit very distinct dynamical behaviour from its continuous counterpart $u_t=d u_{xx}+g(u;a)$ which is itself important for the study of the competition between two stable states (typically $u=0$ and $u=1$) and their heteroclinic connections. 
%The lattice equation~\eqref{eqn:intro:lattice-LDE} has 
%which can be obtained from \eqref{eqn:intro:LDE:as:discretization} by taking 
%$h=1$
%$\nu = d h^2$ and 
which has served as a prototypical example to study key phenomena like pinning and topological chaos. Indeed, it has attracted numerous studies, starting with the threshold propagation results in \cite{bell1981some} and \cite{Bell1984}.
One of the first rigorous studies of propagation failure for~\eqref{eqn:intro:lattice-LDE} was conducted by
Keener in \cite{keener1987propagation},
who established that $c= 0$ can hold for a (non-trivial) interval of bistable parameters $a$.
%Namely, propagation failure of pinning occurs when the equality $c= 0$ holds for some nonempty range of bistable parameters $a$. 
This is in stark contrast to the continuous bistable equation, where a slight change of the detuning parameter $a$ suffices to cause standing waves to move.  
Keener in \cite{keener1987propagation} applied  the
Moser theorem \cite{moser2016stable} to
show that for each $a\in(0,1)$ and sufficiently small diffusion $0<d\ll 1$ one can construct infinitely many horseshoe maps, with each of them giving rise to a  stationary solution of \eqref{eqn:intro:lattice-LDE} with values in $[0,1]$. %which block the propagation of waves. 
%In fact, Keener's pinning region consists of two subregions, one with the infinitely many spatially chaotic waves obtained through the horseshoe maps, and the other, larger region in which pinned waves are necessarily monotonic.
In addition, he constructed a larger region in the $(a,d)$ plane where waves
are pinned. %propagation cannot occur. 
%On the other hand, he also established that pinned waves cannot occur for $a$ in the neighbourhood of zero once the diffusion parameter $d$ is sufficiently large.
On the other hand, he also established regions in the vicinity of $a=0$ and $a=1$
where fronts are guaranteed to propagate. %respectively, in which the pinned waves cannot occur. 

%Moreover, he constructed lower and upper solutions to show that, if a traveling wave solution exists, then it necessarily has $c\neq 0$ once $d \gg 0$ for $a$ in some interval around $0$ and $1$.  

A general theory for the existence of traveling-wave solutions to a broad class of LDEs that includes~\eqref{eqn:intro:kLDE} was developed by 
Mallet-Paret~\cite{Mallet-Paret1999_Fredholm, Mallet-Paret1999},
who performed a direct analysis of 
mixed functional difference equations (MFDEs) such as
\begin{equation}\label{eqn:intro:MFDE}
    -c\Phi'(\xi) = d\left( k \Phi(\xi+1) -(k+1) \Phi(\xi) + \Phi(\xi) \right) + g(\Phi(\xi);a),
\end{equation}
which arises by substituting  $u_i(t) = \Phi(i-ct)$ into
\eqref{eqn:intro:kLDE}.
%However, at the time the general existence result for the traveling-wave solutions of~\eqref{eqn:intro:kLDE} was not yet discovered. To find such solutions, one first starts with an Ansatz 
%$    u_i(t) = \Phi(i-ct) $
%which results in the following 
%This MFDE falls under the general framework developed by Mallet-Paret~\cite{Mallet-Paret1999_Fredholm, Mallet-Paret1999} which
His results guarantee that for each $k>0$, $a\in(0,1)$ and $d>0$ one can find a speed $c\in \R$ and a profile $\Phi:\R\to \R$ that satisfy~\eqref{eqn:intro:MFDE}. %This result implies that one can remove the existence assumption in~\cite{keener1987propagation}. 
It should be remarked that the first existence result for $k=1$ was obtained by Zinner in~\cite{Zinner1992} in the regime $d\gg 0$.

\paragraph{Propagation through graphs} Dynamical systems on graphs serve naturally as a generalization of lattice equation where the interplay between finer graph properties and the dynamics can be investigated \cite{Selley2015, Slavik2020, Stehlik2017}. Trees  represent an important class of graphs as they model processes on discrete media with regular branching  structures \cite{arenas2001communication}. Our paper is closely connected to a recent study by Kouvaris, Kori and Mikhailov \cite{Kouvaris2012}, where approximation techniques are used to study
propagation and pinning phenomena of waves on arbitrarily large, but finite $k$-ary trees. The bi-infinite trees that we consider in this paper (see Fig.~\ref{fig:intro:tree}) do not have a root vertex, in order to avoid the technical difficulties
caused by adding boundaries to our spatial domain. 
%that would arise with adding boundaries to the finite spatial domain, in this paper we consider  trees with infinite amount of layers without the root vertex. 
However, due to the exponential convergence in the tails, we fully expect the traveling waves considered here to play an important organizing role 
for the dynamics on large but finite $k$-ary trees.
%in the
%dynamicasince due to the fact since the traveling waves are local structures in view of their exponentially converging spatial limits $\Phi(-\infty) = 0$ and $\Phi(+\infty) = 1$,  we expect that the dynamics of the traveling waves on large finite $k$-ary trees can be well approximated by the dynamics on $\mathcal{T}_k$.   

We expect that our results could also be relevant for more general graphs. 
For example, let us consider the Erd\"os-R\'{e}nyi random graph $\text{ER}_n(p)$ with $n$ nodes,
where the probability of two nodes being connected is given by $p$ \cite{erd1959and}. 
%If we assume that the number of nodes is large, and that we are 
In the sparse regime where $p=k/n$ for some fixed $k>0$, one can show \cite{Dorogovtsev2014, van2016random}  that  the Erd\"os-R\'{e}nyi random graph $\text{ER}_n(k/n)$ converges locally  in probability as $n\to\infty$ to a Poisson branching process with mean offspring $k$.
We can therefore consider  $k$-ary trees as local approximations of large  Erd\"os-R\'{e}nyi random graphs. Consequently, wave propagation and pinning in random networks can be directly linked to the related 
%directly to the spreading and pinning 
phenomena on trees \cite{Kori2006}. 

Similar ideas were explored very recently in \cite{hoffman2016invasion} for the monostable Fisher-KPP equation on semi-infinite $k$-trees with one root. In this study, the authors consider initial conditions that are zero everywhere except at the root vertex and establish the existence
of a critical diffusion parameter that separates (linear) spreading 
through the tree from extinction. %examine under which diffusion parameters the traveling wave solution spreads down the tree or propagates up the tree to the vertex. They prove the existence of a critical diffusion parameter $d_k$ such that $0<d<d_k$ implies that the solution propagates downwards throughout the tree with positive speed, whereas for $d>d_k$ the solution converges uniformly to the zero state. 
Moreover, their numerical simulations suggest that this conclusion can be transferred in some sense to the dynamics of Erd\"os-R\'{e}nyi random graphs.
%can be well approximated by
%that  on the semi-infinite trees. %Other related studies on graphs, networks and their applications in ecology include \cite{Stehlik2017},  \cite{Selley2015} and \cite{Slavik2020}.

 \paragraph{Comparison principle} 
Turning back to the original equation~\eqref{eqn:intro:kLDE},
%when $k\neq 1$, 
we note
that our main propagation results rely on the construction of appropriate sub- and super-solutions that push  traveling waves  to the left $(c<0)$ or right $(c>0)$, see Fig.~\ref{fig:intro:sub_sol}. We use two different constructions, which yield qualitatively different conclusions. 

%In 
Our first approach %we follow 
follows the outline from Keener~\cite{keener1987propagation} %and construct
to construct smooth but `step-like'  subsolutions. The simple nature of these functions results in %the straightforward expression
a relatively tractable expression for the sub-solution residual, %$\mathcal{J}^-$, 
which we examine thoroughly in \S\ref{sec:small_d}. 
Via this method we obtain
a geometric description for a set $\mathcal{D}^-$ in the $(a,d)$-plane
where the wave speed is guaranteed to be negative.
For the cubic nonlinearity, we are able to explicitly compute the boundary of $\mathcal{D}^-$, thus generalizing and completing the results from \cite{keener1987propagation}.
%
%obtain the set $\mathcal{D}^-$ in which the points $(a,d)$ have  a negative speed.   

This approach has both advantages and disadvantages.  On the one hand, the set $\mathcal{D}^-$  obtained through this method is a priori bounded in $d$,
unlike the actual region where $c < 0$.
On the other hand, this method enables us to exploit a
useful symmetry in the system that allows us to invert the sign of
the wave speed. %allow us to obtain in parallel results for $c<0$ and $c>0$.
In particular, we also obtain a region $\mathcal{D}^+$ close to $a\approx 1$ where the wave speed is guaranteed to be strictly positive.
Moreover, our numerical observations indicate that %for small $d$,  
the lower boundaries of $\mathcal{D}^-$ and $\mathcal{D}^+$
are closely aligned with the edge of the pinning region. This result can be intuitively explained by the fact that traveling profiles close to the pinning regime are themselves almost step-like; see the left panel of Fig.~\ref{fig:intro:sub_sol}. The steep sub-solutions therefore provide a good approximation of the actual wave-profiles. 
%The additional strength of this method is that it has a straightforward geometric interpretation, which we describe in \S\ref{sec:main}. This geometric set-up allows us to explicitly describe the boundaries of $\mathcal{D}^-$ and $\mathcal{D}^+$,  thus generalizing and completing the results from \cite{keener1987propagation}. 

Our second method relies on a more refined construction of sub-solutions.
In particular, we build smooth and wide profiles that agree better with the actual wave-profile $\Phi$ in the $d\gg 0$ regime, see the right panel of Fig.~\ref{fig:intro:sub_sol}. 
For every $a \in(0,1)$ we provide a value $d^*(a)$
so that 
%that there exists $d^*$ such that 
$d>d^*$ implies $c<0$, %for every $a\in (0,1)$.
which %This result 
shows that waves have a preferred $a$-independent direction of propagation.
%as we 
Together, these results allow us to paint a rather complete qualitative picture for general bistable nonlinearities.

%\todo{On the right should be a steep step-like subsolution}
\begin{figure}
\centering
%\begin{subfigure}{0.45\textwidth}
%         \centering
%         \includegraphics[width=\textwidth]{fig/Vladimir_subsolution.pdf}
%        \end{subfigure}
%\begin{subfigure}{0.45\textwidth}
%         \centering
%         \includegraphics[width=\textwidth]{fig/Mia_subsolution.pdf}
%       \end{subfigure}
\includegraphics[width=\textwidth]{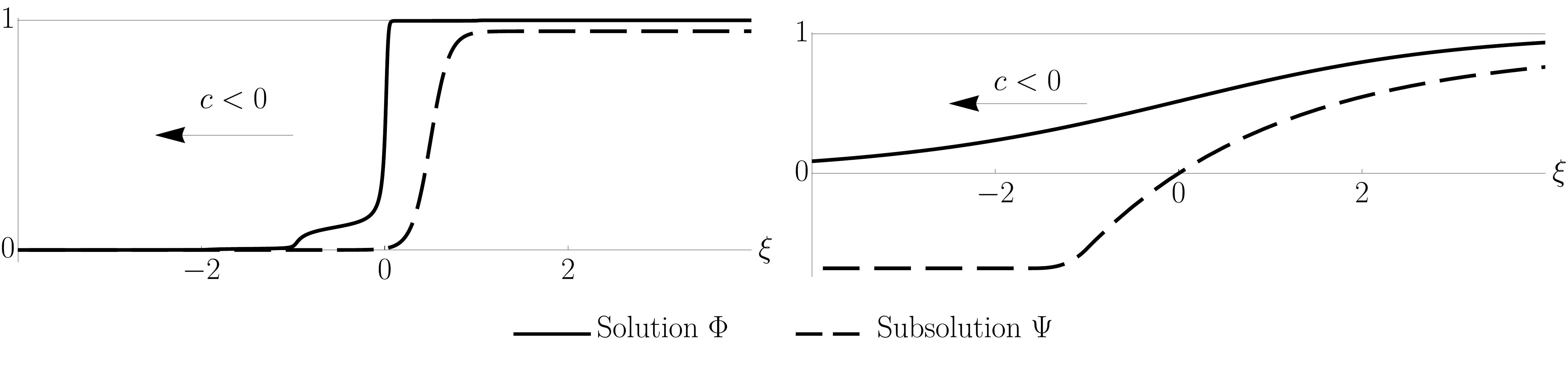}
\vspace{-.75cm} % VS:an ugly workaround; I do not know how to fix it right now
\caption{
% These panels display numerically computed traveling waves for $d =0.00205$, $a=0.1933$ and $k=5$ (left) and $d=0.4$, $a=0.1933$ and $k=5$ (right). The travelling speeds of each wave are $c\approx-0.0081$ and $c = -2.06$, respectively. In each of these two regimes we use a different technique 
% to construct sub-solutions that rigorously predict the sign of the wavespeed. For $d\gg 0$, 
% we construct wide subsolutions $\Psi$ (see $\S$\ref{sec:d_large}), whereas for $d$ small, we use smooth, almost step-like subsolutions (see $\S$\ref{sec:small_d}).
% \note{VS:check caption, $a,d,k$ values should be the correct ones}} 
Subsolutions $\Psi$ and numerically computed travelling waves $\Phi$ with $k=5$ and $a=0.1933$ illustrating our two distinct approaches to construct subsolutions and predict the wavespeed sign. First, for small $d$ we construct steep, almost step-like, subsolutions, see $\S$\ref{sec:small_d}. The left panel shows an example with $d=0.00205$ and $c\approx-0.0081$. Next, for $d\gg 0$, we construct wide subsolutions $\Psi$, see $\S$\ref{sec:d_large}. The right panel illustrates this approach with $d=0.4$ and $c \approx -2.06$.
}\label{fig:intro:sub_sol}
\end{figure}

%\paragraph{Propagation reversal} via wave collisions \cite{Hupkes2019mchrom}. Do we know about
%our reversal is via change in diffusion
%\todo[color=yellow]{Relevant literature? Diffusion-driven instabilities?}

\paragraph{Propagation reversal} 
%The major consequence of our results is the fact that for a fixed $a\approx 1$ and $k>1$ the propagation of the wave solutions to \eqref{eqn:intro:kLDE} can be reversed by the change in the diffusion parameter $d$. If $d>0$ is sufficiently small, the wave is pinned with propagation speed $c=0$. As we increase $d$ the wave moves to the right ($c>0)$, then it is pinned again ($c=0$) and once it crosses a threshold $d^*(a,k)$ it propagates to the left ($c<0$). See Ex.~\ref{ex:propagation:reversal} and Fig.~\ref{fig:ex:propagation:reversal} for an illustration.

To gain some intuition for the diffusion-driven propagation reversal that occurs
%when increasing $d > 0$ 
for $a \approx 1$ and $k > 1$,
%is result,
let us substitute
$\Phi(\xi) = \psi(\xi h)$ 
into the travelling wave MFDE
\eqref{eqn:intro:MFDE} to obtain
\begin{equation}
    -c h \psi'(\xi) - (k-1) d [\psi(\xi + h) - \psi(\xi)] =
    d  [ \psi(\xi -h) + \psi(\xi + h) - 2 \psi(\xi)]
    +  g(\psi(\xi);a) .
\end{equation}
Taylor expanding around $\psi(\xi)$ and sending $h \to 0$ while keeping the quantities
\begin{equation}
    \nu = \frac{1}{2} (k+1) d h^2, \qquad \qquad \sigma = ch + (k-1) d h
\end{equation}
fixed, this MFDE formally reduces to the
travelling wave ODE \eqref{eq:int:trv:wave:ode}.
In particular, the wavespeed identity
\eqref{eq:int:trv:wave:speed:ode}
for the cubic nonlinearity \eqref{eqn:intro:cubic}
now leads to the asymptotic prediction
\begin{equation}
    c  \sim \sqrt{(k+1) d }\left(a - \frac{1}{2}\right) - (k-1) d
\end{equation}
for $d \gg 1$. For $k > 1$ and $a > \frac{1}{2}$ the
right-hand side changes sign at the critical value
\begin{equation}
    d_c(a,k) =  \frac{(k+1)}{(k-1)^2}\left(a - \frac{1}{2}\right)^2
\end{equation}
which we expect to be increasingly accurate in the limit $k \downarrow 1$.

%We believe that such a reversal mechanism of speed reversal has
We believe that such a reversal mechanism for the direction of propagation has not been observed in the literature. Related studies have focussed on other mechanisms such as the scattering or combination of waves. For example,
the numerical results in \cite{nishiura2003scattering} indicate that
the outcome of wave collisions for the PDE  \eqref{eqn:intro:PDE} depend on the properties of a class of unstable solutions called separators. 
The LDE \eqref{eqn:intro:lattice-LDE} admits a class of generalized non-monotone (multichromatic) waves that can reverse their direction
through intricate collision processes \cite{Hupkes2019mchrom}. 
We note that a general theory to fully describe such collisions has not yet been developed.

%Intricate collisions of nonmonotone (so called multichromatic) waves for the LDE \eqref{eqn:intro:lattice-LDE} were studied in \cite{Hupkes2019mchrom}.

%There are, however, related studies that describe propagation reversal via wave collisions for the PDE 
%\eqref{eqn:intro:PDE} and the LDE \eqref{eqn:intro:lattice-LDE}. Numerical results in \cite{nishiura2003scattering} indicate that mechanisms of wave collisions (annihilation, combination or scattering) for the PDE  \eqref{eqn:intro:PDE}  depend on the properties of a class of unstable solutions called separators. Intricate collisions of nonmonotone (so called multichromatic) waves for the LDE \eqref{eqn:intro:lattice-LDE} were studied in \cite{Hupkes2019mchrom}. However, there is no general theory describing fully counterpropagating wave interactions for LDEs.

%\todo[color=yellow, inline]{HJ, it would be nice if you can formulate concisely the limit process here and why the reversal is not then so surprising :-) HJH: Done!}

\paragraph{Organization}
This paper is organised  as follows. We set the stage and %establish a general set-up and
state our main results in {\S}\ref{sec:main}. This section also includes explicit expressions for the propagation and pinning regions for the cubic nonlinearity.
In \S\ref{sec:cp} we summarize several consequences of the comparison principle that we use throughout the paper. We study the pinning region 
%that consists of motononic waves
in \S\ref{sec:pinning} by establishing the existence of invariant intervals. 
In \S\ref{sec:small_d} we
 construct steep sub-solutions to establish the existence of the region $\mathcal{D}^-$ in which the wave speed is negative. Exploiting a symmetry argument allows us to establish the equivalent results for positive speeds in the region $\mathcal{D}^+$. 
These two sections adapt the ideas from \cite{keener1987propagation}  to the more general setting considered in this work. 

We proceed in
{\S}\ref{sec:d_large} with the construction of wide
sub-solutions that work well in the $d \gg 0$ regime.
%yet another type of subsolutions, this time wide profiles that fit better the traveling wave solutions for $d\gg 0$. 
Using the comparison principle we show that $c<0$ for all $d\gg 0$.   Section \S\ref{sec:cubic_proof} is dedicated to the cubic nonlinearity, as we
provide explicit expressions for the boundaries of the sets $\mathcal{D}^-$ and $\mathcal{D}^+$. 
In \S\ref{sec:chaos} we describe chaotic steady solutions to our initial LDE~\eqref{eqn:intro:kLDE} by adapting the set-up from~\cite{keener1987propagation} and~\cite{moser2016stable}. 
We conclude the paper with numerical examples
to illustrate  the reversal of propagation on $k$-ary trees in \S\ref{sec:numerical:examples}.

% \paragraph{Modelling background}
% m: Here we write about the chemist paper Traveling and Pinned Fronts in Bistable ReactionDiffusion Systems on Networks. Also write how they notice that the wave 'spreads' or 'retreats' depending on the size of parameter $d$, and how for $d$ small they encounter pinning. 

% \paragraph{traveling waves on a one-dimensional domain}

% Our work is greatly inspired by a seminal paper~\cite{keener1987propagation} in which the author studies the equivalent of our starting equation~\eqref{eqn:intro:kLDE} in case $k=1$, i.e.,
% \begin{equation}\label{eqn:intro:LDE:k1}
%     \dot{u}_i = d(u_{i+1} - 2u_i + u_{i-1}) + g(u_i;a), i\in \Z.
% \end{equation}
% In that paper, Keener proves that a monotonic stationary solution to~\eqref{eqn:intro:kLDE} exists when the parameter $d$ is small enough. Using the comparison principle he also shows that there exists a parameter-regime in the $a-d$ plane such that solutions to the initial value problem~\eqref{eqn:intro:kLDE} with initial condition $u^0\in \ell^\infty(\Z)$ such that $\liminf_{i\to\infty} u^0_i = 1$ have $c\neq 0$. 
% The first results regarding the existence and stability of the front solutions to~\eqref{eqn:intro:LDE:k1} are shown by Zinner in~\cite{Zinner1991stability, Zinner1992}. 

% \paragraph{Asymmetric wave-speed}
% m:Here write about our results, for the speed in $a-d$ plane is not symmetric anymore - relate them back to chemist paper. 

% \paragraph{Comparison principle}
% Here briefly explain how we constructed super-sub solutions. 

% \paragraph{Outline}

\section{Main results}\label{sec:main}

The main focus of our study is the reaction-diffusion-advection equation
\begin{equation}\label{eqn:main:LDE}
    \dot{u}_i(t) = d [\Delta_k u(t)]_i + g(u_i(t);a),
\end{equation}
posed on the one dimensional lattice $i \in \Z$.  The discrete diffusion-advection operator $\Delta_k :\ell^\infty(\Z)\to \ell^\infty(\Z)$ is defined by
\begin{equation}\label{eqn:main:k_laplacian}
  \Delta_k u  = u_{i-1} - (k+1)u_i+ k u_{i+1}.
\end{equation}
We require the nonlinearity $g$ to satisfy the following standard bistability assumption. 
\begin{itemize}
    \item[\Hg]
    %For each $a\in (0,1)$ the nonlinearity $g=g(\cdot;a)$ belongs to   $C^1(\R)$  and we have
    The map $(u,a)\mapsto g(u;a)$ is $C^1$-smooth  on $\R\times[0,1]$ and we have
    \begin{align*}
        g(0;a) &= g(a;a) = g(1;a) = 0, \\[0.3cm]
        g'(0;a)  &< 0, \qquad
        g'(1;a) < 0, \qquad
         g'(a;a) > 0.  
    \end{align*}
    In addition, the function $g$ satisfies the inequalities
    \begin{equation*}
       g(v;a)>0 \ \text{for}\ v\in (-\infty,0) \cup (a,1), \quad g(v;a)<0 \ \text{for}\ v\in (0, a) \cup (1, \infty).
    \end{equation*}
   % Moreover, $g$ also depends $C^1$-smoothly on the parameter $a$. 
\end{itemize}
Throughout this paper we 
%denote by $'$ the derivative of $g$ with respect to $v$, i.e., 
write $g'(v;a) = \partial_v g(v;a)$. 
At times, we also need to impose the following additional assumptions on $g$.
\begin{itemize}
\item[\Hgone]
  For each $a\in (0,1)$ and $v\in (0,1)$ we have
    $\partial_a g(v;a) < 0$. 
\item[\Hgtwo]
For each $a\in [0,1]$, the nonlinearity $g(\cdot;a)$ belongs to $C^2(\R)$ and we have $g'(a;a) =  0$  for $a\in \{0,1\}$, $g''(0;0) > 0 $ and  $g''(1;1)<0$. Moreover, there exist $a_0$ and $a_1$ in $(0,1)$ such that for each $a\in (0, a_0)$ and $a\in (a_1, 1)$  there exists a unique $v = v(a)$ for which $g''(v;a) = 0$. 
\end{itemize}
 Both \Hgone and \Hgtwo are satisfied for the standard cubic nonlinearity~\eqref{eqn:intro:cubic}, on account of the identities
 \begin{equation*}
     \partial_a g(v;a) = -v(1-v) <0 , \quad  g'(a;a) = -a^2+ a ,
     \quad g''(a;a) = -2a+1
 \end{equation*}
 %as we have $$, and
 and the fact that the equality $g''(v;a) = 0$ holds if and only if $v = (a+1)/3$.

We are specially interested in so-called traveling wave solutions to \eqref{eqn:main:LDE}, which can be written in the form
\begin{equation}\label{eqn:wave_ansatz}
    u_i(t) = \Phi(i-ct),
\end{equation}
for some speed $c\in \R$ and profile $\Phi:\R\to\R$. Substituting~\eqref{eqn:wave_ansatz} into~\eqref{eqn:main:LDE} results in the MFDE
\begin{equation}\label{eqn:main:MFDE}
    -c\Phi'(\xi) = d\big(k\Phi(\xi+1) - (k+1)\Phi(\xi) + \Phi(\xi-1)\big) + g(\Phi(\xi);a).
\end{equation}
Throughout most of the paper we restrict ourselves to  heteroclinic connections that connect the two stable equilibria of the nonlinearity $g$. Therefore,  we also add the boundary conditions
\begin{equation}\label{eqn:main:bc}
    \lim_{\xi\to-\infty} \Phi(\xi) = 0, \qquad \lim_{\xi\to+\infty} \Phi(\xi) = 1.
\end{equation}

Equation~\eqref{eqn:main:MFDE} is a special case of the general problem considered in~\cite{Mallet-Paret1999}. We therefore start by
%In the following proposition we summarize
summarizing the key results from \cite{Mallet-Paret1999} that we use in our work. To simplify our notation, we write
\begin{equation*}
    \mathcal{H} = (0,1)\times (0, \infty)
\end{equation*}
for the set of parameters $(a,d)$ that we consider. 

\begin{proposition}\label{prop:main:MP}\cite[Thm. 2.1]{Mallet-Paret1999}
Suppose that \Hg holds and pick $(a,d )\in \mathcal{H}$ together with $k>0$.
 Then there exist a speed $c=c(a,d, k)$ and a non-decreasing profile $\Phi:\R\to \R$ that solve~\eqref{eqn:main:MFDE} with the boundary conditions~\eqref{eqn:main:bc}.  
 Moreover, $c(a,d, k)$ is uniquely determined and depends $C^1$-smoothly on all parameters when $c(a,d, k) \neq 0$. In this case the profile $\Phi$ is $C^1$-smooth with  $\Phi'>0$ and unique up to translations.
\end{proposition}

%In the standard case $k=1$, provided that $g$ is a cubic nonlinearity given by~\eqref{eqn:intro:cubic}, one can exploit the relation $g(1-v;a) = - g(v;1-a)$ to find that the wave-speed $c$ is, up to the sign change, symmetrical around the axis $a=1/2$ . In particular, for every $d>0$, we have 
%\begin{equation}\label{eqn:main:sym_c}
%c(a, d, 1) = -c(1-a, d, 1).    
%\end{equation}
%This allows the analysis in~\cite{keener1987propagation} to only consider one of the cases $c<0$ or $c>0$ and subsequently pass the results to the other case. In our following result, we provide a parameter-transformation that allows us 
%to generalize~\eqref{eqn:main:sym_c}  for arbitrary bistable nonlinearity.

In the traditional setting where $k=1$ and $g$ is given by the cubic nonlinearity~\eqref{eqn:intro:cubic}, one can exploit the identity $g(1-v;a) = - g(v;1-a)$ to obtain the symmetry relation
%that the wave-speed $c$ is, up to the sign change, symmetrical around the axis $a=1/2$ . In particular, for every $d>0$, we have 
\begin{equation}\label{eqn:main:sym_c}
c(a, d, 1) = -c(1-a, d, 1).    
\end{equation}
This allows the analysis in~\cite{keener1987propagation} to only consider one of the cases $c<0$ or $c>0$ and subsequently transfer the results to the other case. 

The result below provides a useful generalization of this symmetry relation, which will help to interpret and formulate some of our results. As a preparation, 
%
%we pick $k>0$ together with $(a,d)\in \mathcal{H}$
we introduce the transformed parameters
\begin{equation}\label{eqn:main:symmetry}
    \tilde{k} = \frac{1}{k}, \qquad \tilde{a} = 1-a, \qquad \tilde{d} = dk,
\end{equation}
together with the nonlinearity %$\tilde{g}\in C^1(\R\times[0,1]; \R)$, defined by 
\begin{equation}\label{eqn:main:g_tilde}
\tilde{g}(v;\tilde{a}) = -g(1-v;a). 
\end{equation}
Since the function $\tilde{g}$ also satisfies \Hg, Proposition
\ref{prop:main:MP} yields the existence of a transformed speed function
$\tilde{c}$ associated to the solutions of
\eqref{eqn:main:MFDE}-\eqref{eqn:main:bc} with 
$(\tilde{k}, \tilde{a}, \tilde{d}, \tilde{g})$ instead of
$(k, a, d, g)$.
%by the parameters $\tilde{k}>0$ and $(\tilde{a}, \tilde{d})\in \mathcal{H}$.
%which allows us to consider solutions $(\tilde{c}, \tilde{\Phi})$ of \eqref{eqn:main:MFDE}-\eqref{eqn:main:bc} with parameters $\tilde{k}>0$ and $(\tilde{a}, \tilde{d})\in \mathcal{H}$.

%In light of this remark, one can consider \eqref{eqn:main:symmetry_cphi} as a generalization of~\eqref{eqn:main:sym_c}  that holds for arbitrary bistable nonlinearities and arbitrary branch factors.

%For $k>0$ and $(a,d)\in \mathcal{H}$, we introduce the parameters $\tilde{k}>0$ and $(\tilde{a}, \tilde{d})\in \mathcal{H}$ by

%together with the nonlinearity $\tilde{g}\in C^1(\R\times[0,1]; \R)$, defined by 
%\begin{equation}\label{eqn:main:g_tilde}
%\tilde{g}(v;\tilde{a}) = -g(1-v;a). 
%\end{equation}
%The function $\tilde{g}$ also satisfies \Hg, which allows us to consider solutions $(\tilde{c}, \tilde{\Phi})$ of \eqref{eqn:main:MFDE}-\eqref{eqn:main:bc} with parameters $\tilde{k}>0$ and $(\tilde{a}, \tilde{d})\in \mathcal{H}$.

\begin{lemma}\label{prop:main:symmetry}
Suppose that \Hg holds and pick $(a,d)\in \mathcal{H}$ 
together with $k>0$. % Let pair $(c, \Phi)$ be a solution of the MFDE~\eqref{eqn:main:MFDE}-\eqref{eqn:main:bc}. Moreover, we denote by $(\tilde{c}, \tilde{\Phi})$ a solution of  the MFDE~\eqref{eqn:main:MFDE}-\eqref{eqn:main:bc} with parameters $(\tilde{a}, \tilde{d}, \tilde{k})$ and the nonlinearity $\tilde{g}$ in place of $g$. 
Then we have
\begin{equation}\label{eqn:main:symmetry_cphi}
   %\tilde{\Phi}(\xi)  = 1-\Phi(-\xi), \qquad 
   \tilde{c}(\tilde{a}, \tilde{d}, \tilde{k}) = -c(a, d, k). 
\end{equation}
\end{lemma}
\begin{proof}
 Let $(c, \Phi)$ be a solution of the MFDE~\eqref{eqn:main:MFDE}-\eqref{eqn:main:bc}
 and write %$\big(\tilde{c}, \tilde{\Phi}(\xi) \big)  = \big(-c, 1-\Phi(-\xi)\big)$
 $\tilde{\Phi}(\xi)   = 1-\Phi(-\xi)$
 together with $\tilde{c} = -c$.
  A straightforward computation % starting from \eqref{eqn:main:MFDE}
 shows that the pair $(\tilde{c}, \tilde{\Phi})$
 also satisfies  \eqref{eqn:main:MFDE}-\eqref{eqn:main:bc}, 
 but now  with the transformed parameters \eqref{eqn:main:symmetry}
 and nonlinearity \eqref{eqn:main:g_tilde}.
 % $\tilde{c} = \tilde{c}(\tildeOne can then prove \eqref{eqn:main:symmetry_cphi} by   and noting that the pair $(\tilde{c}, \tilde{\Phi})$, where $\tilde{\Phi}(\xi)  = 1-\Phi(-\xi)$, solves \eqref{eqn:main:MFDE}-\eqref{eqn:main:bc} with parameters $(\tilde{a}, \tilde{d}, \tilde{k})$ and the nonlinearity $\tilde{g}$.
\end{proof}

\subsection{Pinned waves}
In our following result we show that for any bistable nonlinearity there exists a nonempty region in the $(a,d)$ plane where waves are pinned, i.e., $c = 0$, see Fig.~\ref{3:fig:main:numerics+theory:k2}. We achieve this 
by showing that there exist two regions with nonempty intersection, one with  $c\leq 0$ and the other with $c\geq 0$. 
To this end, %Before we state this result, for fixed bistable nonlinearity $g$ we 
we define two curves  $d^-:(0,1)\to(0, \infty)$ and $d^+:(0,1)\times (0, \infty)\to (0, \infty)$  by writing
\begin{align}
   d^-(a) := \max_{y\in(a,1)} \dfrac{g(y;a)}{y}, \qquad 
   d^+(a, k) := \max_{y\in (1-a,1)} \dfrac{-g(1-y;a)}{k y} .\label{eqn:pinning:d-}% \label{eqn:pinning:d+}
\end{align}
We note that the $k$-dependence of $d^+$ is directly related
to the transformation \eqref{eqn:main:symmetry}. Finally, we define the analytical pinning region by
\begin{equation}\label{eqn:main:Dzero}
   \mathcal{D}^0(k) % = \mathcal{D}^0_g(k) 
   : = 
   \left\{ (a,d)\in \mathcal{H}:\    0 < d < \min \{d^-(a), d^+(a, k)\}  \right\}.
\end{equation}
 \begin{proposition}\label{prop:monotonic:existence} Assume that \Hg holds and pick $a \in (0,1)$ together with $k>0$. % and let $(\Phi, c)$ be a solution of \eqref{eqn:main:MFDE}-\eqref{eqn:main:bc}. 
Then the following claims hold true.
\begin{enumerate}[(i)]
    \item \label{item:main:monotnic:c_greateq0} For any $d\in \big(0, d^+(a,k)\big)$ we have $c(a,d, k)\geq 0$.
    \item \label{item:main:monotnic:c_lesseq0} For any  $d\in \big(0, d^-(a)\big)$ we have $c(a,d, k)\leq 0$.
\end{enumerate}
%In particular, for any $0 < d < \min \{d^-(a), d^+(a, k)\}$  we have $c(a,d,k) = 0$.
In particular, for any $(a,d) \in \mathcal{D}^0(k)$  we have $c(a,d,k) = 0$.
\end{proposition}
 
 As a follow-up result, we provide more detailed insight into the pinning region. We show that for all sufficiently small $d>0$ one can construct infinitely many bounded solutions to~\eqref{eqn:main:MFDE} with $c=0$. This system
 is said to admit `spatial chaos' due to the fact that these solutions can be constructed from arbitrary sequences in $\{0,1\}$.  

 \begin{proposition}\label{prop:prop_failure:steady_sols:existence}
Assume that \Hg holds and pick $a\in (0,1)$ together with $k>0$. 
Then there exists $d_0= d_0(a, k) > 0$ such that for every $0<d\leq d_0$  and every sequence $(s_i)_{i\in \Z}\subset \{0,1\}$, there is at least one solution of~\eqref{eqn:main:MFDE} that has $c = 0$ together with
\begin{equation*}
    \begin{aligned}
    \Phi(i)\in [0,a), & & \  & \mathrm{if}\ s_i = 0,\\ 
    \Phi(i) \in (a,1], & & \  & \mathrm{if}\ s_i = 1. 
    \end{aligned}
\end{equation*}
\end{proposition}

\subsection{Propagating waves}

The two main results below establish criteria that guarantee
the propagation of waves, i.e., $c\neq 0$. The first one provides
a quantitative lower bound $d^*$ above which the wave 
speed is strictly negative. This lower bound is defined
for all $k > 1$ and $a \in (0,1)$, in clear contrast to the
symmetry \eqref{eqn:main:sym_c} that occurs for the 
cubic nonlinearity with $k=1$. We note that general conditions
that guarantee $c > 0$ for $k=1$ and $a \sim 1$ can be found
in \cite[Thm. 2.6]{Mallet-Paret1999}.

\begin{figure}[t!]
\centering
%\begin{subfigure}{0.45\textwidth}
%         \centering
%         \includegraphics[width=\textwidth]{fig/k_4_theoretical2.pdf}
%        \caption{Only theoretical findings}
%     \end{subfigure}
%\begin{subfigure}{0.45\textwidth}
%         \centering
%         \includegraphics[width=\textwidth]{fig/k_4_numerical2.pdf}
%      \caption{Numerical and theoretical findings}
%    \end{subfigure}
%     \hfill
\begin{subfigure}{0.49\textwidth}
        \centering
        \includegraphics[width=\textwidth]{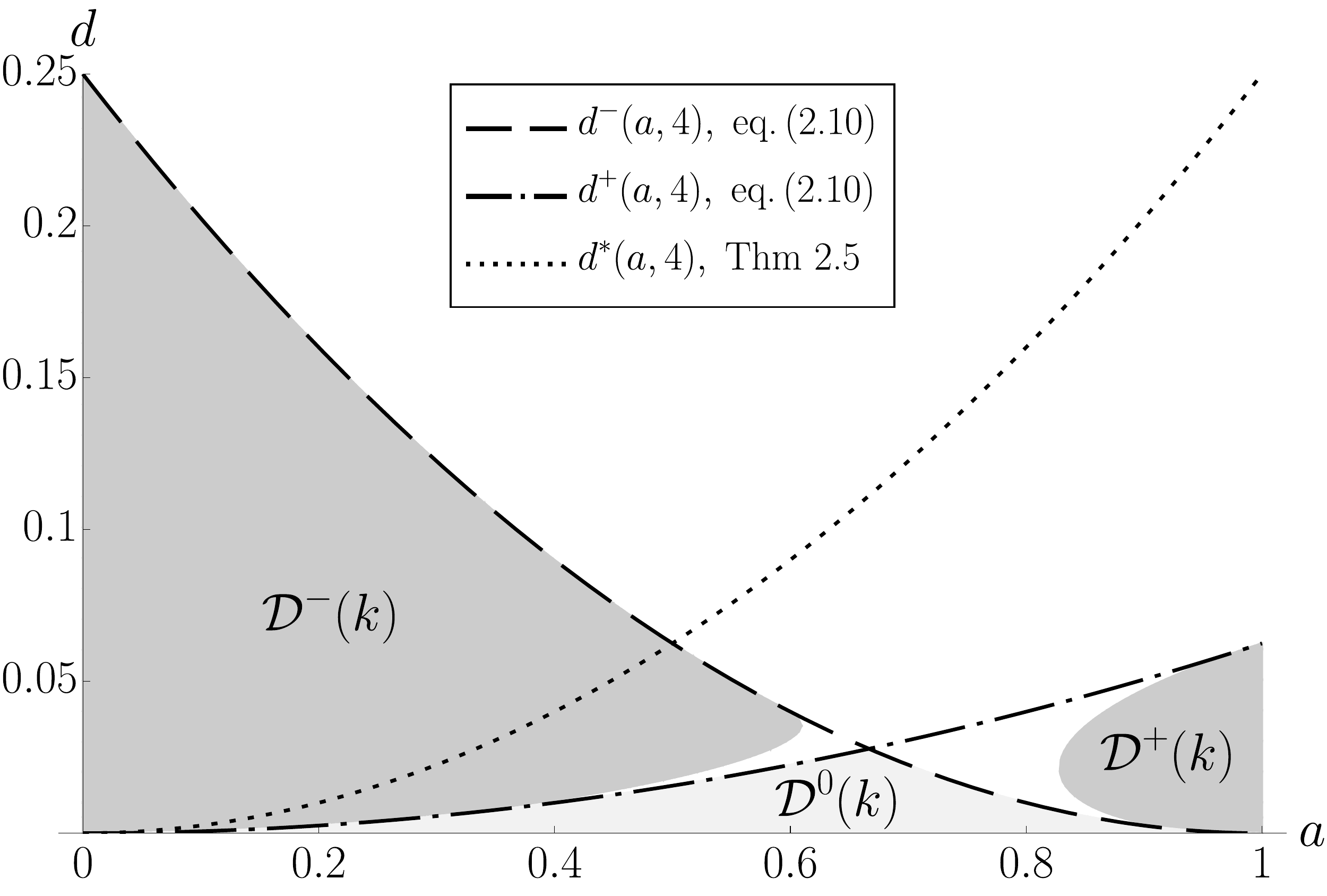}
        %\caption{Only theoretical findings}
    \end{subfigure}
    \hfill
\begin{subfigure}{0.49\textwidth}
        \centering
        \includegraphics[width=\textwidth]{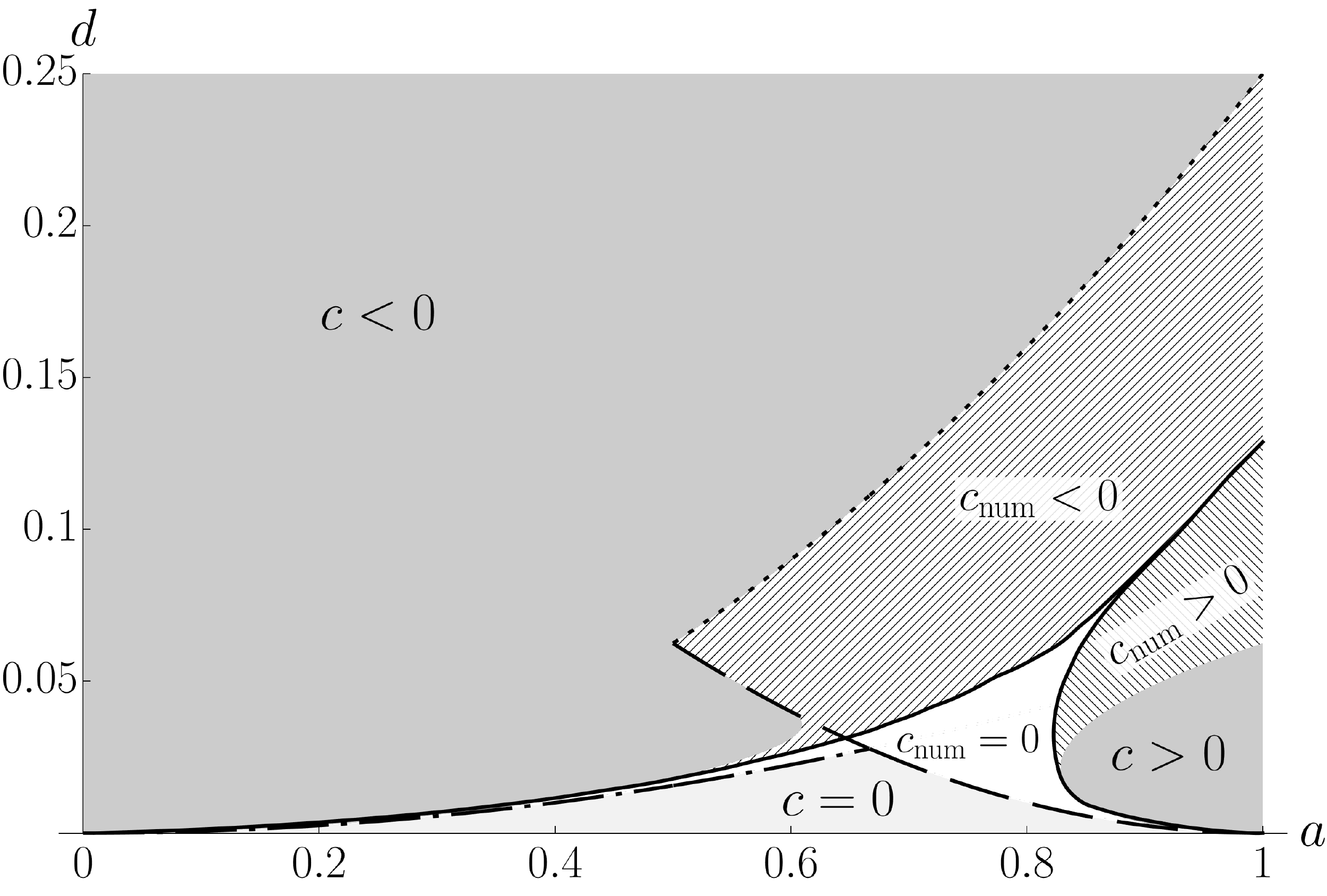}
     %\caption{Numerical and theoretical findings}
    \end{subfigure}

\caption{
%  The left panel depicts our theoretical findings for $g(v;a) = v(1-v)(v-a)$ and $k=4$. In the largest region (shown in red) we have $c<0$ by Theorem~\ref{prop:main_results:c_neg:d_large}.  The sets $\mathcal{D}^-$, in which $c<0$, and  $\mathcal{D}^+$ with $c>0$ are given by the exact formula's from Proposition~\ref{cor:main:cubic:neg} and Corollary~\ref{cor:main:cubic:pos}. In the bottom yellow region we have $c= 0$ by Proposition~\ref{prop:monotonic:existence}. 
%  The right panel compares our theoretical findings with numerically obtained regions where $c<0$ (purple with $-$ dashes), $c>0$ (green with $+$ dashes) and $c = 0$ (white). \note{VS:check caption}}
Our analytical and numerical results for the cubic \eqref{eqn:intro:cubic} and $k=4$. The left panel shows the analytical regions $\mathcal{D}^0(k)$, $\mathcal{D}^-(k)$ and $\mathcal{D}^+(k)$ for which we prove $c=0$, $c<0$ and $c>0$, respectively (see Propositions~\ref{prop:monotonic:existence}, \ref{cor:main:cubic:neg} and Corollary~\ref{cor:main:cubic:pos}). Moreover, for $d>d^*(a,4)$ the negative speed $c<0$ is implied by Theorem~\ref{prop:main_results:c_neg:d_large}. In the right panel these analytical regions (gray) are accompanied by the numerical regions (hatched for $c\neq 0$ and white for $c=0$). Compare with Fig.~\ref{fig:3_ad_regions}.
}
  \label{3:fig:main:numerics+theory:k2}
\end{figure}

\begin{theorem}\label{prop:main_results:c_neg:d_large}
Assume that \Hg holds. Pick $a \in (0,1)$ 
together with $k > 1$ and define the quantity
\begin{align}\label{eq:d-large-d-star}
d^*(a, k) = (1-k^{-1/2})^{-2} d^+(a,k)% \frac{1}{(k-\sqrt{k})\left(1-\tfrac{1}{\sqrt{k}}\right)} \max_{s \in ([1-a,1]} \frac{-g(1-s;a)}{s}.
\end{align}

Then for any $d > d^*(a,k)$ we have
$c(a,d,k)<0$.
%be given. There exists a curve $d^*:(0,1) \times (1, \infty) \to \mathbb{R}^+$ such that . The curve $d^*$ is given by 
\end{theorem}

Our second main result provides an  alternative
set of criteria that guarantee both strictly positive and negative wave speeds.  Our numerical results
for the cubic nonlinearity \eqref{eqn:intro:cubic} show that
the boundary of the associated parameter  sets track the edge of the pinning region rather well for a wide range of parameters $a$; see Fig.~\ref{3:fig:main:numerics+theory:k2}.

%
%o determine the sign of the wave-speed. The formulation of these result which relies on a geometric construction. 
%
%In our supplementary result, Theorem~\ref{prop:main_results:c_neg:d_small}, we give another criteria to determine the sign of the wave-speed,  which relies on a geometric construction. 
%Both of this results, Theorem~\ref{prop:main_results:c_neg:d_large} and Theorem~\ref{prop:main_results:c_neg:d_small} have their advantages. On the one hand, the curve $d^*$ from Theorem~\ref{prop:main_results:c_neg:d_large}  exists for all $a\in (0,1)$, which means that for all $d\gg 0$ we have $c(a, d, k)<0$.  On the other hand, via Theorem~\ref{prop:main_results:c_neg:d_small}  we can find two nonempty sets $\mathcal{D}^-$ and $\mathcal{D}^+$ in which we have negative, resp., positive speed. 
%These regions exist for small $d$ and numerical observations show that they closely approximate the area near the pinning region for cubic nonlinearity, see Figure ~\ref{3:fig:main:numerics+theory:k2}.

%In what follows we provide a geometric method to determine whether the wave is moving to the right, i.e., $c<0$.  For the bi-stable cubic nonlinearity, these criteria closely approximates the region just above the pinning area, see Fig.~\ref{3:fig:main:numerics+theory:k2}. 

The characterization of these sets depend on a geometric construction
involving the graph of the nonlinearity $g$.
To describe this construction,
 we pick parameters $k>0$, $a\in (0,1)$, $A\in (0,1)$, $d>0$ and define a linear map $\ell_d(\cdot; a, k, A):(0, 1)\to \R$ that acts as
 \begin{equation}\label{eqn:main:lin:-}
    \ell_d(v; a,  k, A) : =  d(k+1)(A-v) - g(A;a).
 \end{equation}
This linear map intersects the graph of $-g$ at $v=A$ and slopes downward with a steepness that is proportional to $d$. The smoothness of $g$ now allows
 us to define
\begin{equation}\label{eqn:main:d_diamond:neg}
    %d^\diamond(A;a)  = \inf\{d>0: d(k+1)(A-v) - g(A;a)\geq - g(v;a), \text{ for all } v\in [0, A]\}.
    d^\diamond(A;a)  = \inf\{d>0: \ell_d(v; a,  k, A) \geq - g(v;a), \text{ for all } v\in [0, A]\},
\end{equation} 
representing the minimal value of $d$ that is required to ensure that
$\ell_d$ stays above the graph of $-g$ on $[0,A]$; see Fig.~\ref{3:fig:main:d_diamond:neg}.

%With this definition in mind, we further define 

%We now give a geometric interpretation of $d^\diamond$. The linear map~\eqref{eqn:main:lin:-}
%\begin{equation*}
%    v\mapsto d(k+1)(A-v) - g(A;a)
%\end{equation*}
% crosses the graph $-g(v;a)$ exactly at point $A$. Therefore, 
%to find the value of $d^\diamond$,  one has to find a line with the minimal slope  that starts at point $(A, -g(A;a))$ and stays above the graph $-g(v;a)$ for all $v\in [0, A]$, see Fig.~\ref{3:fig:main:d_diamond:neg}. This construction is possible for any continuous function and bounded function $g$. 

\begin{figure}
\centering
\includegraphics[width=\textwidth]{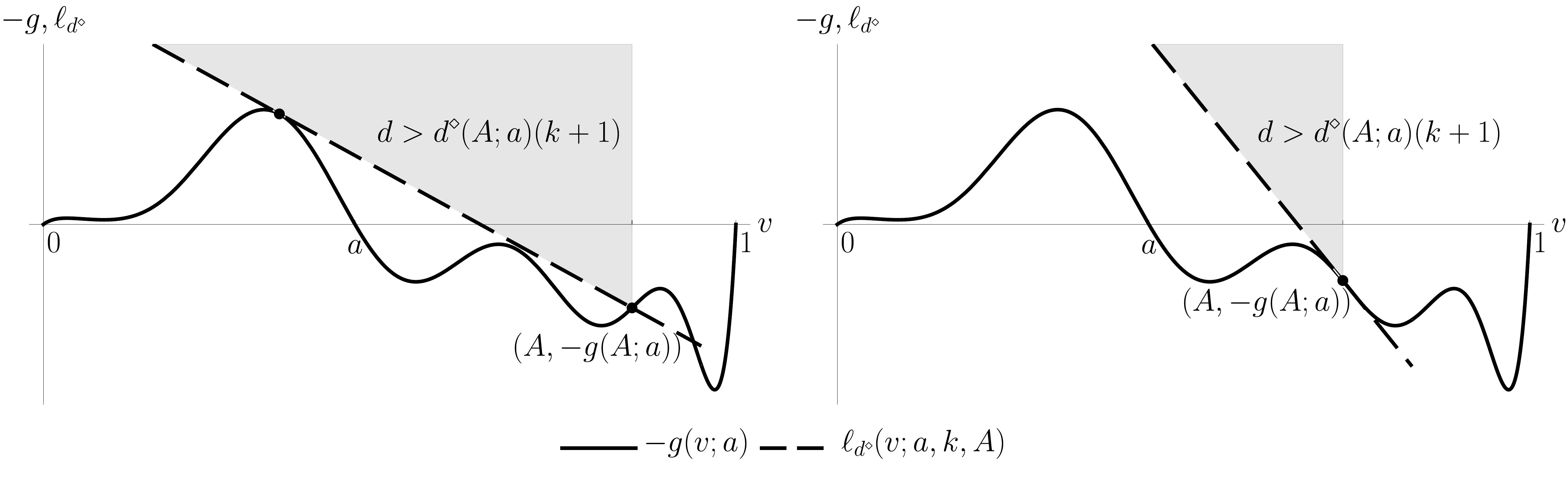}
\vspace{-.75cm} % VS:an ugly workaround; I do not know how to fix it right now
\caption{In the left panel, the value of $d^\diamond = d^\diamond(A;a)$ is determined by the tangential intersection of the line and the graph of $-g$ at a point $v\leq A$. In the right panel we have $-d^\diamond(A;a)(k+1) = g'(A;a).$ } \label{3:fig:main:d_diamond:neg}
\end{figure}

%\begin{figure}
% \centering
%          \includegraphics[width=\textwidth]{fig/d_diamond_neg_2.pdf}
%          \caption{In the left panel, the value of $d^\diamond = d^\diamond(A;a)$ is determined by the tangential intersection of the line and the graph of $-g$ at some point $v\leq A$. In the right panel we have $-d^\diamond(A;a)(k+1) = g'(A;a).$ } \label{3:fig:main:d_diamond:neg}
%\end{figure} 

We now define the set $\mathcal{D}^-\subset \mathcal{H}$ by writing
\begin{equation}\label{eqn:main:Dminus}
   \mathcal{D}^-(k) = \mathcal{D}^-_g(k) : = %\left\{ (a,d)\in \mathcal{H}:\  d\in  \left(d^\diamond(A;a), \dfrac{g(A;a)}{A}\right) \text{ for some } A \in (a,1)  \right\},
   \left\{ (a,d)\in \mathcal{H}:\    d^\diamond(A;a) < d < \dfrac{g(A;a)}{A}  \text{ for some } A \in (a,1)  \right\},
\end{equation}
using the subscript only when required for explicitness.
We will show that $c < 0$ on $\mathcal{D}^-$,
%noting that  the interval $(d^\diamond(A;a), {g(A;a)}/{A})$ is possibly empty for some combinations of parameters $a$ and $A$. 
%Note that  $\mathcal{D^-}$  
which is a priori bounded from above by the function $d^-(a)$  defined in~\eqref{eqn:pinning:d-}.
To tackle the opposite case $c>0$, we exploit the symmetry
\eqref{eqn:main:symmetry_cphi} 
%introduced with Lemma~\ref{prop:main:symmetry} and
and
define the set %$\mathcal{D}^+\subset \mathcal{H}$ by
\begin{equation}\label{eqn:main:Dplus}
\mathcal{D}^+(k) = \left\{(a,d)\in \mathcal{H}: (\tilde{a}, \tilde{d}) \in \mathcal{D}^-_{\tilde{g}}\left( \tilde{k}\right)\right\}.
\end{equation}
Upon introducing the notation (symmetrical to \eqref{eqn:main:d_diamond:neg})
\begin{multline}\label{eqn:main:d_diamond:pos}
    \tilde{d}^\diamond(A;a)  = \inf\left\{\tilde{d}>0: \tilde{d}\left(1+\frac{1}{k}\right)(A-v) + g(1-A;a)\geq  g(1-A;a), 
    \text{ for all } v\in [0, A]\right\},
\end{multline}
the definition \eqref{eqn:main:Dplus}
can be recast in the form
\begin{equation*}
     \mathcal{D}^+(k)  = \left\{ (a,d)\in \mathcal{H}:\   \dfrac{\tilde{d}^\diamond(A;a)}{k} < d <  -\dfrac{g(1-A;a)}{kA}
     %\\
     \text{ for some } A \in (1-a,1)  \right\},
\end{equation*}
which only involves the original nonlinearity. Notice again
that this set  is a priori bounded from above by the function $d^+(a, k)$ defined in~\eqref{eqn:pinning:d-}.

\begin{theorem}\label{prop:main_results:c_neg:d_small}
Assume that \Hg is satisfied and pick $k>0$. Then the following claims hold true.
\begin{enumerate}[(i)]
%\item We have $(a,d)\in \mathcal{D}^-( g, k) \iff (\tilde{a}, \tilde{d})\in \mathcal{D}^+ ( \tilde{g}, \tilde{k})$.
\item\label{item:main:D_nonempty} We have $\mathcal{D}^-(k)\neq \emptyset$ and $\mathcal{D}^+(k)\neq \emptyset$.
\item \label{item:main:c_neg_in_D} For all $(a,d)\in \mathcal{D}^-(k)$ we have  $c(a,d, k) < 0$. Equivalently, for all $(a,d)\in \mathcal{D}^+(k)$ we have $c(a,d, k) > 0$. 
\item\label{item:main:monotonicity_in_D}  
Assume that \Hgone holds and pick any $(a,d)\in \overline{\mathcal{D}^-}\cap \mathcal{H}$. Then we have $c(a',d)<0$ for all $0 <a'<a$. Similarly, pick any $(a,d)\in \overline{\mathcal{D}^+}\cap \mathcal{H}$. Then we have $c(a',d)>0$ for all $a<a'<1$.
\item \label{item:main:curve_in_D} Assume that \Hgtwo holds. Then there exists $\delta_a\in (0,1)$ such that
    \begin{equation}
     \left(a,  \dfrac{g'(a;a)}{k+1}\right )\in 
    \begin{cases}
     \mathcal{D}^-(k),  &\text{ for } a\in (0, \delta_a), \\
     \mathcal{D}^+(k),  &\text{ for } a\in (1-\delta_a, 1). 
    \end{cases}
       \end{equation}
\end{enumerate}
%\todo[inline]{+[hjh: use $a \in (0, \delta_a)$ and $a \in (1-\delta_a, 1)$ or something]}
\end{theorem}
Note that the condition \Hgone
implies that we can fully characterize $\mathcal{D}^-$ and $\mathcal{D}^+$ 
by finding their right and left boundaries, respectively.
In addition, the assumption \Hgtwo guarantees that the set $\mathcal{D}^-$
extends to the corner $(a, d) = (0,0)$,
while $\mathcal{D}^+$ extends to
  $(a,d)= (1,0)$.

\subsection{Cubic nonlinearity}\label{3:subsec:cubic}

In this subsection we apply our techniques to the standard cubic 
nonlinearity~\eqref{eqn:intro:cubic}. In particular,
we obtain explicit expressions for the curves and regions 
that appear in our main results. First, we
describe the functions $d^\pm$ and $d_0$ that characterize
the pinning region and the chaotic behaviour therein. As an immediate consequence we also get an explicit expression for the curve $d^*$,
above which the wave speed is guaranteed to be negative.

%In Proposition~\ref{prop:cubic:pinning}, we first give explicit expressions for pinning regions from Propositions~\ref{prop:monotonic:existence} and~\ref{prop:prop_failure:steady_sols:existence}. 
%We also give the explicit formula for the curve $d^*$ from Theorem~\ref{prop:main_results:c_neg:d_large}. 

\begin{lemma}\label{lemma:cubic:d+d-}
Let $g$ be the standard cubic nonlinearity~\eqref{eqn:intro:cubic}. Then the explicit expressions for the functions $d^-$ and $d^+$ defined by \eqref{eqn:pinning:d-} are given by
\begin{align*}
     d^-(a) = \dfrac{(1-a)^2}{4}, \qquad d^+(a,k) = \dfrac{a^2}{4k}.
\end{align*}
\end{lemma}
\begin{proof}
This claim follows from a straightforward analysis of quadratic expressions. 
\end{proof}

\begin{proposition}\label{prop:cubic:pinning}
Let $g$ be the standard cubic nonlinearity~\eqref{eqn:intro:cubic} and pick parameters $k>0$ and $a\in (0,1)$. Then the following claims hold true.
\begin{enumerate} [(i)]
    \item \label{3:item:cubic:pinn1}
Pick any $d>0$ that satisfies
\begin{equation*}
    d < \min \left\{\dfrac{a^2}{4k}, \dfrac{(1-a)^2}{4}\right\} = \begin{cases}
    \dfrac{a^2}{4k}, \qquad &a \leq \dfrac{1}{1/\sqrt{k}+1}, \\
    \dfrac{(1-a)^2}{4}, \qquad &a\geq  \dfrac{1}{1/\sqrt{k}+1}.
    \end{cases}
\end{equation*}
Then we have $c(a, d, k) = 0$.

\item\label{3:item:cubic:pinn2} The function~$d_0$ from Proposition~\ref{prop:prop_failure:steady_sols:existence} is given by
\begin{equation}\label{eqn:cubic:d0}
    d_0(a, k):= \dfrac{1}{k+1} \min\left\{ k d^+(a, k), d^-(a)\right\}.
\end{equation}
In particular, for any $k>0$, $a\in (0,1)$ and $0<d<d_0(a,k)$ there exist infinitely many bounded solutions to~\eqref{eqn:main:MFDE} with $c=0$.

\item\label{3:item:cubic:neg_speed} Assume that $k>1$ and pick any $d$ that satisfies
\begin{equation*}
    d > \dfrac{a^2}{4(\sqrt{k}-1)^2}.
\end{equation*}
Then we have $c(a, d, k) < 0$.
\end{enumerate}
\end{proposition}
\begin{proof}
The proof of items \textit{\eqref{3:item:cubic:pinn1}} and  \textit{\eqref{3:item:cubic:neg_speed}}  follows directly from Proposition~\ref{prop:monotonic:existence} and Theorem~\ref{prop:main_results:c_neg:d_large} applied to Lemma~\ref{lemma:cubic:d+d-}. %The proof of \textit{\ref{3:item:cubic:neg_speed}} follows from Lemma~\ref{lemma:chaos:cubic}.  
For the proof of \textit{\eqref{3:item:cubic:pinn2}}, see \S\ref{sec:chaos}. 
\end{proof}

\begin{figure}[t]
\centering
\includegraphics[width=\textwidth]{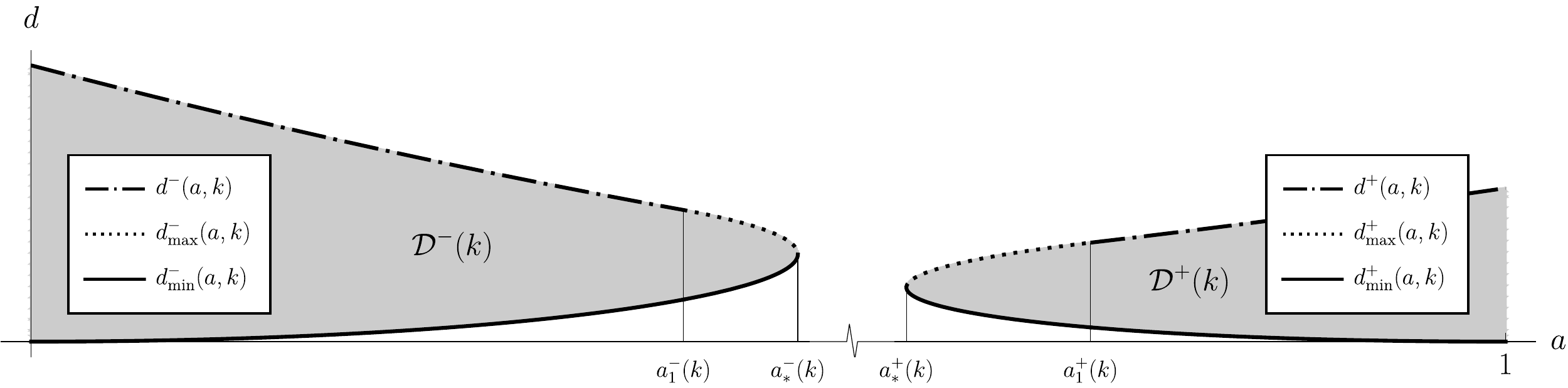}
     
   \caption{Curves and quantities defining regions $\mathcal{D}^-(k)$ and $\mathcal{D}^+(k)$, see Propositions~\ref{prop:cubic:pinning}, \ref{cor:main:cubic:neg} and Corollary~\ref{cor:main:cubic:pos}.   For a wider context, see Fig.~\ref{3:fig:main:numerics+theory:k2}.
%   \note{VS:I hijacked this place to try to put the illustration for cubic regions here. But now, we do not have a picture of the spatial chaos region... Nothing uses this figure as a reference.} 
   }   \label{fig:Dminus_Dplus}
\end{figure} 
%\begin{figure}
%\centering
%\begin{subfigure}{0.45\textwidth}
%         \centering
%         \includegraphics[width=\textwidth]{fig/standard_cubic_1_k_2.pdf}
%       \end{subfigure}
%    \begin{subfigure}{0.45\textwidth}
%         \centering
%         \includegraphics[width=\textwidth]{fig/standard_cubic_1_k_5.pdf}
       
%     \end{subfigure}
%   \caption{These images depict the curves $d^-$, $d^+$ and $d_0$ from Propositions~\ref{prop:monotonic:existence} and \ref{prop:prop_failure:steady_sols:existence} for the standard cubic nonlinearity \eqref{eqn:intro:cubic}, with $k=2$ (left) and $k=5$ (right). These images already suggest that the pinning region is asymmetric, and that the area  in which $c>0$ decreases as $k$ increases.  }   
%\end{figure} 

%Then we proceed to Proposition~\ref{cor:main:cubic:neg} and its Corollary~\ref{cor:main:cubic:pos} in which we give the exact boundaries of the sets 
%$\mathcal{D}^+$ and $\mathcal{D}^-$.  
%Item~\textit{\ref{item:main:monotonicity_in_D}} of Theorem~\ref{prop:main_results:c_neg:d_small} allows us to conclude that $\mathcal{D}^-$ and $\mathcal{D}^+$ are given precisely as the regions between these curves. 

We now set out to find explicit expressions for $\mathcal{D}^-$ and $\mathcal{D}^+$. Item~\textit{\eqref{item:main:monotonicity_in_D}} of Theorem~\ref{prop:main_results:c_neg:d_small} shows that
it suffices to find the outer boundaries of these sets.
To this end,
 we first define %values $a^-_1(k)$ and $a^-_*(k)$ by
 the quantities
\begin{equation}
    a_*^-(k) := 
    1- \frac{2}{\sqrt{4k+1}+1} , \qquad 
    a^-_1(k) := \max\left\{
   1-\dfrac{2}{2\sqrt{k}+1}, 0\right\}, \label{eqn:cubic:neg:a1} 
\end{equation}
for $k > 0$,
together with the
curves %$d^-_{\max}, d^-_{\min}:[0, a^-_*(k)]\times (0, \infty)\to \R$ 
\begin{align*}
    d^-_{\min}(a, k)&:=  \dfrac{2a^2k -a + 2k  - 2(a+1) \sqrt{k} \sqrt{ka^2 -a (2k+1) + k} }{(4k+1)^2}, \\[0.4cm]
    d^-_{\max}(a, k)&:=  \begin{cases}
     \dfrac{(1-a)^2}{4}, &  \text{if } a \in [0, a^-_1(k)), \\[0.3cm]
     \dfrac{2a^2k -a + 2k  + 2(a+1) \sqrt{k} \sqrt{ka^2 -a (2k+1) + k} }{(4k+1)^2} , &\text{if }a\in [a^-_1(k), a^-_*(k)],
     \end{cases}%\\[0.3cm]
\end{align*}
for $k > 0$ and $0 \le a \le a_*^-(k)$.
%As it turns out in our following proposition,
Together, these curves define the boundary of $\mathcal{D}^-$, see Fig.~\ref{fig:Dminus_Dplus}. 

\begin{proposition}\label{cor:main:cubic:neg}[Cubic nonlinearity, negative speed]
Pick $k>0$ and let $g$ be the standard cubic nonlinearity~\eqref{eqn:intro:cubic}. 
Then the following claims hold. 
\begin{enumerate}[(i)]
\item \label{item:main:cubic:1} We have $d^-_{\min}$, $d^-_{\max} \in C([0, a_*^-(k)]\times (0, \infty); \R)$ and
\begin{align*}
   0\leq d^-_{\min}(a, k) &\leq d^-_{\max} (a,k)\leq d^-(a), \quad \text{for all } a\in [0, a^-_*(k)]. %\\
   % d^-_{max}\left(a_*^-(k)\right) &= d^-_{min}\left(a_*^-(k)\right)
\end{align*}
\item\label{item:main:cubic:2} The equality $d^-_{\max}\left(a_*^-(k)\right) = d^-_{\min}\left(a_*^-(k)\right)$ holds.
\item\label{item:main:cubic:3} The set $ \mathcal{D}^-(k)$ is bounded precisely by the graphs of $d^-_{\min}$ and $d^-_{\max}$, namely
\begin{equation}\label{eqn:main:Dminus_characterization}
    \mathcal{D}^-(k) = \left\{(a,d)\in \mathcal{H}:  a< a^-_*(k), \quad d^-_{\min}(a, k) < d< d^-_{\max}(a, k)\right\}.
\end{equation}
\end{enumerate}
\end{proposition}

To formulate the equivalent result for the set $\mathcal{D}^+$, we again 
%pick $k>0$ t
define two values % $ a_*^+(k)$ and $ a_1^+(k)$
\begin{equation}
\begin{aligned}
      a_*^+(k) = %&:= 1- a_*^-\left(\frac{1}{k}\right) \\ &= 
      \dfrac{\sqrt{4k+k^2} - k}{2}, \qquad 
\end{aligned}
     \quad 
\begin{aligned}
      a_1^+(k) %&:= 1- a_1^-\left(\frac{1}{k}\right) \\
            =\min\left\{\dfrac{2\sqrt{k}}{2+\sqrt{k}}, 1 \right\}%\min\{\dfrac{2\sqrt{k}}{2+\sqrt{k}}, 1 \}\label{eqn:cubic:pos:a1} 
\end{aligned}
   \end{equation}
for $k > 0$,
together with  two curves %$d^+_{max}, d^+_{min}:[ a^+_*(k), 1]\times(0, \infty)\to \R$
\begin{align*}
    d^+_{\min}(a, k)&: %\dfrac{ d^-_{min}(1-a, 1/k)}{k} 
    = \dfrac{2a^2 + a(k-4) + 4-k - \sqrt{a^2 + ka - k}}{(4+k)^2} , \\[0.2cm]
    d^+_{\max}(a, k)&:=  %\dfrac{d^-_{max}(1-a, 1/k)}{k} \\
      \begin{cases}
     \dfrac{a^2}{4k}, \quad  &\text{if } a \in (a_1^+(k), 1), \\[0.3cm]
     \dfrac{2a^2 + a(k-4) + 4-k + \sqrt{a^2 + ka - k}}{(4+k)^2} , \ & \text{if }a\in [ a^+_*(k), a_1^+(k)],\\
     \end{cases}
\end{align*}
for $k > 0$ and $a_1^+(k) \le a \le 1$. See Fig.~\ref{fig:Dminus_Dplus} for illustration.

\begin{corollary}\label{cor:main:cubic:pos}[Cubic nonlinearity, positive speed]
Pick $k>0$ and let $g$ be the standard cubic nonlinearity~\eqref{eqn:intro:cubic}.
Then the following claims hold. 
\begin{enumerate}[(i)]
\item
We have $d^+_{\min}$, $d^+_{\max} \in C\left([a_*^+(k), 1]\times (0, \infty); \R\right)$ and
\begin{align*}
   0\leq d^+_{\min}(a, k) &\leq d^+_{\max} (a,k)\leq d^+(a,k), \quad \text{for all } a\in [0, a^+_*(k)]. 
%    d^+_{max}\left(a_*^+(k)\right) &= d^+_{min}\left(a_*^+(k)\right)
\end{align*}
\item The equality $d^+_{\max}\left(a_*^+(k)\right) = d^+_{\min}\left(a_*^+(k)\right)$ holds.
\item  The set $ \mathcal{D}^+$ is bounded precisely by the graphs of $d^+_{\min}$ and $d^+_{\max}$, namely
\begin{equation}\label{eqn:main:Dplus_characterization}
    \mathcal{D}^+(k) = \left\{(a,d)\in \mathcal{H}:  a> a^+_*(k), \ d^+_{\min}(a, k) < d< d^+_{\max}(a, k)\right\}.
\end{equation}
\end{enumerate}
\end{corollary}
\begin{proof}
This result follows directly from Proposition~\ref{cor:main:cubic:neg} and Lemma~\ref{prop:main:symmetry}  by noting that
\begin{equation*}
    a_*^+(k) = 1- a_*^-\left(\frac{1}{k}\right), \qquad 
    a_1^+(k) =  1- a_1^-\left(\frac{1}{k}\right) ,
\end{equation*}
together with
\begin{equation*}
    d^+_{\min}(a, k)= \dfrac{ d^-_{\min}(1-a, 1/k)}{k} , \qquad
    d^+_{\max}(a, k)=  \dfrac{d^-_{\max}(1-a, 1/k)}{k}.
\end{equation*}
\end{proof}

\section{Comparison principles}\label{sec:cp}

%The existence of the monotone wave profile $\Phi$ satisfying \eqref{eqn:main:MFDE} with the appropriate boundary condition was already shown in \cite{Mallet-Paret1999}. Our main interest lies in the sign of the corresponding wave speed $c \in \mathbb{R}$ for given $a,d,k$. To this end we use the following comparison principle. 

%\note{mj: I added monotonicity result - Lemma 3.2 to this section so I propose the following alternative introduction text, or smth similar that mentions this result }

The main tool that we use in this paper to analyze the
LDE \eqref{eqn:main:LDE} is the well-known comparison principle,
which is formulated in the first result below. We will exploit this
principle in a standard fashion to show that solutions with monotonic
initial conditions remain monotonic. In addition, we show
how the sign of the wave speed $c(a,d, k)$ defined in Proposition \ref{prop:main:MP} can be controlled by constructing appropriate 
lower and upper solutions.

\begin{lemma}\label{lem:comparison}
Let $u,v \in C^1\big([0,\infty), \ell^\infty(\mathbb{Z})\big)$ be such that
\begin{align}\label{eq:comparison}
\begin{split}
\dot{u}_i(t) & \geq d [\Delta_k u(t)]_i + g(u_i(t);a), \\
\dot{v}_i(t) & \leq d  [\Delta_k v(t)]_i + g(v_i(t);a) 
\end{split}
\end{align}
and $u_i(0) \ge v_i(0)$ for all $i \in \mathbb{Z}$. Then $u_i(t) \ge v_i(t)$ for $t>0$ and all $i \in \mathbb{Z}$.  
\end{lemma} 

\begin{proof}
The statement is the reformulation of \cite[Lemma 1]{ChenGuoWu2008} with $j=\infty$ and
\[
\mathcal{N}_i u(t) = \dot{u}_i(t) - d[\Delta_k u(t)]_i  - g(u_i(t);a).
\]
\end{proof}

\begin{lemma}\label{lemma:stability_standing:monotonic}
Assume that \Hg holds and pick a non-decreasing sequence $u^0\in \ell^\infty(\Z)$ . Then the solution $u(t)$ to the LDE~\eqref{eqn:main:LDE} with $u(0) = u^0$ is also a  non-decreasing sequence for all $t>0$.
\end{lemma}
\begin{proof}
Define the function $v(t)$ with $v_i(t) = u_{i+1}(t)$. Then this function also satisfies the LDE~\eqref{eqn:main:LDE} and we have $u^0_i \leq v^0_i$ for each $i\in \Z$. By Lemma~\ref{lem:comparison} this implies $u_i(t)\leq v_i(t)$ for all $i\in \Z$ and $t>0$. 
\end{proof}

%We call the function corresponding to $u$ (resp. $v$) in Lemma~\ref{lem:comparison} an upper (resp. lower) solution of the LDE~\eqref{eqn:intro:kLDE}. %\todo{mj: I would maybe omit this text at this point }Contrary to the classical use of lower and upper solutions as starting points of monotone iteration methods we use the comparison principle directly in two ways. %Firstly, we show that solutions of the LDE~\eqref{eqn:intro:kLDE} preserve the monotonicity of the initial condition, Lemma~\ref{lemma:stability_standing:monotonic}.
%Secondly, provided we have an upper (lower) solution in the form of a right (left) traveling wave with the proper ordering it must ``push'' the solution of the LDE~\eqref{eqn:intro:kLDE} in the desired direction with either $c>0$ or $c<0$. 

In order to translate the inequalities
\eqref{eq:comparison} to the context of traveling waves,
we introduce the
operators %\todo{mj: I added the subscript $g$} 
$\mathcal{I}_{a,d, k}:\mathbb{R} \times C^1(\mathbb{R}) \to C(\mathbb{R})$ that act as
\begin{equation}\label{eqn:I_operator}
\begin{aligned}
    \mathcal{I}_{a,d, k}[c, \Phi](\xi) := & -c \Phi'(\xi) - d(\Phi(\xi-1) - (k+1) \Phi(\xi) + k \Phi(\xi+1))  -g(\Phi(\xi);a). 
    \end{aligned}
\end{equation}
This can be interpreted as the residual of the traveling-wave equation \eqref{eqn:main:MFDE}, i.e., $\mathcal{I}_{a,d, k}[c, \Phi]=0$ if and only if the pair $(c, \Phi)$ solves \eqref{eqn:main:MFDE}. 

%Let us define an operator $\mathcal{I}_{a,d,g, k}:\mathbb{R} \times C^1(\mathbb{R}) \to C(\mathbb{R})$ by
%\begin{equation}\label{eqn:I_operator}
%\mathcal{I}_{a,d,g, k}[c, \Phi](\xi) := -c \Phi'(\xi) - d(k\Phi(\xi+1) - (k+1) \Phi(\xi) + \Phi(\xi-1))-g(\Phi(\xi),a),
%\end{equation}
%and state a direct consequence of Lemma~\ref{lem:comparison}. 

\begin{corollary}\label{lem:comparison-MFDE} 
Pick $k>0$, a pair $(a,d)\in \mathcal{H}$ and a function $g$ that satisfies \Hg. Let the pair $(c,\Phi)$ be a solution of~\eqref{eqn:main:MFDE}-\eqref{eqn:main:bc}.
Assume that there exist a
constant $\overline{c}\in \mathbb{R}$
and a bounded function 
$\Psi \in C^1(\mathbb{R})$
 %constant $\overline{c}<0$ (resp. $\overline{c}> 0$) and a function 
%$(\overline{c}, \Psi) \in \R \times C^1(\R)$
that satisfy the properties
\begin{enumerate}[(i)]
    \item $\sup_{\xi\in \R} \Psi(\xi) > 0 \quad (\text{resp.} \; \inf_{\xi\in \R} \Psi(\xi) < 1 )$,
    \item $\Psi(\xi) \leq \Phi(\xi) \quad (\text{resp.} \; \Psi(\xi) \geq \Phi(\xi))$ for all $\xi\in \R$,
    \item $\mathcal{I}_{a,d, k}[\bar{c}, \Psi](\xi) \le 0 \quad (\text{resp.} \; \mathcal{I}_{a,d, k}[\bar{c}, \Psi](\xi) \geq 0)$ for all $\xi\in \R$.
\end{enumerate}
Then we have $c \le \overline{c}$ (resp. $c \ge \overline{c}$).
%
%Let the pair $(c,\Phi)$ be a solution of~\eqref{eqn:main:MFDE}. If there exists a pair $(\bar{c},\Psi)$ such that $\bar{c}<0$ (resp. $\bar{c}> 0)$, $\Psi \in C^1(\mathbb{R})$, the ranges of $\Phi$ and $\Psi$ have nonempty intersection, i.e., 
%\[
%\mathrm{Rng}(\Phi) \cap \mathrm{Rng}(\Psi) \ne \emptyset
%\]
%and the following hold
%\begin{align*}
%\Psi(\xi) &\leq \Phi(\xi) \quad (\text{resp.} \; \Psi(\xi) \geq \Phi(\xi)), \\
%\mathcal{I}_{a,d,g, k}[\bar{c}, \Psi](\xi) &\le 0 \quad (\text{resp.} \; \mathcal{I}_{a,d, k}[\bar{c}, \Psi](\xi) \geq 0)
%\end{align*}
%for all $\xi \in \mathbb{R}$. Then $c \le \bar{c}$ (resp. $c \ge \bar{c}$).
\end{corollary}

\begin{proof}
Without loss of generality, we consider the case $\Psi \le \Phi$.
Let us define two time-dependent sequences,  $v_i(t) := \Psi(i-\bar{c}t)$ and $u_i(t) := \Phi(i-ct)$, for $i\in \Z$.
%The corresponding doubly-infinite time-dependent sequence $v_i(t) := \Psi(i-\bar{c}t)$ satisfies (see~\eqref{3:eqn:I_operator})
%\[
%\dot{v}_i(t)  \leq d(v_{i-1}(t)-(k+1)v_i(t)+k v_{i+1}(t)) + g(v_i(t);a) 
%\]
%We assume that the profile $\Psi$ is shifted in such manner that $v_i(0)\le u_i(0)$ for all $i \in \mathbb{Z}$. Therefore,
By construction, the assumptions of Lemma~\ref{lem:comparison} are satisfied and we therefore have 
\begin{equation}\label{3:eqn:comparison:v<u}
  \Psi(i-\overline{c}t) \leq \Phi(i-ct)
\end{equation}
for all $i\in \Z$ and $t>0$. 

To show that $c\le\overline{c}$, we assume to the contrary that $c>\overline{c}$. 
Let $\xi_0  \in \R $ be such that $\Psi(\xi_0) = M > 0$.   Due to the shift-invariance of the MFDE~\eqref{lem:comparison-MFDE},  we can  shift both $\Psi$ and $\Phi$ to have $\xi_0 = 0$. 
In addition, due to the first
limit in \eqref{eqn:main:bc}
we can find an integer $i_1 < 0$
so that $\Phi(\xi)<M/2$ for all $\xi \le i_1$.  We now write
\begin{equation}
i_1 = - (c - \overline{c}) t_1
\end{equation}
and pick $t_2 \ge t_1 > 0$ in such a way that
$\overline{c} t_2 = i_2 \in \mathbb{Z}$.
We now have $v_{i_2}(t_2) = \Psi(i_2 - \overline{c} t_2  ) = \Psi(0) = M$. On the other hand, 
since 
\begin{equation}
i_2 - c t_2 = 
(\overline{c} - c) t_2 
\le (\overline{c} - c) t_1 = i_1
\end{equation}
we have $u_{i_2}(t_2) = \Phi(i_2 - ct_2) < M/2 $, which clearly contradicts~\eqref{3:eqn:comparison:v<u}.

\end{proof}

%When trying to find a lower or an upper solution $\Psi$ to be used in Corollary~\ref{lem:comparison-MFDE}, one has to verify two conditions. The first is to show the sign of the function $\mathcal{I}_{a,d,g, k}[\bar{c},\Psi](\, \cdot \,)$. The second one is to show that the profile $\Psi$ is suitably ordered with respect to the solution $\Phi$. To this end, we have only a partial information about $\Phi$; namely the boundary condition. We deal with this lack of information by designing various profiles $\Psi$ to be either constant or sufficiently separated from levels $u=0$ and $u=1$ for $|\xi| \gg 1$. One can argue that some of the results can be optimized by exploiting the asymptotic expansions of $\Phi$ near $\pm \infty$, \cite[Theorem 2.2]{Mallet-Paret1999}. 
%Also, since we are interested in parameter regions in which either $c>0$ or $c<0$, we usually assume $\bar{c} \ll 1$.

\section{Pinned monotonic waves}\label{sec:pinning}

In this section we follow the approach from~\cite{keener1987propagation} to establish Proposition~\ref{prop:monotonic:existence}. 
The series of Lemmas~\ref{lemma:monotonic:aux_lema_1}, \ref{lemma:monotonic:aux_lema_2} and~\ref{prop:stability_standing:invariant_intervals_monotonic} yield the existence of two invariant  intervals $(x_1, 1]$ and $[0, y_2)$  for the LDE~\eqref{eqn:main:LDE}. 
More precisely, choosing  $(a,d)\in \mathcal{H}$ with $d<d^-(a)$, we have $ u_i(t) \in (x_1, 1]$ provided that $u_i^0 \in  (x_1, 1]$. This feature blocks propagation to the right since traveling waves are known to be strictly monotonic \cite{Mallet-Paret1999}. 
On the other hand,  the interval $[0, y_2) $ is invariant for the LDE~\eqref{eqn:main:LDE} when $d<d^+(a, k)$, which blocks propagation to the left. 
%Once when we have the existence of these invariant intervals, we use the fact that  the traveling waves are nondecreasing  \cite{Mallet-Paret1999} to finalize the proof of Proposition~\ref{prop:monotonic:existence}.

\begin{lemma}\label{lemma:monotonic:aux_lema_1}
Consider the setting of Proposition~\ref{prop:monotonic:existence}. Pick any $a\in(0,1)$ and $d< d^-(a)$. Then there exist two points $x_1, x_2$, with $a<x_1<x_2<1$ such that 
\begin{equation}\label{eqn:pinning:aux1}
    dy - g(y;a) < 0 , \qquad y\in (x_1, x_2).
\end{equation}
\end{lemma}
\begin{proof}
Let us take $d<d^-(a)$. By definition of $d^-$, there exists  $x_0\in (a,1)$ such that 
\begin{equation*}
   d < \dfrac{g(x_0;a)}{x_0}.
\end{equation*}
The strict inequality ensures that there exists an interval $(x_1, x_2)$ around $x_0$   such that~\eqref{eqn:pinning:aux1} holds. 
\end{proof}

\begin{lemma}\label{lemma:monotonic:aux_lema_2}
Consider the setting of Proposition~\ref{prop:monotonic:existence}. Pick any $a\in(0,1)$ and $d< d^+(a, k)$. Then there exist two points $y_1, y_2$ with  $1-a<y_1<y_2<1$ such that 
\begin{equation*}
    d ky + g(1-y;a) < 0 , \qquad y\in (y_1, y_2).
\end{equation*}
\end{lemma}
\begin{proof}
The proof is analogous to that of Lemma~\ref{lemma:monotonic:aux_lema_1}.
\end{proof}

\begin{lemma}\label{prop:stability_standing:invariant_intervals_monotonic}
Assume that \Hg holds and pick a pair $(a,d) \in \mathcal{H}$
together with a non-decreasing sequence $u^0\in \ell^\infty(\Z)$ that
has $0\leq u^0_i\leq 1 $ for all $i\in \Z$. Let $u(t)$ be the solution to the LDE~\eqref{eqn:main:LDE} with $u(0) = u^0$. Then the following claims hold. 
\begin{enumerate} [(i)]
    \item \label{item:stability:invariant_intervals:mon:1} If $d < d^+(a, k)$ and $u^0_i \in [0, y_2)$ for some $i \in \mathbb{Z}$, then $u_i(t)\in [0, y_2)$ for all $t>0$.
    \item\label{item:stability:invariant_intervals:mon:2} If $d < d^-(a)$ and $u^0_i \in (x_1, 1]$ for some $i \in \mathbb{Z}$, then $u_i(t)\in (x_1, 1]$ for all $t>0$.
\end{enumerate}
\end{lemma}
\begin{proof}

By the comparison principle we have $u_i(t)\in [0,1]$ for all $t>0$. By Lemma~\ref{lemma:stability_standing:monotonic} we also know that $u_i(t)$ is a monotonic sequence for all $t>0$. Assume that $u_i^0 \in [0, y_2)$ and that there exists $t>0$ such that $u_i(t) \geq y_2$. A continuity argument  ensures that there exists $t_0>0$ such that $u_i(t_0) \in (y_1, y_2)$ and $\dot{u}_i(t_0) \geq 0 $. However, by Lemma~\ref{lemma:monotonic:aux_lema_2}, we have 
\begin{equation*}
\begin{aligned}
  \dot{u}_i(t_0) &= d\left(u_{i-1}(t_0) - u_i(t_0) + k \big(u_{i+1}(t_0) -  u_i(t_0)\right) \big) + g(u_i(t);a) \\
  &\leq d\big(  k (1-  u_i(t_0)) \big) + g(u_i(t);a)< 0, 
\end{aligned}
\end{equation*}
which contradicts our assumption $\dot{u}_i(t_0) \geq 0  $. This proves item \textit{(\ref{item:stability:invariant_intervals:mon:1})}. 
Item~\textit{(\ref{item:stability:invariant_intervals:mon:2})} follows similarly. 
\end{proof}
\begin{proof}[Proof of Proposition~\ref{prop:monotonic:existence}]
 In view of Proposition~\ref{prop:main:MP}, we can find a solution $(\Phi, c)$ to~\eqref{eqn:main:MFDE}-\eqref{eqn:main:bc}. Setting
 \begin{equation*}
     u_i^0:=\Phi(i),
 \end{equation*}
 we see that $u^0$ is a non-decreasing sequence connecting $0$ and $1$. Let $I_1\in \Z$ be such that
 \begin{align*}
     \Phi(i)<  y_2 , \qquad i\leq I_1.
\end{align*}
Assume now that $0<d<d^+(a, k)$. By item~\textit{(\ref{item:stability:invariant_intervals:mon:1})} of Lemma~\ref{prop:stability_standing:invariant_intervals_monotonic} the associated wave solution $u_i(t) = \Phi(i-ct)$ has 
$u_i(t) < y_2 $ for all $t>0$ and $i\leq I_1$, which implies $c\geq 0$. 
Item~\textit{(\ref{item:main:monotnic:c_lesseq0})} follows analogously. 
 \end{proof}

\section{Small \texorpdfstring{$d$}{Lg} regime}\label{sec:small_d}

% \begin{figure}
% \centering
% \begin{subfigure}{0.45\textwidth}
%          \centering
%          \includegraphics[width=\textwidth]{fig/subsolution.pdf}
%         %\caption{Subsolution $v$ as defined by~\eqref{eqn:wave_propagation:subsolution}. }
%         %\label{3:fig:wave_propagation:subsolution}
%      \end{subfigure}
%      \begin{subfigure}{0.45\textwidth}
%          \centering
%          \includegraphics[width=\textwidth]{fig/subsolution2.pdf}
%       % \caption{Supersolution $v$ as defined by~\eqref{eqn:wave_propagation:supersolution}.} 
%       \end{subfigure}
     
%      \caption{These images show how a sub-solution  with negative speed causes the wave to move to the left. At $t=0$ we have $\Psi(\xi)\leq \Phi(\xi)$. The wave profile $\Phi$ must also travel with negative speed to ensure that the correct ordering is preserved for $t>0$. \note{VS:Do we want this? Nothing refers to the figure (no label either).}}
% \end{figure}\label{3:fig:wave_propagation:subsolution}

The main goal of this section is to establish Theorem~\ref{prop:main_results:c_neg:d_small} %\ref{prop:main_results:c_pos:d_small}
by constructing appropriate  sub-solutions. In light of the a priori upper bounds for the regions $\mathcal{D}^-$ and $\mathcal{D}^+$, we consider this the `small $d$'-regime. The geometric interpretation that we develop here will allow us to  find explicit characterizations for these sets in {\S}\ref{sec:cubic_proof}
in the special case that $g$ is the standard cubic nonlinearity~\eqref{eqn:intro:cubic}. %, and numerical observations suggest that this approach yields more accurate results for parameters $(a,d)$ close to the pinning region.

Following the approach developed by Keener~\cite{keener1987propagation},
we fix $a\in (0,1)$ and $A\in (a,1)$
and set out to construct a smooth but steep sub-solution $\Psi$ that connects zero to $A$, see Fig.~\ref{fig:intro:sub_sol}.
We first show that the corresponding sub-solution residual 
%$\mathcal{J}^-$ is 
can be controlled by the expression
\begin{equation}\label{eqn:small_d:J-}
   \mathcal{J}^-(a, d, A,  k) := \max_{v\in [0,A]} \big(d(k+1)v -  dkA -g (v;a)  \big), %\max_{v\in [0,A]} \left(-g(v;a) - d(kA - (k+1)v) \right).
\end{equation}
which forces $c(a,d,k) < 0$ whenever it is negative.

\begin{lemma}\label{lemma:small_d:subsolution_construction}
Consider the setting of Theorem~\ref{prop:main_results:c_neg:d_small}. Pick $(a, d)\in \mathcal{H}$ and
suppose that there exists $A\in (a,1)$ with the property $\mathcal{J}^-(a, d, A,  k)<0$.  Then we have $c(a,d, k) < 0$. 
 \end{lemma}
\begin{proof}
The strict inequality $A<1$ allows us  to choose $\xi_0$ and $\xi_1$ so that 
\begin{equation*}
    \Phi(\xi_0)>A \qquad \text{and} \qquad  0<\xi_1-\xi_0<1.  
\end{equation*}
We use $\xi_0$ and $\xi_1$ to define a smooth function $\Psi:\R\to\R$ that satisfies
\begin{equation}\label{eqn:wave_propagation:subsolution}
    \Psi(\xi) = \begin{cases}
     0, \quad   & \xi\leq \xi_0, \\
     A, \quad & \xi \geq \xi_1,
    \end{cases}
\end{equation}
and is strictly increasing for  $\xi_0<\xi<\xi_1$.   We will show that there exists $\overline{c}<0$ such that 
\begin{equation*}
    \mathcal{I}_{a,d, k}(\overline{c}, \Psi) \leq 0, 
\end{equation*}
which yields $c(a,d, k)<0$ using Corollary~\ref{lem:comparison-MFDE}. 

To this end, we define 
\begin{equation}\label{eqn:wave_prop:c_pos:eps}
    \epsilon:=\min_{v\in [0, A]} g(v;a) + d(kA - (k+1)v) > 0,
\end{equation}
which allows us to choose $\overline{c}<0$ in such a way that
\begin{equation*}
    |\overline{c}\Psi'(\xi)| \leq \epsilon.
\end{equation*}
For $\xi< \xi_0$ we have $\Psi'(\xi) = 0$ and
\begin{equation*}
    \mathcal{I}_{a,d, k}(\overline{c}, \Psi) \leq  - d k \Psi (\xi+1)  \leq 0.
\end{equation*}
If $\xi>\xi_1$ we again have $\Psi'(\xi) = 0 $ and
\begin{align*}
     \mathcal{I}_{a,d, k}(\overline{c}, \Psi) \leq  dA  - g(A;a)  < 0.
\end{align*}
For $\xi\in [\xi_0, \xi_1]$ we have $\Psi(\xi)\in [0, A]$, $\Psi(\xi-1)=0$
and $\Psi(\xi+1)=A$, which gives
\begin{align*}
\mathcal{I}_{a,d, k}(\overline{c}, \Psi) \leq  \epsilon - d\big(kA - (k+1)\Psi(\xi) \big) - g(\Psi(\xi);a) \leq 0,
\end{align*}
as desired.
\end{proof}

A key ingredient towards establishing Theorem~\ref{prop:main_results:c_neg:d_small}
is to find an explicit relation between the set $\mathcal{D}^-$ and the expression $\mathcal{J}^-$. This is achieved in the following result,
using a geometric construction that is illustrated in
Fig.~\ref{3:fig:small_d:J}.

%In our following result, we give an explicit relation between the set $\mathcal{D}^-$ and the expression $\mathcal{J}^-$. This alternative characterization of the set $\mathcal{D}^-$ is the one that we use in the proof of Theorem~\ref{prop:main_results:c_neg:d_small} to show that $c<0$ in $ \mathcal{D}^-$. 

\begin{figure}[t!]
\centering
        \includegraphics[width = .8\textwidth]{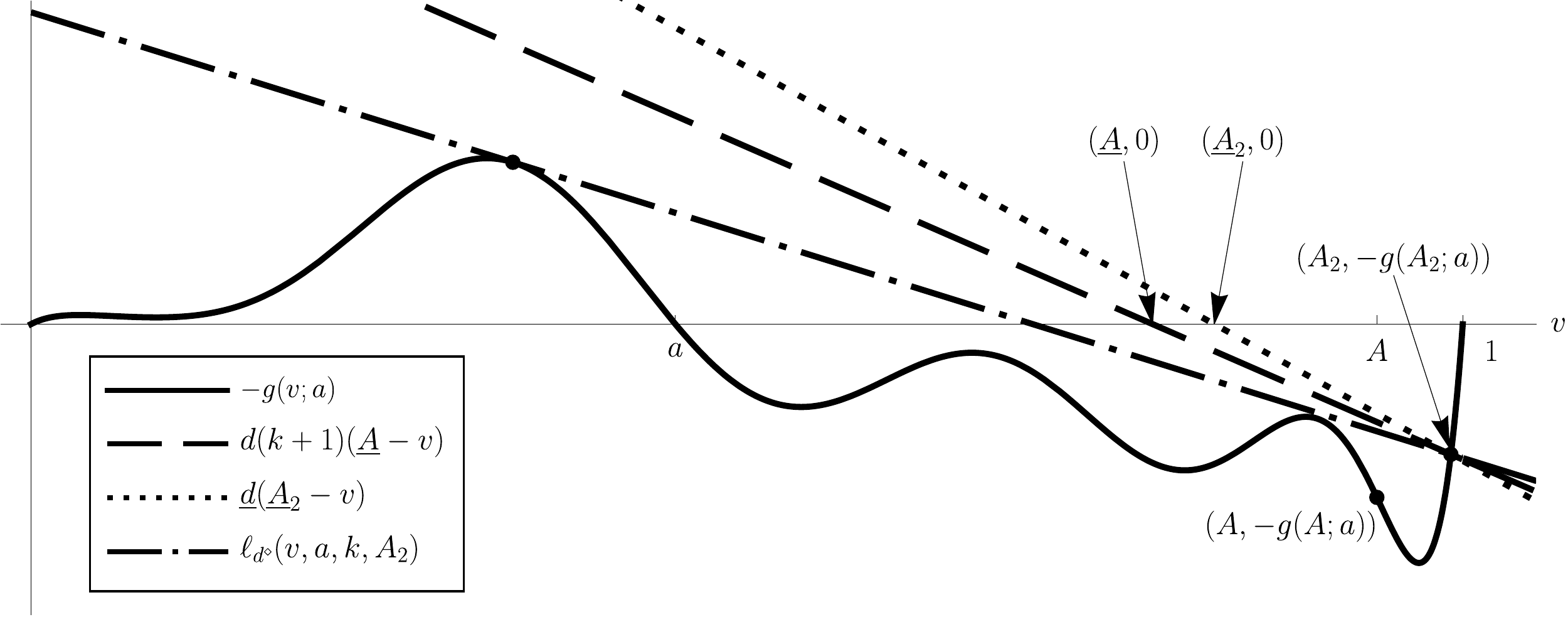}
   
\caption{
% The left panel provides a geometric proof for Proposition~\ref{prop:small_d:J-}. The right panel depicts the function $g_1$ defined by~\eqref{eqn:small_d:g1} which has a unique local maximum on $(0, A^*)$ at $v=a$. \note{VS: check caption}
The geometric intuition behind the ideas and quantities from the proof for Proposition~\ref{prop:small_d:J-}. Note that $d^\diamond(k+1)<d(k+1)<\underline{d}$.
}
\label{3:fig:small_d:J} 
\end{figure}

% \begin{figure}[t!]
% \centering
% \begin{subfigure}{0.45\textwidth}
%      \begin{center}
%         \includegraphics[width = \textwidth]{fig/d_diamond_neg_non_empty.pdf}
%     \end{center}
% \end{subfigure}
%   \begin{subfigure}{0.45\textwidth}
%      \begin{center}
%         \includegraphics[width = \textwidth]{fig/g1.pdf}
%     \end{center}
% \end{subfigure}
% \caption{The left panel provides a geometric proof for Proposition~\ref{prop:small_d:J-}. The right panel depicts the function $g_1$ defined by~\eqref{eqn:small_d:g1} which has a unique local maximum on $(0, A^*)$ at $v=a$. }
% \label{3:fig:small_d:J} 
% \end{figure}

\begin{proposition}\label{prop:small_d:J-} Consider the setting of Theorem~\ref{prop:main_results:c_neg:d_small}. Then the following two statements are equivalent.
\begin{enumerate}[(i)]
    \item We have $(a, d) \in \mathcal{D}^-(k)$.
    \item There exists $A\in (a,1)$ for which  $\mathcal{J}^-(a, d, A,  k) < 0.$
\end{enumerate}
\end{proposition} 

\begin{proof}
Assuming (i), there exists $A\in (a,1)$ so that for all $v\in [0, A]$ we have  $  d(k+1)(A - v) - g(A;a) \geq -g(v;a)$. Since also  $d A< {g(A;a)}$,
this implies that
\begin{align*}
    \mathcal{J}^-(a, d, A, k) &=  d(k+1)(v-A) + dA -g (v;a) \\
    &< d(k+1)(v-A)  - g(A;a) - g(v;a) \leq 0.
\end{align*}

To establish the opposite inclusion, we assume (ii)
%that $\mathcal{J}^-(a, d, A,g, k) < 0$ for some $A\in (a,1)$, and we 
and
 write $\underline{A} = \frac{k}{k+1} A$.
The line through $(\underline{A}, 0)$ with  slope $-d(k+1)$ intersects the graph of $-g$ at some point $v = A_2 > A$, see Fig.~\ref{3:fig:small_d:J}. 
We automatically have $d > d^\diamond(A_2;a)$ by definition of $d^\diamond(A_2;a)$ in \eqref{eqn:main:d_diamond:neg}, so it suffices to show that $d < g(A_2;a)/A_2$.

To this end, we write $\underline{A}_2 = \frac{k}{k+1}A_2$ and 
 point out that
the slope of the line connecting the points $(\underline{A}_2,0)$ and $(A_2, -g(A_2;a))$  is given by $\underline{d} = -\frac{g(A_2;a)}{A_2}(k+1)$. Since we have $\underline{A}_2  > \underline{A}$, the inequality $-d(k+1) > -\underline{d}$ must also hold, which immediately implies $d< {g(A_2;a)}/{A_2}$.
\end{proof}

We now continue with two essential observations concerning
the quantities $d^\diamond$ and $\mathcal{J}^-$.
These will allow us to conclude that $\mathcal{D}^-$ is non-empty
and - when \Hgone holds - free of holes.

%A Theorem~\ref{prop:main_results:c_neg:d_small}

%our sub-solutions can be seen as 

%The set $\mathcal{D}^-$ is directly related to a sub-solution residual for `step-like' functions. Namely, for fixed $a\in (0,1)$ and $A\in (a,1)$ we follow Keener's approach from~\cite{keener1987propagation} and construct a smooth, steep function $\Psi$ that connects zero-level with the $A$-level, see Fig.~\ref{3:fig:wave_propagation:subsolution}. For this construction, we employ the theory set up in \S\ref{sec:cp} which results in the explicit expression $\mathcal{J}^-$, given below by~\eqref{eqn:small_d:J-}. Provided that the sign of $\mathcal{J}^-$ is negative, we show in Lemma~\ref{lemma:small_d:subsolution_construction} that   our wave-profile $\Psi$ travels with strictly negative speed $c$.

%To define $\mathcal{J}^-$, we pick

\begin{lemma}\label{lemma:small_d:D_nonempty}
Consider the setting of Theorem~\ref{prop:main_results:c_neg:d_small}. There exist $0 < A^*< 1$ and $a_- >0 $ such that for all $a\in (0, a_-)$ we have
\begin{equation}\label{eqn:neg_speed:D_nonempty}
    d^\diamond(A^*; a) < \dfrac{g(A^*;a)}{A^*}.
\end{equation}
\end{lemma}

\begin{proof}
We first note that it suffices to find $0 < A^* < 1$
for which~\eqref{eqn:neg_speed:D_nonempty}
holds at $a = 0$. Indeed, $g(A;a)/A$ and $d^\diamond(A;a)$ are continuous 
with respect to both their arguments, the latter since it is the supremum of a 
difference quotient on a
compact interval that depends continuously on these arguments. 
%In addition, $g(A;a)/A$ is also a continuous function in both variables.  
%Therefore, the inequality~\eqref{eqn:neg_speed:D_nonempty} can be extended to some  neighbourhood of $a=0$.

%We first remark that $d^\diamond(A;a)$ is a continuous function in both parameters as it %can be seen as the supremum of a function (namely a difference quotient) on a
%compact interval that depends continuously on the variables $a$ and $A$. 
%At $a=0$ the function $v\mapsto -g(v;a)$  is strictly negative on $(0,1)$. Therefore, at $A=1$ we have $d^\diamond(A, a) = 0$. We claim that there exists $A^*< 1$ such that~\eqref{eqn:neg_speed:D_nonempty} holds for $a=0$. 
To show this, we define
\begin{equation*}
    \epsilon = \dfrac{k}{8(k+1)}
\end{equation*}
  and use the fact that $g'(1;0)<0$ to pick $A_* \in (1/2, 1)$ in such a way that
  %we can find $A_\epsilon \in (1/2, 1)$ such that for all $A\in [A_\epsilon, 1)$, 
  the line connecting the points $(A_*, -g(A_*;0))$ and $(\epsilon, 0)$ is above the graph of $-g(\cdot;0)$ on $[0, A]$. The slope of this line is given by
$   d_\epsilon = -\frac{g(A_*;0)}{A_*-\epsilon}$, 
which 
using the fact that
$kA - (k+1)\epsilon >0$
implies 
%for all $A\in [A_\epsilon, 1)$ we have
\begin{equation*}
    d^\diamond(A_*;0) \leq \frac{d_\epsilon}{k+1} = \dfrac{g(A_*;0)}{(k+1)(A_*-\epsilon)} <  \frac{g(A_*;0)}{A_*},
\end{equation*}
as desired.
%In order to have $d^\diamond(A;0) < g(A;0)/A$, it is enough to prove that $kA - (k+1)\epsilon >0$. However, this follows from the choice of parameters $\epsilon$ and our apriori requirement that $A>1/2$. Therefore, we can set $A^* = A_\epsilon$. 
\end{proof}

%\subsection{Sub-solution residual \texorpdfstring{$\mathcal{J}^-$}{Lg}}

%In the following lemma we construct an explicit subsolution $\Psi$ for the LDE~\eqref{eqn:main:LDE}. 

\begin{lemma}\label{lemma:small_d:monotonic}
Consider the setting of Theorem~\ref{prop:main_results:c_neg:d_small} and assume furthermore that \Hgone is satisfied. Pick $(a, d)\in \overline{\mathcal{D}^-}\cap \mathcal{H}$ and $\underline{a} \in (0,a)$. Then we have $(\underline{a},d) \in \mathcal{D}^-$.
\end{lemma}
\begin{proof}
By Proposition~\ref{prop:small_d:J-} we have
\begin{equation*}
    \mathcal{D}^-(k) = \bigcup_{a\in (0,1)} \bigcup_{A\in (a,1)} \left\{(a,d)\in \mathcal{H}: \mathcal{J}^-(a, d, A, k) < 0\right\}.
\end{equation*}
Let us now pick $(a,d)\in \overline{\mathcal{D}^-} \cap \mathcal{H}$.
The continuity of $\mathcal{J}^-$ with respect to $A$
implies that
$$\mathcal{J}^-(a,d, A,  k) \leq 0 $$ holds for some $A\in (a,1]$.
Note that $A =a$ is excluded here since $\mathcal{J}^- = da$ for $A = a$.
%together with the identity  $\mathcal{J}^-(a,d,a,g,k)=
%this
%implies
%\begin{equation*}
%    \overline{\mathcal{D}^-(g,k)}\cap \mathcal{H} \subseteq  \bigcup_{a\in (0,1)} \bigcup_{A\in (a,1)} \left\{(a,d)\in \mathcal{H}: \mathcal{J}^-(a, d, A, g, k) \leq 0\right\}.
%\end{equation*}
%Therefore, for $(a,d)\in \overline{\mathcal{D}^-} \cap \mathcal{H}$ the inequality $\mathcal{J}^-(a,d, A, g, k) \leq 0 $ holds for some $A\in (a,1]$.

For $\overline{a} < a$,  the assumption \Hgone implies that 
\begin{equation*}
    g(v;\overline{a}) + d(kA - (k+1)v) >   g(v;a) + d(kA - (k+1)v) \geq 0, \quad v\in (0, A], 
\end{equation*}
while for $v=0$ we have $g(0;\overline{a}) + dkA  = dkA >0$. Therefore,  $\mathcal{J}^-(\overline{a}, d, A,  k)<0$ holds. 
\end{proof}

In the following lemma we explore how  the extra condition \Hgtwo leads to the explicit inclusion $(a, g'(a;a)/(k+1))\in \mathcal{D}^-(k)$ for $a\approx 0$.
In particular, translated into the language of  $d^\diamond$, the first item  implies that $g'(a;a)/(k+1) > d^\diamond(A;a)$ for any $A\in (a,1)$. The second item then ensures that there exists $A$ such that $d^\diamond(A;a) < g(A;a)/A$.

\begin{lemma}\label{lemma:proof_Prop22:aux1}
Pick $k>0$ and assume that the nonlinearity $g$ satisfies \Hg and \Hgtwo. 
Then there exist a constant $\delta_a \in(0,1)$ such that for all $a\in (0, \delta_a)$ we have
\begin{enumerate}[(i)]
    \item  $ g''(a;a) >0 $,
    \item\label{item:neg:speed:D-} $0< \dfrac{g'(a;a)}{k+1} < \max_{A\in [a+\frac{2a}{k}, 1]}\dfrac{g(A;a)}{A}$.
\end{enumerate}
   %g''(a;a) < 0 \qquad  \text{ and } \qquad  \dfrac{g'(a;a)}{k+1} &\leq \max_{A\in [(k+1)(1-a), 1]}\dfrac{-g(1-A;a)}{kA} , \quad \text{for all} \quad a\in(a^*_+,1) \label{eqn:small_d:g':2}

\end{lemma}
\begin{proof} 
Due to the assumption \Hg, the function 
$$ D_-(a) =  \max_{A\in [a+\frac{2a}{k}, 1)}\dfrac{g(A;a)}{A} $$  is well defined, positive and decreasing  for $a\leq 1/(k+2)$. In particular, we have $D_-(0)>0$.
The assumption \Hgtwo ensures that $0 = g'(0,0)< D_-(0) (k+1)$. Now the existence of $\delta_a$ and the claim of item~\textit{(\ref{item:neg:speed:D-})} follow from the continuity properties of the nonlinearity $g$ and the function $D_-$. The inequality $g''(a;a)>0$ for $a\in (0, \delta_a)$ follows again from \Hgtwo by reducing $\delta_a$ if necessary. 
\end{proof}

\begin{proof}[Proof of Theorem~\ref{prop:main_results:c_neg:d_small}]
Item \textit{(\ref{item:main:D_nonempty})} follows directly from Lemma~\ref{lemma:small_d:D_nonempty}. 
To show item \textit{(\ref{item:main:c_neg_in_D})}, we first employ Proposition~\ref{prop:small_d:J-} in combination with Lemma~\ref{lemma:small_d:subsolution_construction} to conclude that $c<0$ in $\mathcal{D}^-$.
The result $c>0$ in $\mathcal{D}^+$ now follows from Lemma~\ref{prop:main:symmetry}.

Item \textit{(\ref{item:main:monotonicity_in_D})} for $\mathcal{D}^-$ is a direct consequence  of Lemma~\ref{lemma:small_d:monotonic} and Proposition~\ref{prop:small_d:J-}. To show the equivalent result for $\mathcal{D}^+$, we assume that $(a,d)\in \mathcal{D}^+$ and take $\overline{a}>a$. 
In view of the definition~\eqref{eqn:main:Dplus} for $\mathcal{D}^+$,  
we have $(1-a, dk)\in \mathcal{D}^-(\tilde{g}, 1/k)$ and consequently $(1-\overline{a},dk )\in \mathcal{D}^-(\tilde{g}, 1/k)$. 
%By definition~\eqref{eqn:main:Dplus} of $\mathcal{D}^+$ we have  
In particular, this implies $(\overline{a}, d)\in \mathcal{D}^+$. 

To show item \textit{(\ref{item:main:curve_in_D})},
we take $\delta_a$ from Lemma~\ref{lemma:proof_Prop22:aux1} and implicitly define the quantity $A^*\geq a+2a/k$ by writing
\begin{equation*}
    \dfrac{g(A^*;a)}{A^*}= \max_{A\in [a+\frac{2a}{k}, 1]}\dfrac{g(A;a)}{A} .
\end{equation*}
For $d = d (a) := g'(a;a)/(k+1)$, the function $ {g}_1(v;a)$ defined by
\begin{equation}\label{eqn:small_d:g1}
{g}_1(v;a) = -g(v;a) - d(kA^* - (k+1)v)
\end{equation}
%has the following properties: 
satisfies $g_1(0;a) < 0$ and ${g}_1(A^*;a)  = -g(A^*;a) + d^* A^* < 0$ by Lemma~\ref{lemma:proof_Prop22:aux1}. Moreover, its unique local maximum or inflection point is achieved at %the point 
$v=a$ since
\begin{equation*}
{g}_1'(a;a)  = -g'(a;a) + (k+1)d =0 \qquad \hbox{ and } \qquad {g}_1''(a;a) = -g''(a;a)< 0. 
\end{equation*}
The value in the local minimum is $ {g}_1(a;a) = -d(kA^* - (k+1)a) \geq -ad/k>0 $.
%\begin{itemize}
 %   \item  
 %   $g_1(0;a) < 0$;
 %   \item ${g}_1(A^*;a)  = -g(A^*;a) + d^* A^* < 0$ by Lemma~\ref{lemma:proof_Prop22:aux1};
 %   \item its unique local maximum or inflection point is achieved at the point $v=a$ since ${g}_1'(a;a)  = -g'(a;a) + (k+1)d =0$ and ${g}_1''(a;a) = -g''(a;a)< 0$. The value in the local minimum is $ {g}_1(a;a) = -d(kA^* - (k+1)a) \geq -ad/k>0 $.
%\end{itemize}
Therefore, we have $\mathcal{J}^-(a, d, A^*, g, k) < 0$, which implies $c(a, d, k)<0$ by Lemma~\ref{lemma:small_d:subsolution_construction} and Proposition~\ref{prop:small_d:J-}.

To show that $(a, d(a))\in \mathcal{D}^+(k)$ for $a\approx 1$, where $d(a) = g'(a;a)/(k+1)$ it suffices to show that $(\tilde{a}, \tilde{d}(\tilde{a})) \in \mathcal{D}^-_{\tilde{g}}( 1/k)$. We note now that the nonlinearity $\tilde{g}$ also satisfies \Hg and \Hgtwo. Therefore, by repeating the procedure above, we have
\begin{equation*}
   \tilde{d}(\tilde{a}) = \frac{\tilde{g}'(\tilde{a}; \tilde{a})}{\tilde{k}+1 } \in \mathcal{D}^-_{\tilde{g}}( \tilde{k}).
\end{equation*}
Translating back to our original coordinates we obtain
\begin{equation*}
    d(a) = \dfrac{\tilde{d}(\tilde{a})}{k} = \dfrac{g'(a;a)}{k+1} \in \mathcal{D}^+_g( k).
\end{equation*}
\end{proof}

%%%%%%%%%%%%%%%%%%%%%%%%%%%%%%%%%%%%%%%%%%%%%%%%%%%%%%%%%%%%%%%%%%%%%%%%%%%%%%%%%%%%%%%%%%%%%%

\section{Large $d$ regime}\label{sec:d_large}
In this section, we prove Theorem~\ref{prop:main_results:c_neg:d_large} and show that the profile~\eqref{eqn:wave_ansatz} satisfies $c<0$ if $d \gg 0$. We achieve this
by applying Corollary~\ref{lem:comparison-MFDE} to a second class of sub-solutions $\Psi$,
which have milder growth than those from
{\S}\ref{sec:small_d}.
 Using the notation \eqref{eqn:main:k_laplacian} for the discrete diffusion-advection operator we use
\begin{align}\label{eq:d-large-delta-k}
\Delta_k[\Phi](\xi) := \Phi(\xi-1) - (k+1) \Phi(\xi) + k \Phi(\xi+1),
\end{align}
which allows us to rewrite~\eqref{eqn:I_operator} as %\todo[color=green]{VS: Add something about the second order difference equation. $\Delta_k[\Phi] = m$, $m \in \mathbb{R}$. What is the difference between this equation on $\mathbb{R}$ and $\mathbb{Z}$?}
\[
\mathcal{I}_{a,d, k}[c,\Phi] = -c \Phi' - d \Delta_k [\Phi]-g(\Phi,a). 
\]
%Firstly, we discuss certain properties of the discrete $k$-Laplace operator $\Delta_k$ and its connection to the classical second derivative. The standard discrete Laplace operator with $k=1$ 
%The operator
%\[
%\Delta_1[\Phi](\xi) = \Phi(\xi+1) - 2 \Phi(\xi) + \Phi(\xi-1)
%\]
%can be understood as a discrete version of the second derivative of a smooth function.
%Moreover,  $\Phi''(\xi) \gtrless 0$ for all $\xi \in \mathbb{R}$ implies $\Delta_1[\Phi](\xi) \gtrless 0$ for all $\xi \in \mathbb{R}$. We remark that  the reverse implication does not hold. For example, the function $\chi(\xi) = 4 \xi^4-\xi^2$ satisfies $\Delta_1 [\chi](\xi) = 6(1+8\xi^2) >0$ but $\chi''(0)=-2 <0$. 

%Nevertheless, owing to the similarity we call a function $\Phi$ \textit{convex (resp. concave) with respect to $\Delta_k$} if $\Delta_k[\Phi] > 0$ (resp. $<0$). Choosing $k>1$ means, that the curvature of the profile in the positive direction from the reference point influences the value of $\Delta_k$ more significantly than the one in the negative direction. This is heavily exploited in the further text since we construct a bounded and a strictly convex function with respect to $\Delta_k$ for $k>1$. Such construction is not possible if $k=1$. 

Since the term $\Delta_k [\Psi]$ appears with a negative sign in the residual expression $\mathcal{I}_{a,d,k}$, our goal is to construct a simple subsolution $\Psi$ with a strictly positive sign of $\Delta_k [\Psi]$. By choosing $d>0$ large enough the contribution of  $d \Delta_k [\Psi]$ can then be used  to overcome the impact of the nonlinearity $g$.

We approach the construction of the profile $\Psi$ in a stepwise fashion.
%number of steps. 
First of all, for $l > 1$ and $A>0$ we define the function $\kappa_{l,A}:\R\to \R$ by writing
\begin{align}\label{eq:kappa}
\kappa_{l,A}(\xi) =A\left(1-l^{-\xi}\right).
\end{align}
One can directly compute that $\kappa_{l,A}$ is strictly increasing with
\begin{equation}\label{eqn:large_d:delta_kappa}
\Delta_k[\kappa_{l,A}] = Al^{-\xi}(k-l)\left(1-\tfrac{1}{l}\right).
\end{equation} 
We therefore have $\Delta_k[\kappa_{l,A}]>0$ if and only if $l\in (1, k)$. However, both $\kappa_{l,A}$ and its derivative $\kappa'_{l,A}(\xi)$ are unbounded as $\xi \to -\infty$, which prevents us from controlling the sign of $\mathcal{I}_{a,d,k}$.

To circumvent this drawback we define a modified profile
\begin{align}\label{eq:d-large-psi}
\Psi_{l,A}(\xi) = \begin{cases}
A\left(1-l-\frac{\log(l)}{3}\right), & \xi \le -1-\frac{1}{l}, \\
A \, p_{l}(\xi), & -1-\frac{1}{l} < \xi \le -1, \\
\kappa_{l,A}(\xi), & \xi > -1 \\
\end{cases}
\end{align}
in which $p_{l}(\xi)$ is the cubic polynomial given by
\[
p_{l}(\xi) = \frac{1}{3} \log(l) \big(l \xi + (1+l)\big)^3 +1-l- \frac{\log(l)}{3}.
\]
It can be verified by a direct computation that the profile $\Psi_{l,A}$ is $C^1$-smooth. The profile $\Psi$ has a bounded derivative 
\begin{align}\label{eq:d-large-psi-derivative-estimate}
0 \le \Psi'_{l,A}(\xi) \le \Psi'_{l,A}(-1) = \kappa'_{l,A}(-1) = A l \log(l), \quad \xi \in \mathbb{R}.
\end{align}
%and it is convex with respect to $\Delta_k$ provided $1<l<k$. Since the profile~\eqref{eq:d-large-psi} contains the
On account of the cubic polynomial $p_l$ 
it is rather cumbersome to provide
a precise expression for $\Delta_k[\Psi_{l,A}]$.
%is rather cumbersome. 
We rather point out several key qualitative features, %of $\Delta_k[\Psi_{l,A}]$ 
which guarantee that this expression is non-negative and are used later in the proof of Lemma~\ref{lem:d-large-lower-solution}. %These properties yield almost complete qualitative description of $\Delta_k[\Psi_{l,A}]$. 
\begin{lemma}\label{lem:d-large-psi-delta}
Pick $k>1$ together with $l \in (1,k)$ and $A\in(0,1)$.
%and let $\Psi_{l,A}$ be defined by~\eqref{eq:d-large-psi}. 
Then
%the following properties hold.
\begin{enumerate}[(i)]
    \item \label{it:d-large-psi-delta-i} 
        $\Delta_k[\Psi_{l,A}](\xi) = 0$ for $\xi \le -2-\tfrac{1}{l}$, 
    \item \label{it:d-large-psi-delta-ii}
        $\Delta_k[\Psi_{l,A}](\xi)$ is strictly increasing for $\xi \in \left(-2-\tfrac{1}{l}, -1-\tfrac{1}{l}\right]$,
    \item \label{it:d-large-psi-delta-iii}
        $\Delta_k[\Psi_{l,A}](\xi)$ is concave for $\xi \in \left(-1-\tfrac{1}{l},-1 \right)$, 
    \item \label{it:d-large-psi-delta-iv}
        $\Delta_k[\Psi_{l,A}](\xi) > \Delta_k[\kappa_{l,A}](\xi) $ %=  A l^{-\xi}(k-l)(1-\tfrac{1}{l})$ 
        for $\xi \in [-1,0)$, 
    \item \label{it:d-large-psi-delta-v}
        $\Delta_k[\Psi_{l,A}](\xi) = 
        \Delta_k[\kappa_{l,A}](\xi)$
        %A l^{-\xi}(k-l)(1-\tfrac{1}{l})$ 
        for $\xi \ge 0$. 
\end{enumerate}
\end{lemma}
\begin{proof}
The claims~\eqref{it:d-large-psi-delta-i} and~\eqref{it:d-large-psi-delta-v} follow directly from the definitions~\eqref{eq:d-large-delta-k} and~\eqref{eq:d-large-psi}. To establish ~\eqref{it:d-large-psi-delta-ii},
we pick $\xi \in \left(-2-\tfrac{1}{l}, -1-\tfrac{1}{l}\right]$
and note that  $\Psi_{l,A}(\xi-1) = \Psi_{l,A}(\xi)
= A\left(1-l-\frac{\log(l)}{3}\right)$,
%, i.e., a constant with respect to $\xi$, and
while $\Psi_{l,A}(\xi+1)$ is strictly increasing in $\xi$. %This proves~\eqref{it:d-large-psi-delta-ii}. 

Turning to (iii), we pick
$\xi \in \left(-1-\tfrac{1}{l},-1 \right)$
and compute
%Note that $\Psi_{l,A}$ is analytical in this interval. We directly compute
\[
\frac{\mathrm{d}^2}{\mathrm{d}\xi^2}  \Delta_k[\Psi_{l,A}](\xi) = -A l^{-1-\xi} \log(l) \big(2(1+k)l^{3 + \xi}(1+l+l \xi)+k \log(l)\big) < 0,
\]
since $1+l+l \xi > 1+l+l(-1-\tfrac{1}{l}) =  0$ and $l>1$. 
%This proves~\eqref{it:d-large-psi-delta-iii}. 
For the remaining claim~\eqref{it:d-large-psi-delta-iv},
we take $\xi \in [-1,0)$
%The equality in~\eqref{it:d-large-psi-delta-iv} follows directly from~\eqref{eq:d-large-kappa-delta}.
%We %next 
and
observe that $\Psi_{l,A}(\xi-1) > \kappa_{l,A}(\xi-1)$,  while on the other hand
$\Psi_{l,A}(\xi) = \kappa_{l,A}(\xi)$
and
$\Psi_{l,A}(\xi+1) = \kappa_{l,A}(\xi+1)$. %This directly yields $\Delta_k[\Psi_{l,A}](\xi) > \Delta_k[\kappa_{l,A}](\xi)$ and 
%proves the rest of
%establishes~\eqref{it:d-large-psi-delta-iv}. 
\end{proof}

At this point it is convenient
to extend the definition~\eqref{eq:d-large-d-star} of $d^*(a, k)$
by writing
\begin{align}
\label{eq:d-large-d-star-l}
    d^*(a;l,A, k) = \frac{1}{(k-l)\left(1-\tfrac{1}{l}\right)} \max_{s \in \left[1-\tfrac{a}{A},1\right]} \frac{-g(A(1-s);a)}{A s}, 
\end{align}
which reduces to~\eqref{eq:d-large-d-star}
in the special case
$A = 1$ and $l = \sqrt{k}$, which maximizes the value of $d^*$ when the other two parameters ($a$ and $k$) are fixed.
The following lemma 
shows that $\Psi_{l,A}$ can act as a sub-solution with negative speed
when $d > d^*(a;l,A,k)$.
%solution in the sense of Corollary~\ref{lem:comparison-MFDE}. 

%In order to prove Theorem~\ref{prop:main_results:c_neg:d_large} we update our notation. Let us define 
%\begin{align}
%\label{eq:d-large-d-star-l}
%    d^*(a;l,A, k) = \frac{1}{(k-l)\left(1-\tfrac{1}{l}\right)} \max_{s \in \left[1-\tfrac{a}{A},1\right]} \frac{-g(A(1-s);a)}{A s}. 
%\end{align}
%The use of the symbol $d^*$ clearly coincides with~\eqref{eq:d-large-d-star} but this can be justified; we now view $d^*$ from~\eqref{eq:d-large-d-star} as a special case of~\eqref{eq:d-large-d-star-l}. Indeed, 
%applying a direct substitution $l \mapsto \sqrt{k}$ and $A \mapsto 1$ we obtain
%\[
%d^*(a;\sqrt{k},1) = \frac{1}{(k-\sqrt{k})\left(1-\tfrac{1}{\sqrt{k}}\right)} \max_{s \in \left[1-a,1\right]} \frac{-g(1-s;a)}{s},
%\]
%and clearly $d^*(a;\sqrt{k},1)=d^*(a)$ in which the former is defined by~\eqref{eq:d-large-d-star-l} and the latter is defined  by~\eqref{eq:d-large-d-star}. We shall now refer to $d^*$ as in~\eqref{eq:d-large-d-star-l} unless stated otherwise. 

%The following lemma states under which conditions the profile $\Psi_{l,A}$ can act as a lower solution in the sense of Corollary~\ref{lem:comparison-MFDE}. 

\begin{lemma}\label{lem:d-large-lower-solution}
Assume that \Hg holds and pick $a \in (0,1)$, $A \in [a,1)$, $k>1$ and $l \in (1,k)$.
%be given. 
Then for any $d > d^*(a;l,A,k)$ there exists
$\overline{c} < 0$ so that
%\begin{equation}
    $\mathcal{I}_{a,d,k}[\overline{c},\Psi_{l,A}](\xi) \le 0$
    for all $\xi \in \mathbb{R}$.
%\end{equation}
%
%If $d> d^*(a;l,A)$ in which $d^*$ is defined by~\eqref{eq:d-large-d-star-l} then there exists $\bar{c}<0$ such that $\Psi_{l,A}$ is a lower solution with speed $\bar{c}$. 
\end{lemma}
\begin{proof}
The condition $d>d^*(a;l,A,k)$ ensures the existence of $\varepsilon >0$ such that 
\begin{align}\label{eq:d-large-d-l}
 d=d^*(a;l,A,k)+\frac{ \varepsilon}{(k-l)\left(1-\tfrac{1}{l}\right)}.
\end{align}
Note that Lemma~\ref{lem:d-large-psi-delta}
allows us to define the positive constant
\begin{align}\label{eq:d-large-C}
C := \inf_{\xi \in \left(-1-\tfrac{1}{l},0\right)} \Delta_k[\Psi_{l,A}](\xi) > 0,
\end{align}
enabling us to choose
 $\overline{c} < 0$ in such a way that
\begin{align}
\bar{c} &>
\max \left\{
-\frac{\varepsilon}{\log(l)},
%\label{eq:d-large-cnd-c-1} \\
%\bar{c} &> 
 -\frac{ \varepsilon C}{A(k-l)\left(1-\tfrac{1}{l}\right)l \log(l)} \right\}.
 \label{eq:d-large-cnd-c-2}
\end{align}
%in which 
%\begin{align}\label{eq:d-large-C}
%C := \inf_{\xi \in \left(-1-\tfrac{1}{l},0\right)} \Delta_k[\Psi_{l,A}](\xi).
%\end{align}
%Note that Lemma~\ref{lem:d-large-psi-delta} ensures $C>0$.
We now proceed to show 
that $\mathcal{I}_{a,d, k}[\overline{c},\Psi_{l,A}](\xi) \le 0$ for all $\xi \in \mathbb{R}$ by considering three separate cases. % for $\xi$.

%We need to show that $\mathcal{I}_{a,d,g, k}[\overline{c},\Psi_{l,A}](\xi) \le 0$ for all $\xi \in \mathbb{R}$. We proceed to separate the claim into three cases, namely $\xi \le -1-\tfrac{1}{l}$, $\xi \in (-1-\tfrac{1}{l},0]$ and $\xi >0$. 

\textbf{Case 1:} For $\xi \le -1-\tfrac{1}{l}$, 
we have
$\Psi'_{l,A}(\xi) = 0$ and $\Delta_k[\Psi_{l,A}](\xi) \ge 0$
by items~\eqref{it:d-large-psi-delta-i} and~\eqref{it:d-large-psi-delta-ii}
of Lemma~\ref{lem:d-large-psi-delta}.
Using \Hg, we hence find
\begin{align*}
\mathcal{I}_{a,d, k}[\bar{c},\Psi_{l,A}](\xi) &= -\bar{c}\Psi'_{l,A}(\xi) - d \Delta_k[\Psi_{l,A}](\xi) - g(\Psi_{l,A}(\xi);a) \\
&\le - g(\Psi_{l,A}(\xi);a) < 0. 
\end{align*}
%Note that $\Psi'_{l,A}(\xi) = 0$ and that $\Delta_k[\Psi_{l,A}](\xi) \ge 0$ (Lemma~\ref{lem:d-large-psi-delta},~\eqref{it:d-large-psi-delta-i} and~\eqref{it:d-large-psi-delta-ii}). The last inequality follows from \Hg. 

\textbf{Case 2:} For $\xi \in (-1-\tfrac{1}{l},0]$, we compute
\begin{align*}
\mathcal{I}_{a,d, k}[\overline{c},\Psi_{l,A}](\xi) &= -\overline{c} \Psi'_{l,A}(\xi) - d \Delta_k[\Psi_{l,A}](\xi) - g(\Psi_{l,A}(\xi);a) \\
&\le -\overline{c} \Psi'_{l,A}(\xi) - d \Delta_k[\Psi_{l,A}](\xi)\\ 
&\le \frac{\varepsilon C}{A(k-l)\left(1-\tfrac{1}{l}\right)l \log(l)} A l \log(l) - \left(d^*(a;l,A,k) +\frac{ \varepsilon}{(k-l)\left(1-\tfrac{1}{l}\right)}\right) C \\
&= -d^*(a;l,A,k) C < 0. 
\end{align*}
Here the first inequality follows from \Hg,
while the second one uses~\eqref{eq:d-large-cnd-c-2},~\eqref{eq:d-large-psi-derivative-estimate},~\eqref{eq:d-large-d-star-l} and~\eqref{eq:d-large-C}, respectively.

\textbf{Case 3:} For $\xi >0$, we have $\Psi_{l,A}(\xi) = \kappa_{l,A}(\xi)$  by~\eqref{eq:d-large-psi} and $\Delta_k[\Psi_{l,A}](\xi) = \Delta_k[\kappa_{l,A}](\xi)$ by
item~\eqref{it:d-large-psi-delta-v}
of Lemma~\ref{lem:d-large-psi-delta}. We
now use
~\eqref{eq:d-large-cnd-c-2},~\eqref{eqn:large_d:delta_kappa} and~\eqref{eq:kappa}
to
compute
\begin{align*}
\mathcal{I}_{a,d, k}[\overline{c},\Psi_{l,A}](\xi) &= -\overline{c} \Psi'_{l,A}(\xi) - d \Delta_k[\Psi_{l,A}](\xi) - g(\Psi_{l,A}(\xi);a) \\
& \le  \frac{\varepsilon}{\log(l)} A l^{-\xi} \log(l) - d A l^{-\xi} (k-l)\left(1- \tfrac{1}{l}\right) - g(A(1-l^{-\xi});a)  \\
&=  \varepsilon A l^{-\xi} - d A l^{-\xi} (k-l)\left(1- \tfrac{1}{l}\right) - g(A(1-l^{-\xi});a). \\
\end{align*}
%In the inequality, we used~\eqref{eq:d-large-cnd-c-1},~\eqref{eq:d-large-kappa-der},~\eqref{eq:d-large-kappa-delta} and~\eqref{eq:d-large-kappa}, respectively.  
Next, we use the substitution
\[
x:= l^{-\xi}, \quad x \in (0,1) 
\]
together with~\eqref{eq:d-large-d-star-l} to obtain
\begin{align}
\mathcal{I}_{a,d, k}[\overline{c},\Psi_{l,A}](\xi) &\le  \varepsilon A x - \left(d^*(a;l,A,k) +\frac{ \varepsilon}{(k-l)\left(1-\tfrac{1}{l}\right)}\right) A x(k-l)\left(1-\tfrac{1}{l}\right) - g(A(1-x);a) , \nonumber \\
%&= \varepsilon A x - \left(\max_{s \in \left[1-\tfrac{a}{A},1\right]} \frac{-g(A(1-s);a)}{A s}+ \varepsilon \right) A x - g(A(1-x);a) \nonumber \\
&= x \left(-  \max_{s \in \left[1-\tfrac{a}{A},1\right]} \frac{-g(A(1-s);a)}{s} - \frac{g(A(1-x);a)}{x} \right) . \label{eq:d-large-before-2-cases}
\end{align}

Now, if $x \in \left[1-\tfrac{a}{A},1\right)$, then \[
\max_{s \in \left[1-\tfrac{a}{A},1\right]} \frac{-g(A(1-s);a)}{s} \ge - \frac{g(A(1-x);a)}{x}
\] 
and if $x \in \left(0,1-\tfrac{a}{A}\right)$, then $-g(A(1-x);a)<0$. Either way, we deduce from~\eqref{eq:d-large-before-2-cases} that
\[
\mathcal{I}_{a,d, k}[\bar{c},\Psi_{l,A}](\xi) \le 0 , \quad \xi >0,
\]
%
%Combining the three cases above yields 
%\[
%\mathcal{I}_{a,d,g, k}[\bar{c},\Psi_{l,A}](\xi) \le 0 , \quad  \xi \in \mathbb{R}
%\]
%and
which completes the proof. 
\end{proof}

%We are now ready to prove Theorem~\ref{prop:main_results:c_neg:d_large}. 

\begin{proof}[Proof of Theorem~\ref{prop:main_results:c_neg:d_large}]
Substituting $A(1-s) = 1-y$, we
find 
\[
d^*(a;l,A, k) = 
\frac{1}{(k-l)\left(1-\tfrac{1}{l}\right)}
\max_{s \in \left[1-\tfrac{a}{A},1\right]} \frac{-g(A(1-s);a)}{A s} = 
\frac{1}{(k-l)\left(1-\tfrac{1}{l}\right)}
\max_{y \in \left[1-a,1\right]} \frac{-g(1-y;a)}{A-1+y}.
\]
Clearly, $d^*(a;l,A,k)$ is continuously decreasing with respect to $A$ and
\[
\lim_{A \to 1^-} d^*(a;l,A,k) =  d^*(a;l,1,k).
\]
Consequently, whenever $d > d^*(a;\sqrt{k}, 1, k)$,
there exists $\overline{A} < 1$ (close to $1$)
so that $d>d^*(a;\sqrt{k},\overline{A},k)$.
An application of Lemma~\ref{lem:d-large-lower-solution} now concludes the proof. 
%
%Assume that $d>d^*(a;\sqrt{k},1)$ or equivalently $d>d^*(a)$ in which $d^*$ is defined by~\eqref{eq:d-large-d-star}. We show that there exists $\bar{A} \in (a,1)$ such that the assumptions of Lemma~\ref{lem:d-large-lower-solution} are satisfied with $l=\sqrt{k}$, i.e., $d>d^*(a;\sqrt{k},\bar{A})$ (see~\eqref{eq:d-large-d-star-l}). We simultaneously show that~\eqref{eq:d-large-d-star} is a optimal bound for $d$ using our method of the proof via construction of $\Psi_{l,A}$. 

%Note that $d^*(a;l,A)$ can be written as a product of two separated terms
%\[
%d^*(a;l,A) = d_1^*(l) \, d_2^*(a;A),
%\]
%in which 
%\[
% d_1^*(l) = \frac{1}{(k-l)\left(1-\tfrac{1}{l}\right)}, \quad d_2^*(a;A) = \max_{y \in \left[1-\tfrac{a}{A},1\right]} \frac{-g(A(1-y);a)}{A y}.
%\]
%Setting $l=\sqrt{k}$ minimizes $d_1^*(l)$ in the interval $(1,k)$. 

%Next, we apply a substitution $A(1-s) = 1-y$ to $d_2^*$ to obtain
%\[
%d_2^*(a;A) = \max_{s \in \left[1-\tfrac{a}{A},1\right]} \frac{-g(A(1-s);a)}{A s} = \max_{y \in \left[1-a,1\right]} \frac{-g(1-y;a)}{A-1+y}.
%\]
%Clearly, $d_2^*(a;A)$ is continuously decreasing in $A$ and
%\[
%\lim_{A \to 1^-} d^*(a;\sqrt{k},A) =  d^*(a;\sqrt{k},1).
%\] 
%There thus exists $\bar{A} \approx 1$ such that $d>d^*(a;\sqrt{k},\bar{A})$. The application of Lemma~\ref{lem:d-large-lower-solution} concludes the proof. 
\end{proof}

%%%%%%%%%%%%%%%%%%%%%%%%%%%%%%%%%%%%%%%%%%%%%%%%%%%%
\section{Cubic nonlinearity}\label{sec:cubic_proof}

The aim of this section is to prove Proposition \ref{cor:main:cubic:neg},
%and \ref{cor:main:cubic:pos} 
which explicitly describes the region $\mathcal{D}^-$ % and $\mathcal{D}^+$
for the  standard cubic nonlinearity
 \begin{equation}\label{eqn:cubic:def}
     g(v;a) = v(1-v)(v-a).
 \end{equation}
We achieve this
by finding explicit expressions  for the slope $d^\diamond$ defined in \eqref{eqn:main:d_diamond:neg}. %In view of definition \eqref{eqn:main:Dminus}, 
 %for $a\in(0,1)$ we  denote by $\Delta(a)$ the set of admissible parameters $A$, namely
%\begin{equation}
%    A\in \mathcal{A}(a) \iff d^\diamond(a;A) \leq \dfrac{g(A;a)}{A}.
%\end{equation}
The definition of  $d^\diamond(A;a)$ directly depends on the convexity regions of our cubic nonlinearity. Namely, there exists a unique inflection point $v_{i} = v_{i}(a)$ on the interval $(0,1)$ such that $g$ is convex on $(0, v_{i})$ and concave on $(v_{i}, 1)$. A straightforward computation shows that 
\begin{equation*}
    {v}_{i}(a) = \dfrac{a+1}{3}.
\end{equation*}
%We make two important observations, which we justify in Lemmas~\ref{lemma:cubic:neg:case2:d_diamond} and \ref{lemma:cubic:neg:case1:d_diamond}. %\begin{enumerate}
%   \item If $A\in ({v}_{i}, 1)$, %then $-d^\diamond (k+1) $ is precisely the slope  at which the linear function $$v\mapsto d^\diamond (k+1)(A-v) - g(A;a)$$ touches the nonlinearity $-g(v;a)$ tangentially at some touching point $u_{tp}$, see Fig~\ref{3:fig:main:d_diamond:neg}.
%   the linear function $$v\mapsto d^\diamond (k+1)(A-v) - g(A;a)$$ touches the nonlinearity $-g(v;a)$ tangentially at some touching point $u_{tp}$, see Fig~\ref{3:fig:main:d_diamond:neg}, left. 
%   \item For $A\in (a,{v}_{i})$, we have  $d^\diamond(A;a) = g'(A;a)/(k+1)$. 
%\end{enumerate}
\begin{lemma}\label{lemma:cubic:neg:case2:d_diamond}
Let $g$ be the standard cubic nonlinearity~\eqref{eqn:cubic:def}. Pick any $a\in (0,1)$ and  $A\in \left(v_i , 1\right)$. Then   the linear function $$v\mapsto d^\diamond (k+1)(A-v) - g(A;a)$$ touches the nonlinearity $v \mapsto -g(v;a)$ tangentially at some touching point $u_{tp}$; see Fig~\ref{3:fig:main:d_diamond:neg} (left). Moreover, 
we have
\begin{align}
    u_{tp} &= \dfrac{1}{2}(1+a - A) \in \left(\frac{a}{2}, v_i\right), \label{eqn:cubic:neg:case2:utp:def} \\
    d^\diamond(A;a) &= \dfrac{-3A^2 + 2(a+1)A + (a-1)^2}{4(k+1)} > 0. \label{eqn:cubic:neg:case2:d_diamond:def}
\end{align}
\end{lemma}
\begin{proof}%[Proof of Lemma~\ref{lemma:cubic:neg:case2:d_diamond}] 
In order to find $d^\diamond(A;a)$ and the touching point $u_{tp}$ for $A\in (v_i, 1)$ we exploit the  idea used by Keener in~\cite{keener1987propagation} for $k=1$ and 
 match the coefficients of two cubic polynomials. In particular,  we write
\begin{equation}\label{eqn:cubic:neg:matching}
    g(v;a) + d^\diamond (k+1)(A-v) - g(A;a) = (v- u_{tp})^2 (A-v). 
\end{equation}
The polynomial on the right-hand-side is always positive on $[0, A]$. Moreover, if we show that $u_{tp}<A$, 
then $d^\diamond (k+1) $  is indeed the smallest possible slope such that the line $ d^\diamond (k+1)(A-v) - g(A;a)$ stays above the graph of $-g$ for $v\in [0, A]$. However, this inequality follows easily from $A>v_i$ which implies that $u_{tp} < v_i < A $.
%The values $u_{tp}$ and $ d^\diamond(A;a) $ that follow directly by matching the coefficients in~\eqref{eqn:cubic:neg:matching}. 
\end{proof}
\begin{lemma}\label{lemma:cubic:neg:case1:d_diamond}
Pick $a\in (0, \frac{1}{2})$ and $A\in (a,v_i)$. Then we have
\begin{equation}
    d^\diamond(A;a) = \frac{g'(A;a)}{k+1}. 
\end{equation}
\end{lemma}
\begin{proof} %[Proof of Lemma~\ref{lemma:cubic:neg:case1:d_diamond}]
The choice $a<\frac{1}{2}$ implies that the function $-g$ is concave on $(0, v_i)$.
This implies that the line with the smallest slope that stays above the 
graph of $-g$ on the interval $(0, A)\subset (0, v_i)$ is indeed given by $d^\diamond (k+1)(A-v) - g(A;a)$ for $d^\diamond = g'(A;a)/(k+1)$. 
%Therefore,
%any line passing through through point $(A, -g(A))$ with the slope $-d(k+1)> -g'(A;a)$ stays below the graph of $-g$ until it crosses it at some point $\overline{A} < A.$ Therefore, we must have $d\leq g'(A;a)/(k+1)$.
%On the other hand, for any $d\geq g'(A;a)/(k+1)$ the line crossing through the point $(A, -g(A;a))$ with the slope $-d(k+1)$ stays above the function $-g(v;a)$. This shows that the smallest slope is indeed given with $g'(A;a)/(k+1)$.
\end{proof}
Recall the definition \eqref{eqn:main:Dminus}
and pick $a\in(0,1)$. We  denote by $\mathcal{A}(a)$ the set of admissible parameters $A$, namely
\begin{equation}
    A\in \mathcal{A}(a) \iff d^\diamond(A;a) \leq \dfrac{g(A;a)}{A}.
\end{equation}
On account of Lemmas~\ref{lemma:cubic:neg:case2:d_diamond} and \ref{lemma:cubic:neg:case1:d_diamond},  we have to separately consider the two  cases $A\in (a, v_i)$ and $A\in (v_i, 1)$ in our study of $d^\diamond(A;a)$. We therefore define two subsets of $\mathcal{A}(a)$, namely
\begin{align*}
    \mathcal{A}_1(a) = (a, v_i) \cap \mathcal{A}(a), \qquad 
    \mathcal{A}_2(a) = [v_i, 1) \cap \mathcal{A}(a).
\end{align*}
A key point in our analysis is that the contribution from the parameters $A\in (a, v_i)$  can be safely neglected. In particular, we have the following result.
%However, our analysis shows that the contribution from the parameters $A\in (a, v_i)$ does not influence the upper and lower bounds of intervals $(d^\diamond(A;a), g(A;a)/A)$. In particular, we have the following result.

\begin{lemma}\label{lemma:cubic:A_2} Let $g$ be the standard cubic nonlinearity \eqref{eqn:cubic:def}. Then we have
the identities% hold. 
\begin{align}
    \min_{A\in \mathcal{A}(a)} d^\diamond(A;a) = \min_{A\in \mathcal{A}_2(a)} d^\diamond(A;a), \\
    \max_{A\in \mathcal{A}(a)} \dfrac{g(A;a)}{A} = \max_{A\in \mathcal{A}_2(a)} \dfrac{g(A;a)}{A} .
\end{align}
\end{lemma}
\begin{proof}
See \S\ref{3:subsec:cubic:lemma}.
\end{proof}

%\begin{lemma}\label{lemma:cubic:neg:case2:d_diamond}
%Pick any $A\in \left(v_i , 1\right)$. Then we have
%\begin{align}
%    u_{tp} &= \dfrac{1}{2}(1+a - A) \\
%    d^\diamond(A;a) &= \dfrac{-3A^2 + 2(a+1)A + (a-1)^2}{4(k+1)} > 0 \label{eqn:cubic:neg:case2:d_diamond:def}
%\end{align}
%\end{lemma}

In the following lemma we further characterize the set $\mathcal{A}_2(a)$. In particular, we show that there exists an upper bound on $a$ for which $\mathcal{A}_2(a)$ is not an empty set. 
\begin{lemma}\label{lemma:cubic:neg:case2:cond}
Let $g$ be the standard cubic nonlinearity \eqref{eqn:cubic:def}.  Pick any parameter $a\in (0,1)$ and recall the value $a_*^-(k)$ defined by \eqref{eqn:cubic:neg:a1}. Then the following claims hold.
\begin{enumerate}[(i)]
    \item\label{item:cubic:neg:A2empty} If $a>a_*^-(k)$ then 
    \begin{equation*}
        \mathcal{A}_2(a) =\emptyset.
    \end{equation*}
    \item \label{item:cubic:neg:A2boundaries} If $a\leq a_*^-(k)$ then 
$$ \mathcal{A}_2(a) =  [v_i, 1) \cap [A_2^-(a), A_2^+(a)],$$  %= \left[\min\left\{v_i, A_2^-(a)\right\}, A_2^+(a)\right],
where $A_2^-(a)$ and $A_2^+(a)$ are defined by
\begin{align*}
    A_2^-(a) &=  \dfrac{(1+a)(1+2k) - 2\sqrt{k^2(a-1)^2 - ka}}{4k+1},\\
    A_2^+(a) &= \dfrac{(1+a)(1+2k) + 2\sqrt{k^2(a-1)^2 - ka}}{4k+1}.
\end{align*}
Moreover, we have the inequalities
\begin{equation}\label{lemma:cubic:neg:case2:ordering}
     A_2^+(a) \geq \dfrac{a+1}{2} , \qquad A_2^+(a) \in (1-a, 1).
\end{equation}
\end{enumerate}
\end{lemma} 

\begin{proof} 
Pick $A\in [v_i, 1)$. %We first show item~\textit{\ref{item:cubic:neg:A2boundaries}}. 
By Lemma \ref{lemma:cubic:neg:case2:d_diamond}, we have $A\in \mathcal{A}_2(a)$ if and only if
 \begin{equation}\label{eqn:cubic:neq:inequality:A}
      \dfrac{-3A^2 + 2(a+1)A + (a-1)^2}{4(k+1)} \leq (1-A)(A-a).
\end{equation}
%Solutions to this quadratic inequality are precisely given by
This quadratic inequality has solutions if and only if $A\in [A_2^-(a), A_2^+(a)]$, which  are well defined for 
\begin{equation*}
    k^2(a-1)^2 - ka \geq 0,
\end{equation*}
which is equivalent to $a\leq a_*^-(k)$. On the other hand,  for $a>a_*^-(k)$ there is no solution to~\eqref{eqn:cubic:neq:inequality:A}, establishing \textit{(\ref{item:cubic:neg:A2empty})}.

The inequality $A_2^+(a)> (1+a)/2$ follows directly from
\begin{equation*}
    A_2^+(a)\geq \dfrac{(1+a)(1+2k)}{4k+1} = \dfrac{(1+a)}{2} + \dfrac{(1+a)}{2(4k+1)} > \dfrac{(1+a)}{2}.
\end{equation*}
To show $A_2^+(a) > 1-a$,  we write
\begin{equation*}
    A_2^+(a) - (1 - a)
    = 2 \frac{a(1+3k)-k+\sqrt{k^2(1 -a)^2 - ka}}{1+4k}.
\end{equation*}
For $a\geq k/(1+3k)$ the numerator is immediately positive. To examine the case $a<k/(1+3k)$ we define the quadratic expression $\mathcal{Q}_1(a,k)$ by
\begin{equation*}
\mathcal{Q}_1(a,k)
= k^2(1-a)^2 - ka - 
    (a + 3ka - k)^2
    = a \big( k + 4k^2-(1 + 6k +8 k^2)a  \big).
\end{equation*}
This is strictly positive for
$0 < a \leq \frac{k}{1 + 3k}$, since  $\mathcal{Q}_1(0, k) = 0$ and
\begin{equation*}
   \mathcal{Q}_1\left( \frac{k}{1 + 3k}, k \right)
    = \frac{k}{1 + 3k} \frac{ k^2(1+4k)}{1+3k}
    = \frac{k^3(1+4k)}{(1 + 3k)^2} > 0.
\end{equation*}
To establish our final inequality $ A_2^+(a) < 1$, we note that 
\begin{equation}\label{eqn:cubic:A2:1}
    1 -A_2^+(a)
    =  \frac{2k -(1 + 2k) a  -2 \sqrt{k^2(1 -a)^2 - ka}}{1+4k}.
\end{equation}
Upon writing
\begin{equation*}
\mathcal{Q}_2(a,k)
= \big(2k - (1+2k)a  \big)^2   - 4\big( k^2(1-a)^2 - ka \big)
    = 
    ( 1 + 4k)a^2  ,
\end{equation*}
we see that \eqref{eqn:cubic:A2:1} is indeed strictly positive.
\end{proof}

\begin{lemma}\label{lemma:cubic:neg:max}
Let $g$ be the standard cubic nonlinearity \eqref{eqn:cubic:def}. Pick $k>0$
together with $a\in (0, a_*^-(k))$ and recall the constant $a^-_1(k)$ defined by~\eqref{eqn:cubic:neg:a1}. Then we have
\begin{equation*}
    \max_{A\in \mathcal{A}_2(a)} \dfrac{g(A;a)}{A} = 
    \begin{cases}
    \dfrac{(1-a)^2}{4}, \qquad &\text{if }a\in (0, a^-_1(k)], \\[0.3cm]
    \dfrac{g(A_2^-(a);a)}{A_2^-(a)}, \qquad &\text{if } a\in [a^-_1(k), a_*^-(k)).
    \end{cases} 
\end{equation*}
%Moreover, for $k\leq 1/4$ the interval $(0, a^-_1(k)]$ is empty. In that case we have 
%\begin{equation}
%    \max_{A\in \mathcal{A}_2(a)} \dfrac{g(A;a)}{A} =  \dfrac{g(A_2^-(a);a)}{A_2^-(a)}.
%\end{equation}
\end{lemma}

\begin{proof}
Let us first define  $A_{\mathrm{max}} = \frac{a+1}{2}$. A standard analysis shows that 
\begin{equation*}
    \max_{A\in (0,1)}\frac{g(A;a)}{A} =  \frac{g(A_{\mathrm{max}};a)}{A_{\mathrm{max}}} .
\end{equation*}
By Lemma~\ref{lemma:cubic:neg:case2:cond} we have 
\begin{equation*}
    \max_{A\in \mathcal{A}_2(a)} \dfrac{g(A;a)}{A} = \begin{cases}
    \dfrac{g(A_{\mathrm{max}};a)}{A_{\mathrm{max}}}, & A_2^-(a)\leq A_{\mathrm{max}},   \\[0.3cm]
    \dfrac{g(A_2^-(a);a)}{A_2^-(a)}, & A_2^-(a)\geq A_{\mathrm{max}}.
    \end{cases}
\end{equation*}
 We claim that for $a\in (0,1)$ we have
\begin{equation}\label{eqn:cubic:neg:A2<Amax}
A_2^-(a)\leq A_{\mathrm{max}}(a) \iff a\in (0, a^-_1(k)). 
\end{equation}
Indeed, the inequality on the left can be written as 
\begin{equation*}
    \dfrac{(1+a)(1+2k) - 2\sqrt{k^2(a-1)^2 - ka}}{4k+1} \leq \dfrac{1+a}{2},
\end{equation*}
which reduces to  
\begin{equation}\label{eqn:cubic:neg:k/4}
    a^2(4k-1) - 2a(4k+1) + 4k-1 \geq 0.
\end{equation}
If $k>1/4$ then this expression
is positive for 
$$
a\leq  1 - \dfrac{4\sqrt{k}-2}{4k-1} = 1-\dfrac{2}{2\sqrt{k}+1} \qquad \text{and}\qquad a\geq 1 + \dfrac{2}{2\sqrt{k}-1}.
$$

We recognize that the first value is exactly equal to $a^-_1(k)$, while the second value is greater than $1$ and therefore not of  interest. 
For $k\leq 1/4$ there is no solution of \eqref{eqn:cubic:neg:k/4} in the set of positive numbers. 
\end{proof}
%\textbf{Remark}
%Even though we consider the regime $k>1$, we note that for $k\leq 1/4$ the analysis of~\eqref{eqn:cubic:neg:k/4} would yield two zero points that are both smaller than $1$

\begin{lemma}\label{lemma:cubic:min_d:A2}
Let $g$ be the standard cubic nonlinearity \eqref{eqn:cubic:def} and pick any $a\in (0, a^-_*(k))$. Then we have
\begin{equation*}
    \min_{A\in \mathcal{A}_2^+(a)} d^\diamond(A;a) = \dfrac{g(A_2^+(a);a)}{A_2^+(a)}.
\end{equation*}
\end{lemma}

\begin{proof}

 The graph of $d^\diamond(A;a)$ is a downwards parabola, positive on some superset of $(0, 1)$, with the maximum at $A=v_i\leq A_2^+(a)<1$. Therefore,  the minimum is attained at the right boundary  $A_2^+(a)$. 
\end{proof}

\begin{proof}[Proof of Proposition~\ref{cor:main:cubic:neg}]
Direct computation yields 
\begin{align*}
    \dfrac{g(A_2^-(a);a)}{A_2^-(a)} &= \dfrac{2a^2k -a + 2k  + 2(a+1) \sqrt{k} \sqrt{ka^2 -a (2k+1) + k} }{(4k+1)^2}, \\
    \dfrac{g(A_2^+(a);a)}{A_2^+(a)} &= \dfrac{2a^2k -a + 2k  - 2(a+1) \sqrt{k} \sqrt{ka^2 -a (2k+1) + k} }{(4k+1)^2}.\\
\end{align*}
Applying Lemmas~\ref{lemma:cubic:A_2}, \ref{lemma:cubic:neg:max} and~\ref{lemma:cubic:min_d:A2} now guarantees that the upper and lower boundary of the set $\mathcal{D}^-$ are given by $d_{\mathrm{max}}$ and $d_{\mathrm{min}}$. The fact that the cubic nonlinearity satisfies \Hgone ensures that the whole set $\mathcal{D}^-$ is given as the area between these curves,  establishing \textit{(\ref{item:main:cubic:3})}. Items \textit{(\ref{item:main:cubic:1})} and \textit{(\ref{item:main:cubic:2})} follow directly from the construction of $d_{\mathrm{max}}$ and $d_{\mathrm{min}}$.
\end{proof}

\subsection{Proof of Lemma~\ref{lemma:cubic:A_2}}\label{3:subsec:cubic:lemma} 
 In this section we complete our analysis of the cubic nonlinearity by establishing Lemma~\ref{lemma:cubic:A_2}. 
In addition to the points $a_*^-(k) $ and $ a^-_1(k)$ defined by~\eqref{eqn:cubic:neg:a1}, we  introduce a third value that plays an important role in this section, namely 
%we define $a_2(k)$ by
\begin{align}
    a_2(k)&: = \min \left\{0, 1 - \dfrac{2\sqrt{k+4}}{\sqrt{k+4} + 3\sqrt{k}} \right\}. %\dfrac{2+5k - 3\sqrt{k^2 + 4k}}{2(2k-1)}  \label{eqn:cubic:neg:a2}
    \label{eqn:cubic:neg:a2}
\end{align}
In the following lemma we show that these  three points are always ordered, irrespective of $k>0$. 
\begin{lemma}\label{lemma:cubic:neg:a1a2a*}
For every $k>0$ we have the ordering
\begin{equation}\label{eqn:cubic:ordering}
    a_2(k) \leq  a^-_1(k) < a^-_*(k).
\end{equation}
\end{lemma}

\begin{proof}

Our first observation is that for $k>0$ we have
\begin{align*}
    a^-_1(k) &> 0 \iff k > \frac{1}{4}, \\
    a_2(k) &> 0 \iff k > \frac{1}{2}, \\
    a^-_*(k) &> 0 \iff k>0.
\end{align*}
Therefore, for $k\leq \frac{1}{2}$ the ordering $a_2(k) \leq  a^-_1(k)$ trivially holds.
For $k>\frac{1}{2}$, the inequality $a_2(k) \leq a^-_1(k)$ is equivalent to 
\begin{align*}
   \dfrac{1}{2\sqrt{k}+1} \leq \dfrac{\sqrt{k+4}}{\sqrt{k+4} + 3\sqrt{k}},
\end{align*}
which is in turn equivalent to
\begin{equation*}
\sqrt{k}\left( 2\sqrt{k+4} - 3\right) \geq 0.
\end{equation*}
This holds for all $k>0$. 
To show $a^-_1(k)\leq a^-_*(k)$ we apply the bound $\sqrt{4k+1}\geq 2\sqrt{k} $ to the denominator of $a_*^-(k)$.   This concludes the proof. 
\end{proof}

\begin{lemma}\label{lemma:cubic:neg:case1:cond}
Let $g$ be the standard cubic nonlinearity \eqref{eqn:cubic:def}. Pick $k>0$ and $a\in (0, \frac{1}{2})$. Then we have
\begin{equation*}
   \mathcal{A}_1(a) \neq \emptyset \iff a\in \left(0, a_2(k)\right).
\end{equation*}
\end{lemma}

\begin{proof}
In view of Lemma~\ref{lemma:cubic:neg:case1:d_diamond}, we have $d^\diamond(A;a)\leq g(A;a)/A$ if and only if
\begin{equation}
   \dfrac{ g'(A;a) } { k+1 } \leq  \dfrac{g(A;a)}{A}, 
\end{equation}
which can be rewritten as 
\begin{equation*}
    f(A;a) := A^2(k-2) + A(1-k)(a+1) + ka \leq 0.
\end{equation*}
To examine this quadratic function, we first note that  $f(a;a) = a(1-a) > 0$ and  $ f(1;a)= a-1 < 0$. By showing that 
$$\overline{f}(a):= f(v_i;a) = f\left(\frac{a+1}{3};a\right)>0 \iff a  > a_2(k),$$
it follows that $f$ must also be positive on $(a, v_i)$. Consequently, there exists no $A\in (a, v_i)$ such that $\overline{f}(A)\leq 0$. To establish this claim, we compute 
%Plugging in the value $A = v_i$ we obtain a new quadratic inequality
%in parameter $a$, namely 
\begin{equation}\label{cubic:parabola1}
   \overline{f}(a) =  \dfrac{1}{9}\left((1-2k)a^2 + (2+5k)a + 1-2k\right).
\end{equation}
For $k>1/2$, the graph of the mapping $a\mapsto \overline{f}(a)$ is a downward orientated parabola with two roots, the smaller of which is given exactly by $a_2(k)$. Moreover, we can directly check that the expression $\overline{f}(1/2)$ is equal to $0.25 > 0$. Therefore, for all $a\in (a_2(k), 1/2)$ we have $ \overline{f}(a) > 0$. 
For $k\leq 1/2$, all roots of $a\mapsto \overline{f}(a)$ are nonpositive, which implies that $\mathcal{A}_1(a)$ is an empty set for all $a\in (0, \frac{1}{2})$.
\end{proof}

%\begin{lemma}
%Pick $a\in (0,1)$. Then we have 
%\begin{equation*}
%    \cup_{A\in (0,1)} \left(d^\diamond(a;A), \dfrac{g(A;a)}{A}\right) \neq \emptyset \iff a\in (0, a^-_*(k)).
%\end{equation*}
%\end{lemma}

%\begin{proof}
%The conclusion follows directly from Lemmas \ref{lemma:cubic:neg:case1:cond} and \ref{lemma:cubic:neg:case2:cond}  in combination with the inequality $a_2(k)\leq a^-_*(k)$, which is shown in Lemma~\ref{lemma:cubic:neg:case2:ordering}.
%\end{proof}

\begin{lemma}\label{lemma:cubic:neg:lem1}
Let $g$ be the standard cubic nonlinearity \eqref{eqn:cubic:def}. Pick any $a\in (0, a^-_*(k))$. Then we have
\begin{equation*}
    \max_{A\in \mathcal{A}_1(a)} \dfrac{g(A;a)}{A} \leq  \max_{A\in \mathcal{A}_2(a)} \dfrac{g(A;a)}{A}.
\end{equation*}
\end{lemma}

\begin{proof}
If $a> a_2(k)$   the claim trivially holds since $\mathcal{A}_1(a) = \emptyset$. 
If $a\leq a_2(k)$ then we automatically have $a\leq a^-_1(k)$ due to Lemma~\ref{lemma:cubic:neg:case2:ordering}.  
By Lemma~\ref{lemma:cubic:neg:max}
 the maximum of $g(A;a)/A$ is attained on $(0, a^-_1(k)]$ as $A_\mathrm{max} $ belongs to $ \mathcal{A}_2(a)$. Therefore,  the contribution from the values of $A\in \mathcal{A}_1(a) $ cannot exceed this maximum. 
\end{proof}

\begin{lemma}\label{lemma:cubic:neg:lem2}
Let $g$ be the standard cubic nonlinearity \eqref{eqn:cubic:def}. Pick any $a\in (0, a_*^-(k))$. Then we have
\begin{equation*}
    \min_{A\in \mathcal{A}_1(a)} d^\diamond(A;a) \geq  \min_{A\in \mathcal{A}_2(a)} d^\diamond(A;a).
\end{equation*}
\end{lemma}

\begin{proof}
If $a\geq a_2(k)$ the claim trivially holds since $\mathcal{A}_1(a) =\emptyset$. We therefore assume
 $a \in (0, a_2(k))$ and recall from Lemma~\ref{lemma:cubic:min_d:A2} that 
\begin{equation*}
    \min_{A\in \mathcal{A}_2(a)} d^\diamond(A;a) = \dfrac{g(A_2^+(a))}{A_2^+(a)}. 
\end{equation*}
By Lemma~\ref{lemma:cubic:neg:case2:cond}  we also know that $A_2^+(a) > 1-a$, which in turn gives 
\begin{equation}\label{eqn:cubic:utp}
u_{tp}(A_2^+(a)) < a. 
\end{equation}
Assume now to the contrary that there exists $A\in \mathcal{A}_1(a)\subset (v_i, a)$ for which
\begin{equation}\label{eqn:cubic:contr}
    \dfrac{g(A_2^+(a);a)}{A_2^+(a)} > \dfrac{g'(A;a)}{k+1}.
\end{equation}
Since $-g$ is concave on $(v_i, a)$ the linear map $g'(A;a)(A-v) - g(A;a)$ crosses the $v$-axis at some point $\tilde{A}>a$. However, \eqref{eqn:cubic:utp} automatically implies that $d^\diamond(A_2^+(a);a) \leq g'(A;a)/(k+1)$, which clearly contradicts ~\eqref{eqn:cubic:contr} and hence establishes our claim. 
\end{proof}

\begin{proof}[Proof of Lemma~\ref{lemma:cubic:A_2}]
The claim follows directly from Lemmas~\ref{lemma:cubic:neg:lem1} and \ref{lemma:cubic:neg:lem2}.
\end{proof}

\section{Spatial chaos}\label{sec:chaos}
%\todo[inline]{+[hjh: maybe spatial chaos for title of section?]}
To prove Proposition~\ref{prop:prop_failure:steady_sols:existence}, we follow the outline from \cite{keener1987propagation} and adapt the Moser theorem from \cite{moser2016stable}. 
We first note that the solutions of the MFDE~\eqref{eqn:main:MFDE} with $c=0$ are equivalent to
 steady-state solutions of~\eqref{eqn:main:LDE}, i.e., sequences $(u_i)_{i\in \Z}$ that satisfy the difference equation
\begin{equation}\label{eqn:prop_failure:wave_eqn:c=0}
    d\left(u_{i-1} - (k+1)u_i+ k u_{i+1}\right) + g(u_i;a) = 0.
\end{equation}

 To find a solution to~\eqref{eqn:prop_failure:wave_eqn:c=0}, we introduce a new sequence $(v_i)_{i\in\Z}$ by setting $v_{i}:=u_{i-1}$. This allows to rewrite~\eqref{eqn:prop_failure:wave_eqn:c=0} as the two-dimensional recursion relation 
\begin{equation}\label{eqn:prop_failure:2d_system}
    \begin{cases}
     v_{i+1} &= u_i, \\
     u_{i+1} &= \dfrac{k+1}{k} u_i - \dfrac{v_i}{k} - \dfrac{g(u_i;a)}{kd},
    \end{cases}
\end{equation}
for $i\in \Z$. Writing $\phi:\R^2\to\R^2$ for the map 
\begin{equation}\label{eqn:prop_failure:phi}
    \phi(u,v) := \left(\frac{k+1}{k}u -\frac{1}{k} v - \frac{g(u;a)}{kd} , u\right), 
\end{equation}
we notice that solving~\eqref{eqn:prop_failure:2d_system} is equivalent to constructing a sequence $(u_i, v_i)_{i\in \Z}$ in $\R^2$ that has
\begin{equation}\label{eqn:prop_fauilure:shift}
    \phi(u_i, v_i) = (u_{i+1}, v_{i+1}).
\end{equation}
The inverse of the mapping $\phi$ is given by 
\begin{equation}
    \phi^{-1} (\Tilde{u}, \Tilde{v}) = \left(\Tilde{v}, (k+1)\Tilde{v} -  k \Tilde{u} - \dfrac{1}{d}g(\Tilde{v};a)\right) 
\end{equation}
and a straightforward calculation shows that $\phi$ and $\phi^{-1}$ are further related by the identity
\begin{equation*}
    \phi^{-1} = R_k \phi R_k,
\end{equation*}
where $R_k$ is given by
\begin{equation*}
R_k = 
    \begin{pmatrix}
      0 & 1 \\
      k & 0
    \end{pmatrix}. 
\end{equation*}
In the special case $k=1$, this matrix represent reflection through  the line $v=u$. 
\subsection{The Moser theorem}

We first define a few notions that we use throughout this section. We call a curve $v = v(u)$ a \textit{horizontal curve} if $0\leq v(u)\leq 1$ for $0\leq u\leq 1$. Analogously, we call a curve $u = u(v)$ a \textit{vertical curve} if $0\leq u(v) \leq 1$ for $0\leq v\leq 1$. For two disjoint horizontal curves $0\leq v_1(u) < v_2(u)\leq 1$ we call the set 
\begin{equation*}
    U = \left\{(u,v): 0\leq u\leq 1: v_1(u) \leq v \leq v_2(u) \right\}
\end{equation*}
a \textit{horizontal strip}. Similarly, we define a \textit{vertical strip} as an area $V$  lying between disjoint vertical curves $0\leq u_1(v) < u_2(v)\leq 1$, namely
\begin{equation*}
    V = \left\{(u,v): 0\leq v\leq 1: u_1(v) \leq u \leq u_2(v) \right\}.
\end{equation*}
We also introduce the space $\mathcal{S}$ containing all bi-infinite sequences with elements in $\{0,1\}$, i.e., 
\begin{equation*}
    \mathcal{S}:= \left\{(\dots,s_{-2}, s_{-1}, s_0, s_1, s_2, \dots): s_i \in \{0,1\}\right\}. 
\end{equation*}
This space $\mathcal{S}$  when endowed with an appropriate topology makes a topological space \cite{moser2016stable}, on which we define the forward shift  $\sigma:\mathcal{S}\to \mathcal{S}$  by
\begin{equation*}
    [\sigma(s)]_i = s_{i+1}.
\end{equation*}
\begin{theorem}\cite[Moser]{moser2016stable}\label{thm:prop_failure:moser}
Suppose for $n \in \{0, 1\}$  that $U_n, V_n$  are disjoint horizontal and vertical, respectively, strips in $Q:=[0,1]^2$ that additionally satisfy
\begin{enumerate}[(i)]
    \item $\phi(V_n) = U_n$, $n\in \{0,1\}$.
    \item\label{item:Moser:boundaries} The vertical
 boundaries of $V_n$ are mapped to vertical boundaries of $U_n$ and the horizontal boundaries
 of $V_n$ are mapped to horizontal boundaries of $U_n$.
\end{enumerate}
 Then there exists a function $\tau:\mathcal{S}\mapsto \mathcal{Q}$ such that
  \begin{equation*}
      \phi \tau = \tau \sigma.
  \end{equation*}
In addition, the function $\tau$ satisfies 
\begin{equation*}
    \phi^{i}\tau(s)\in U_{s_i}, \qquad s\in \mathcal{S}, \ i\in \Z. 
\end{equation*}
\end{theorem}
Stated informally, we say that $\phi$ possesses the shift $\sigma$ on sequences of elements of $\{0,1\}$ as a subsystem. 
The main consequence of the Moser theorem is that for every sequence $s\in \mathcal{S}$ we can find a sequence $(u_i, v_i)_{i\in \Z}$ satisfying~\eqref{eqn:prop_fauilure:shift} with $(u_i, v_i) \in U_{s_i}$ for every $i\in \Z$. To achieve this, we simply set  $(u_0, v_0): = \tau(s)$ and $(u_i, v_i) := \phi^i (u_0, v_0)$.

\paragraph{Construction of horizontal and vertical strips}
Let us define a function $h$ that acts as 
\begin{equation}\label{eqn:moser:h}
    h(v;a, d) = (k+1) v - \dfrac{1}{d} g(v;a).
\end{equation}
To construct the strips $U_n$ and $V_n$, for $n=0, 1$, we need to ensure that the parameter $d$ is small enough so that the following assumption holds.

 \begin{itemize}
 \item[\Hd]  There exist points $y_0$ and $y_1$, satisfying $0<y_0<a$ and $a<y_1<1$ such that
     \begin{equation*}
        \begin{aligned}
           h(y_0;a, d) &> k+1,    \\
           h(y_1;a, d) &< 0, \\
           h'(v;a, d) & > 0 , \quad  v\in (0, y_0) \cup (y_1, 1) .
        \end{aligned}
     \end{equation*}
 \end{itemize}

\begin{lemma}\label{lemma:prop_failure:strips_existence}
Assume that conditions \Hg and \Hd hold. Then there exist horizontal strips $U_0$, $U_1$, and vertical strips $V_0$, $V_1$ that satisfy the assumptions of Theorem~\ref{thm:prop_failure:moser}.
\end{lemma}

\begin{proof}[Proof of Proposition~\ref{prop:prop_failure:steady_sols:existence}]
It suffices to show that the function $h$ defined by~\eqref{eqn:moser:h} satisfies  condition \Hd for all sufficiently small $d> 0$. Indeed,  we can then combine the Moser Theorem~\ref{thm:prop_failure:moser} and Lemma~\ref{lemma:prop_failure:strips_existence} to obtain the desired conclusion. 

On the interval (0, a) we have
\begin{align*}
    h(v; a, d) - (k+1) = (k+1)(v-1) - \dfrac{1}{d} g(v;a) \geq -(k+1) -  \dfrac{1}{d} g(v;a).
\end{align*}
We now choose $\delta>0$ in such a way that the function $ v\mapsto g(x;a)$ is strictly negative and decreasing on $(0, \delta)$. By choosing $d$ small enough we can therefore achieve $h(\delta;a, d) - (k+1) > 0$. This shows that we can choose $y_0 = \delta$.
The point $y_1$ can be found analogously.
\end{proof}

\begin{lemma}
Assume that conditions \Hg and \Hd hold. 
Then there exist six points $(x_i)_{i=1}^3$ and $(z_i)_{i=0}^2$ that satisfy the identities
\begin{equation*}
    \begin{aligned}
  % h(x_1;a, d) = ({k+1}) x_1  - \dfrac{1}{d} g(x_1;a) = 1 \\[0.3cm]
  %   h(x_2; a, d) =   (k+1)x_2  - \dfrac{1}{d} g(x_2;a) = k \\[0.3cm]
  %  h(x_3; a, d) =  (k+1) x_3 -\dfrac{1}{d} g(x_3;a) = k+1
   h(x_1;a, d) &= 1, \qquad
   h(x_2; a, d) = k, \qquad
    h(x_3; a, d) =   k+1, \\
    h(z_0; a, d) &= 0, \qquad
    h(z_1;a,d) = 1, \qquad
    h(z_2;a,d) = k,
 \end{aligned} 
\end{equation*}
 %    h(z_0; a, d) = (k+1)z_0 - \dfrac{1}{d} g(z_0;a, d) = 0 \\[0.3cm]
 %    h(z_1;a,d) = ({k+1}) z_1 - \dfrac{1}{d} h(z_1;a, d) = 1 \\[0.3cm]
 %     h(z_2;a,d) =  (k+1) z_2 - \dfrac{1}{d} g(z_2;a, d) = k
together with the identities
$$0<x_1< x_2< x_3 < y_0 < a < y_1< z_0 < z_1 < z_2<1.$$
\end{lemma}
\begin{proof}
The existence of  $x_1, x_2$ and $x_3$ follow directly from assumption \Hd. 
In addition, we have $h(1;a, d) = k+1$ and $h(y_1;a, d) <0$. Again, the monotonicity assumption ensures that we can find points $z_0<z_1<z_2<1$ that satisfy the claim.
\end{proof}

\begin{proof}[Proof of Lemma~\ref{lemma:prop_failure:strips_existence}]
We define the curves $u_1$ and $u_2$ by writing
\begin{equation*}
    \begin{aligned}
         u_1 &:=\left\{(u,v)\in \R^2: 0\leq u\leq x_1, \ v = (k+1) u - \dfrac{1}{d} g(u;a)\right\} = \phi^{-1}\{ (0, \Tilde{v}): 0\leq \Tilde{v}\leq x_1\}, \\ 
         u_2 &:= \left\{(u,v)\in \R^2: x_2 \leq u\leq x_3, \ v = (k+1) u - k - \dfrac{1}{d} g(u;a)\right\} = \phi^{-1}\{ (1, \Tilde{v}): x_2\leq \Tilde{v}\leq x_3\}.
        \end{aligned}
\end{equation*}
\begin{figure}
    \begin{subfigure}{0.61\textwidth}
         \centering
         \includegraphics[width=\textwidth]{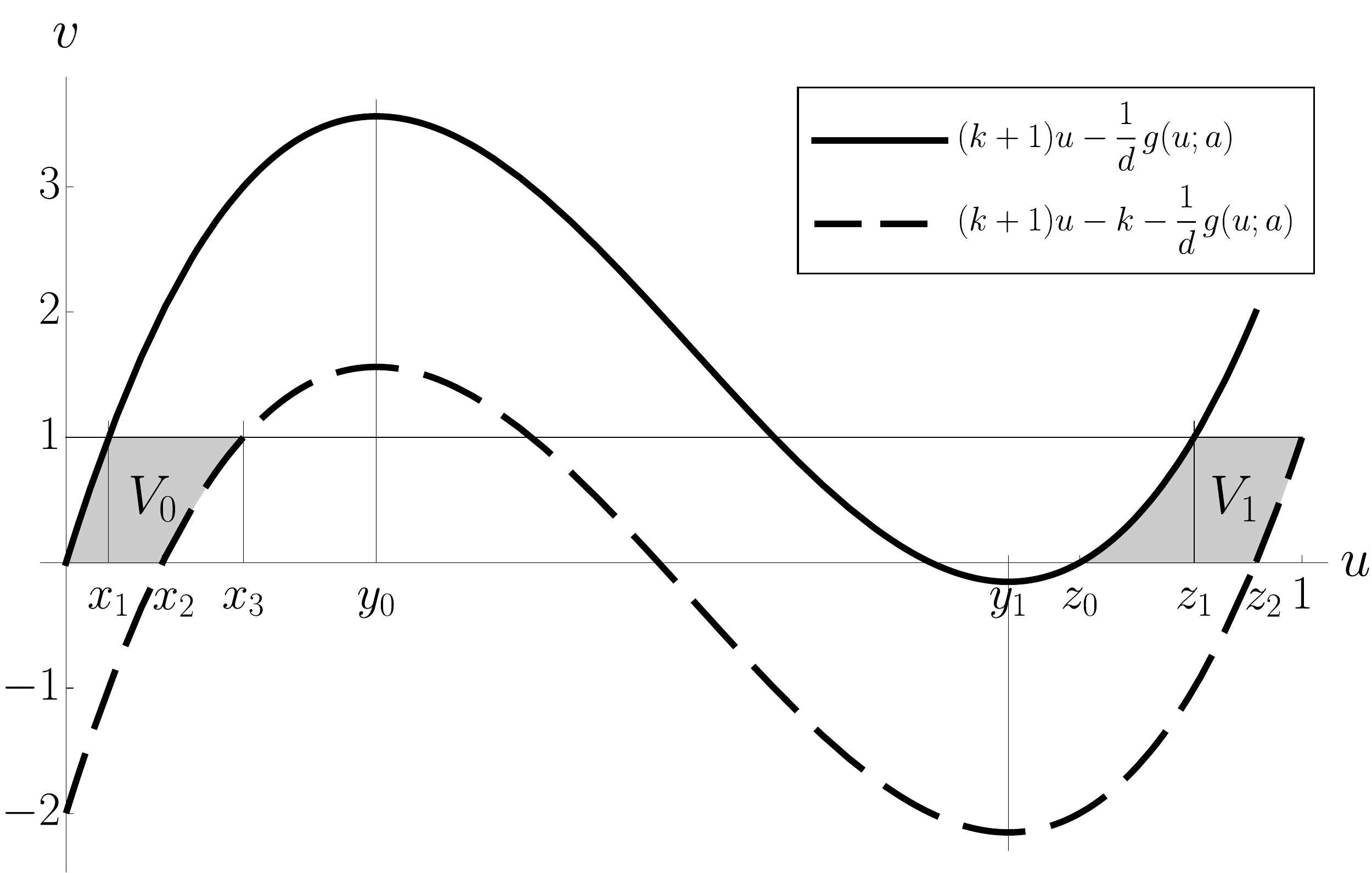}
     \end{subfigure}\hspace{.05\textwidth}
     \begin{subfigure}{0.36\textwidth}
         \centering
         \includegraphics[width=\textwidth]{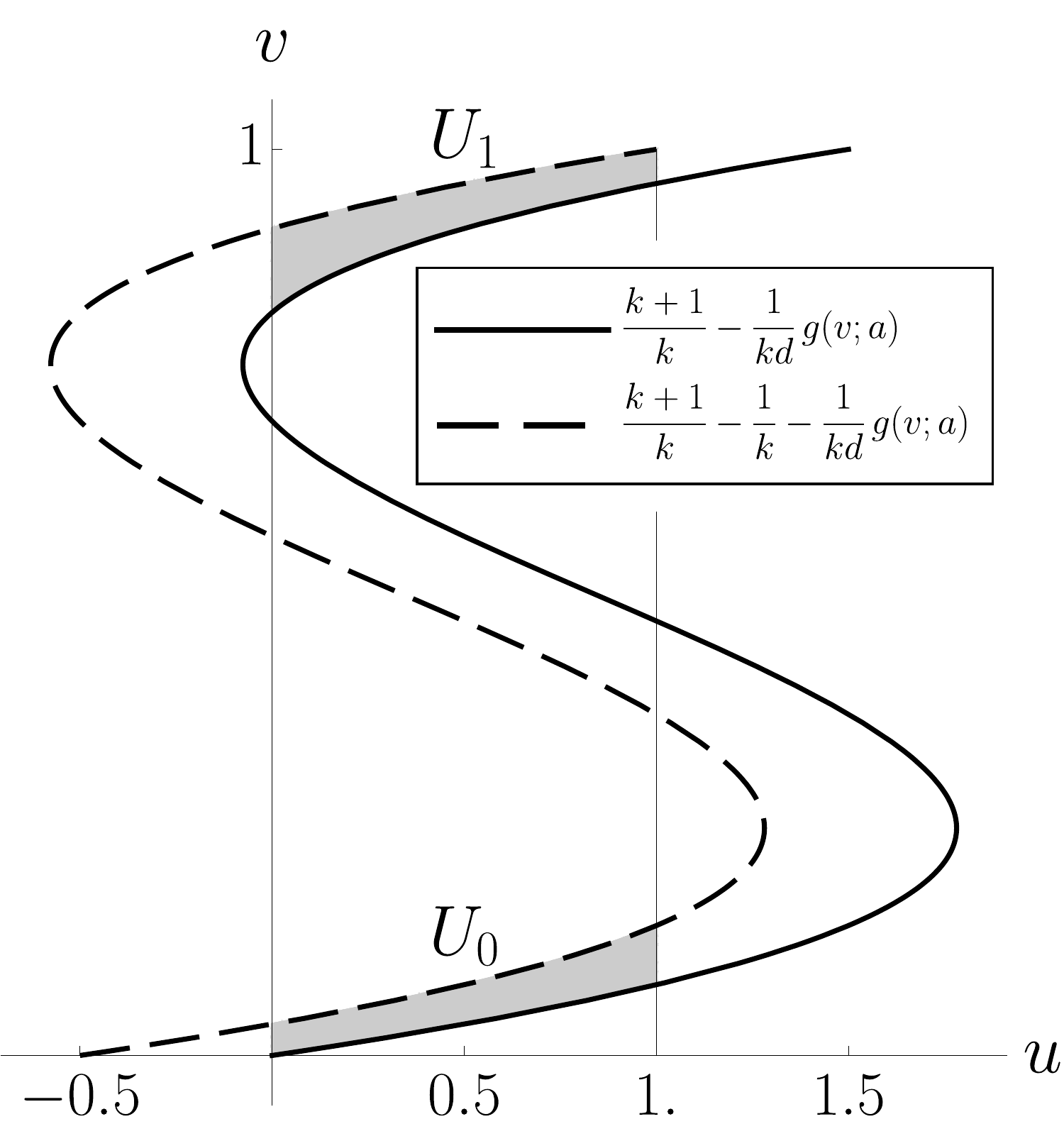}
     \end{subfigure}
    \caption{The sets $V_n$, $U_n$, $n=0,1$ for $k=2$, $d = 0.014$ and the cubic nonlinearity $g(u;a) = u(1-u)(u-a)$ with $a = 0.52$.}
    \label{3:fig:my_label}
\end{figure}
Using the definition of the points $x_1, x_2$ and $x_3$, we see that
the curve $u_1$ connects  the points $(0,0)$ and $(x_1, 1)$, while the curve $u_2$ connects  the points $(x_2, 0)$ and $(x_3, 1)$.
Due to the monotonicity of the  mapping $u\mapsto (k+1) u - \frac{1}{d} g(u;a)$ on $[0, y_0]$, both of these curves  can be represented as graphs $u_1(v)$ and $u_2(v)$ for $v\in [0,1]$.  This proves that these are indeed vertical curves. We now define the set $V_0$ as the area lying between those two curves, and we set $U_0:=\phi(V_0)$.

It remains to show that $U_0$ is a horizontal strip. The
horizontal boundaries of $V_0$, characterized by $\left\{(u,0): 0\leq u\leq x_2\right\}$ and $\left\{(u,1): x_1\leq u\leq x_3\right\}$, respectively,  are mapped by $\phi$ to the curves 
\begin{equation*}
    \begin{aligned}
      v_1&:=\left\{ \left(\frac{k+1}{k}u - \frac{g(u;a)}{kd} , u\right): 0\leq u\leq x_2\right\}, \\
      v_2&:=\left\{ \left(\frac{k+1}{k}u - \frac{1}{k} - \frac{g(u;a)}{kd} , u\right): x_1\leq u\leq x_3\right\}.
    \end{aligned}
\end{equation*}
The curve $v_1$ connects the point $(0, 0)$ with $(1,x_2)$ whereas the curve $v_2$ connects the point $(0, x_1)$ with $(1,x_3)$ and both of these curves are monotonically increasing, implying that they are horizontal strips. 

Finally, the left vertical boundary $u_1$ of $V_0$ is by definition mapped to the set $\{ (0, \Tilde{v}): 0\leq \Tilde{v}\leq x_1\}$, while the right  vertical boundary $u_2$ of $V_0$ is  mapped to the set $\{ (1, \Tilde{v}): x_2\leq \Tilde{v}\leq x_3\}$. This shows that $U_0$ is indeed a horizontal strip, with $V_0$ and $U_0$ satisfying item~\textit{(\ref{item:Moser:boundaries})}.

To construct the set $V_1$, we define the curves $u_3$ and $u_4$ by writing
\begin{equation*}
    \begin{aligned}
         u_3 &:=\left\{(u,v)\in \R^2: z_0\leq u\leq z_1, \ v = (k+1) u - \dfrac{1}{d} g(u;a)\right\} = \phi^{-1}\{ (0, \Tilde{v}): z_0\leq \Tilde{v}\leq z_1\}, \\[0.3cm]
         u_4 &:= \left\{(u,v)\in \R^2: z_2 \leq u\leq 1, \ v = (k+1) u - k - \dfrac{1}{d} g(u;a)\right\} = \phi^{-1}\{ (1, \Tilde{v}): z_2\leq \Tilde{v}\leq 1\}. 
        \end{aligned}
\end{equation*}
Straightforward checks show that
the curve $u_3$ connects  the points $(z_0,0)$ and $(z_1, 1)$, while the curve $u_4$ connects  the points $(z_2, 0)$ and $(1, 1)$
The  map 
\begin{equation*}
    u\mapsto (k+1) u - \frac{1}{d} g(u;a)
\end{equation*}
is increasing on $[y_1, 1]$ so both of these curves  can be represented as graphs $u_3(v)$ and $u_4(v)$ for $v\in [0,1]$. We define the set $V_1$ as the area lying between these two curves and  we set $U_1:=\phi(V_1)$.
%\todo[inline]{[hjh: what is $\Psi$?] }

The function $\phi$ maps the horizontal boundaries of $V_1$,  characterized by the sets 
$$\left\{(u,0): z_0\leq u\leq z_2\right\} \text{ and }\left\{(u,1): z_1\leq u\leq 1\right\},$$ 
to the curves
\begin{equation*}
    \begin{aligned}
      v_3&:=\left\{ \left(\frac{k+1}{k}u - \frac{g(u;a)}{kd} , u\right): z_0\leq u\leq z_2\right\}, \\
      v_4&:=\left\{ \left(\frac{k+1}{k}u - \frac{1}{k} - \frac{g(u;a)}{kd} , u\right): z_1\leq u\leq 1\right\}.
    \end{aligned}
\end{equation*}
The curve $v_3$ connects the point $( 0, z_0)$ with $(1,z_2)$ and curve $v_4$ connects the point $(0, z_1)$ with $(1,1)$. Both of these curves are monotonically increasing. 

As before, 
 the left %vertical 
 boundary $u_3$ of $V_1$ is  mapped to %the set
 $\{ (0, \Tilde{v}): z_0\leq \Tilde{v}\leq z_2\}$, while the right  
 %vertical 
 boundary $u_4$ of $V_1$ is  mapped to the set $\{ (1, \Tilde{v}): z_2\leq \Tilde{v}\leq 1\}$. This finally proves that $U_1$ is  a horizontal strip, with $V_1$ and $U_1$ satisfying item~\textit{(\ref{item:Moser:boundaries})}.
\end{proof}

In our final result we give the explicit formula for the curve $d_0(a,k)$ for the standard cubic linearity~\eqref{eqn:intro:cubic}.

\begin{lemma}\label{lemma:chaos:cubic}
Consider the setting of Proposition~\ref{prop:prop_failure:steady_sols:existence}, let $g$ be the standard cubic nonlinearity and define the function $d_0$  by~\eqref{eqn:cubic:d0}. Then for any  $0<d<d_0(a,k)$ condition \Hd holds.
\end{lemma}
\begin{proof}
One can check that for $d>0$ the quadratic inequalities
\begin{equation*}
\begin{aligned}
    d(k+1)(v-1) - v(1-v)(v-a) &> 0, \\
   d(k+1)v - v(1-v)(v-a) &< 0
\end{aligned}
\end{equation*}
 have a solution in the set of real numbers if and only if $0<d<d_0(a,k)$. 
\end{proof}

\section{Numerical examples}\label{sec:numerical:examples}
In this final section we showcase two results of our numerical experiments. In the first example we fix the diffusion parameter $d$, the branching parameter $k$ and study the dependence of the wave speed $c$ on the detuning parameter $a$.

% \begin{figure}[t!]
%     \centering
%     \includegraphics[width = .9\textwidth]{fig/k23_equal.pdf}
%     \caption{The images in the first row show a colormap of the numerical speed $c$ obtained as a solution to the fixed point problem~\eqref{eqn:numerics:fp}, for $k=2$ (left) and $k=3$ (right). The strong contrast between dark purple and white, and between white and dark green depict the steep jump between the values $|c|\gg 0$ and $|c|\approx 0$. In the second row we plot the mapping $a\mapsto c(a, 0.025, k)$ to visualize the pinning interval for $d=0.025$. \note{VS: but $c(a,d)$ is a continuous function...} } 
%     \label{3:fig:numerics:k23}
% \end{figure}

\begin{figure}[t!]
\centering
\begin{subfigure}{0.49\textwidth}
        \centering
        \includegraphics[width=\textwidth]{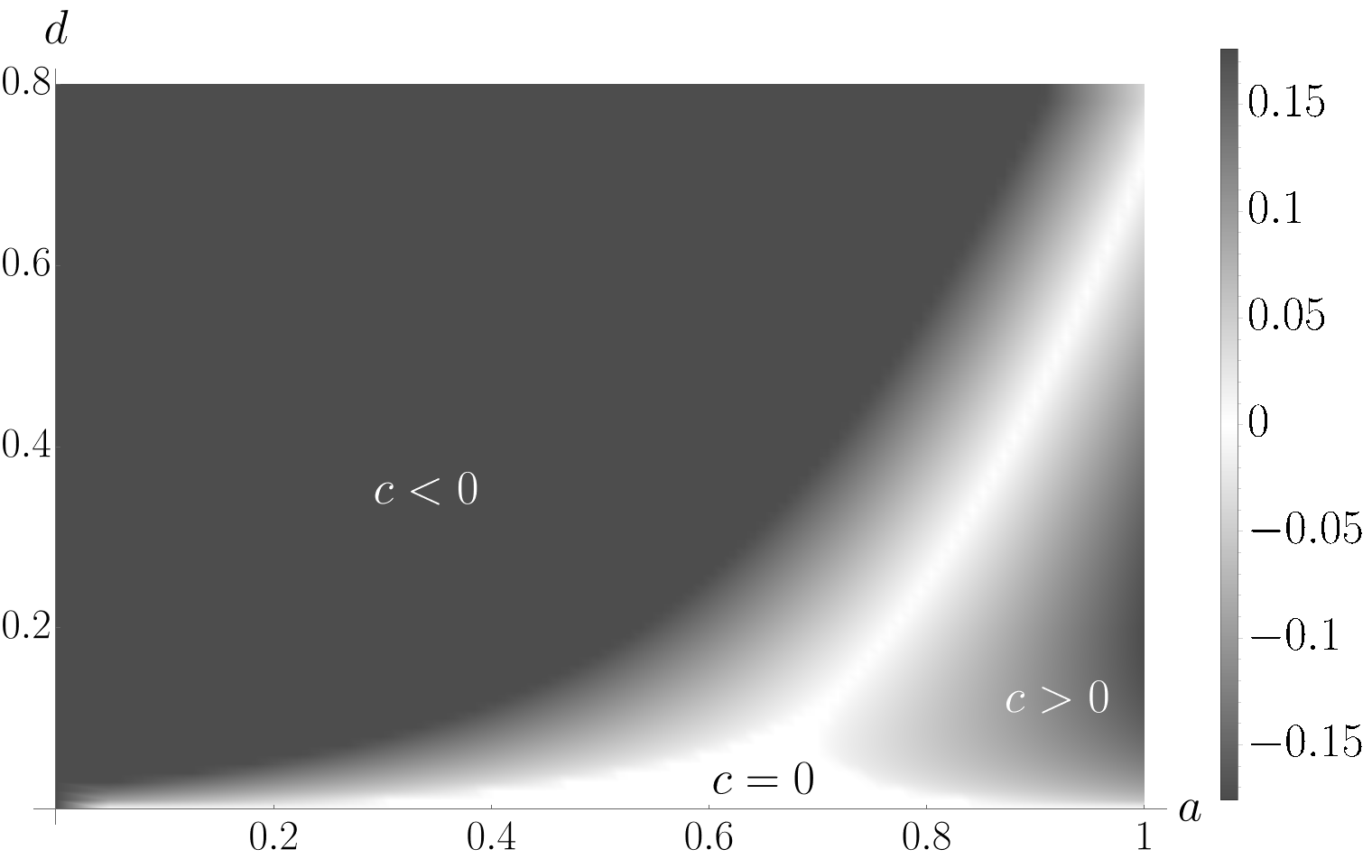}
        \caption{Heat map, $k=2$}
    \end{subfigure}
    \hfill
\begin{subfigure}{0.49\textwidth}
        \centering
        \includegraphics[width=\textwidth]{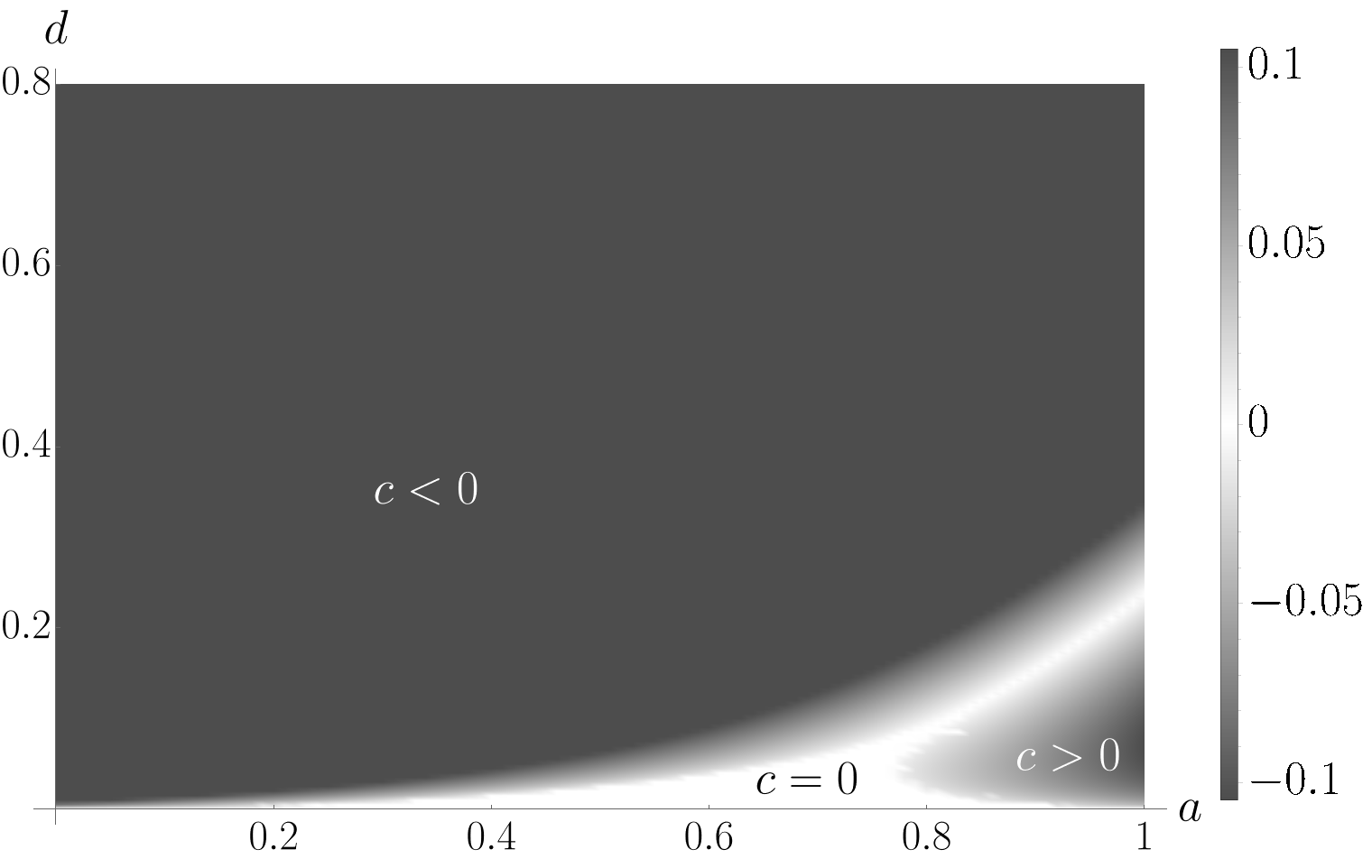}
        \caption{Heat map, $k=3$}
    \end{subfigure}
    \begin{subfigure}{0.49\textwidth}
        \begin{flushleft}
        \includegraphics[trim = 40 0 -100 0, width=\linewidth]{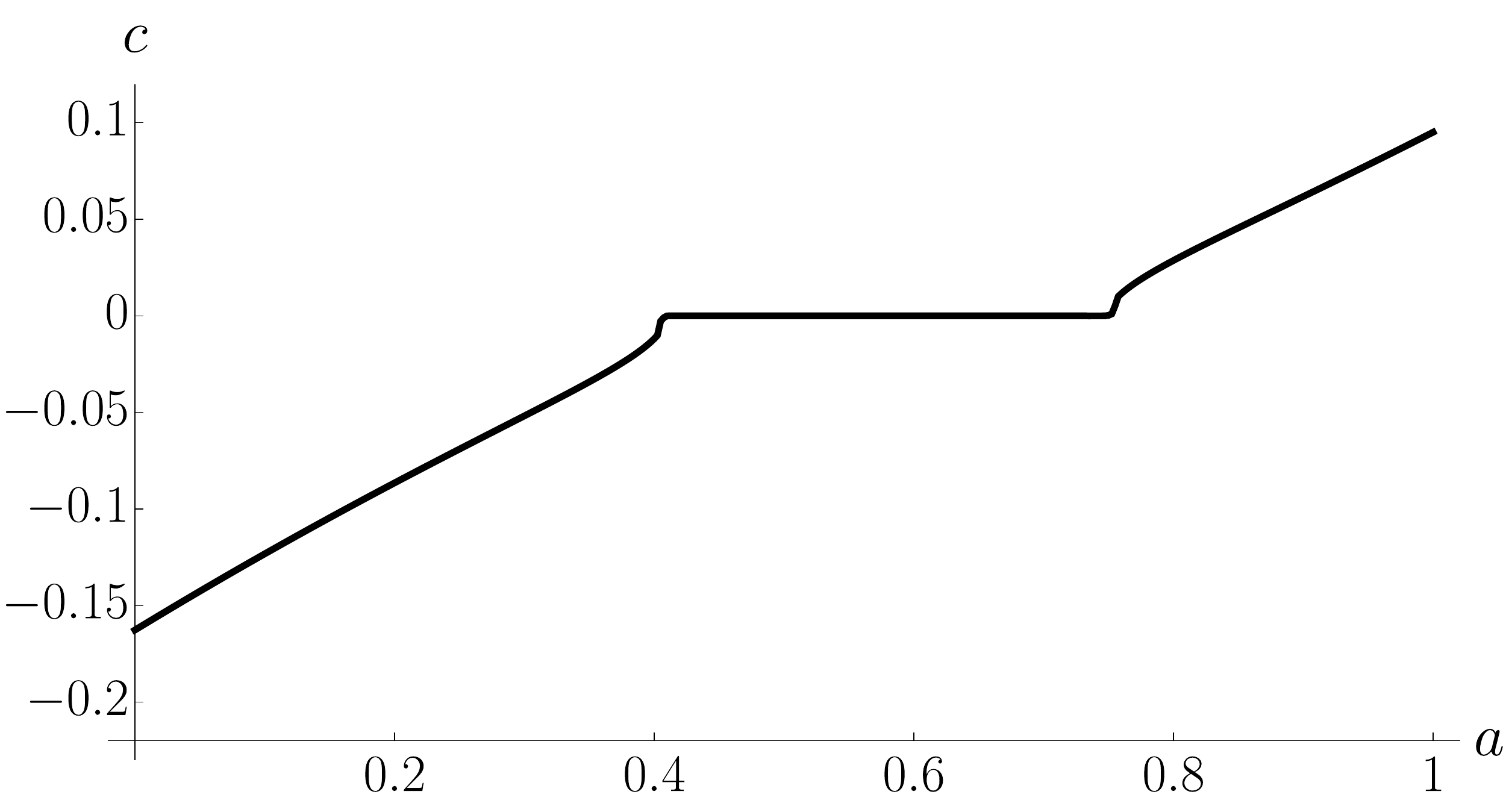}
        \caption{Speed, $k=2$}
        \end{flushleft}
    \end{subfigure}
    \begin{subfigure}{0.49\textwidth}
        \centering
        \includegraphics[trim = 20 0 -80 0, width=\linewidth]{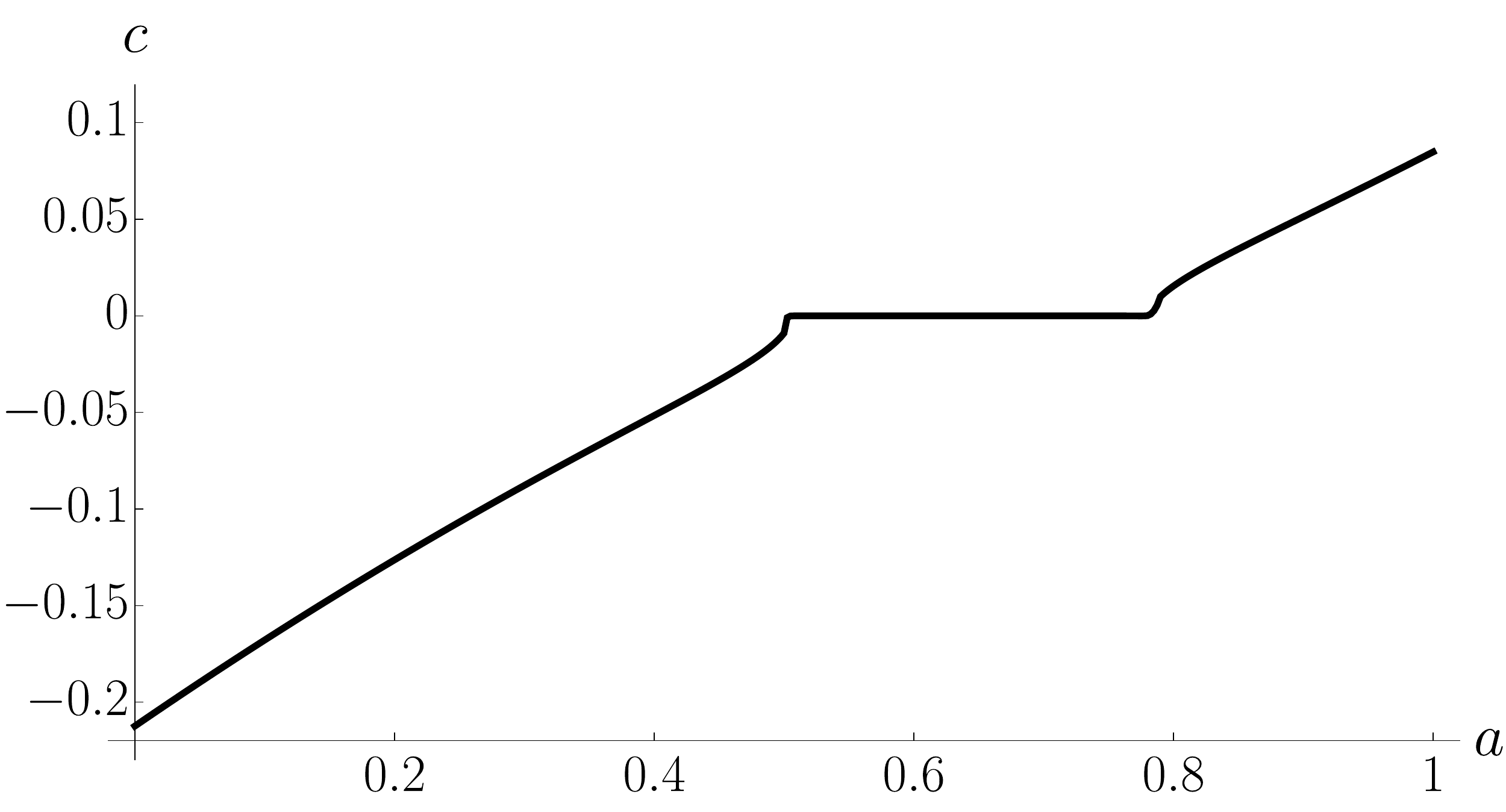}
        \caption{Speed, $k=3$}
    \end{subfigure}
\caption{The dependence of wave speed $c$ on the detuning parameter $a$, studied in Example~\ref{ex:speed_sign} for $k=2$ (left panels) and $k=3$ (right panels). The top panels show the heat maps, the darker the color the higher the value of $|c|$. The bottom panels then show the speed for fixed $d=0.025$.
 }
 \label{3:fig:numerics:k23}
\end{figure}

\begin{example}\label{ex:speed_sign} (Propagation direction)
In order to validate our theoretical findings for the standard cubic nonlinearity \eqref{eqn:intro:cubic}, we numerically solved the MFDE \eqref{eqn:main:MFDE}
on a domain $[-L, L]$ for some large $L\gg 1$ with boundary conditions $\Phi(-L) = 0$, $\Phi(L) = 1$. For fixed $(a,d)\in \mathcal{H}$, we divided our domain into $N_L \gg 1 $ segments. Upon writing %\todo[inline]{+[hjh: use $\Delta x$]} 
$\Delta x = 2L/N_L$ 
we have $N_L$ unknown variables - a speed $c$ and $N_{L}-1$ spatial points 
$$( \Phi_{1}, \dots  \Phi_{N_L - 1}),$$
where each point $\Phi_i$ approximates the value of $ \Phi(-L + i \Delta x)$. It is important to note that $N_L$ is chosen in such a manner that $1/\Delta x = I_0 \in \N$.

Moreover, we discretized the first derivatives in \eqref{eqn:main:MFDE} by the fourth order central difference scheme. The complete discretization scheme then takes the form
\begin{equation}\label{eqn:numerics:discrete_scheme}
    \begin{aligned}
      0=  &- \dfrac{c(8\Phi_{i+1} - 8\Phi_{i-1} -\Phi_{i+2} + \Phi_{i-2})}{12\Delta x} \\
      & \indent  - d\left(k\Phi_{i+I_0} - (k+1)\Phi_i  + \Phi_{i-I_0}\right)
     - g(\Phi_i;a) 
\end{aligned}
\end{equation}
for $i=1, \dots N_{L-1}$, to which we also add the boundary conditions $\Phi_{i}=0$ for all $i\leq 0$ and $\Phi_{i}=0$ for all $i\geq L$.    Adding the requirement 
\begin{equation}\label{eqn:numerics:discrete_scheme:2}
    \Phi_{\left\lfloor \frac{N_L}{2} \right\rfloor } - \frac{1}{2} = 0,
\end{equation} to compensate for the shift-invariance, we   rewrite this problem  in the compact form as
\begin{equation}\label{eqn:numerics:fp}
F(c, \Phi_{1}, \dots , \Phi_{N_L-2}, \Phi_{N_L - 1} ) = 0,    
\end{equation}
where the function $F:\R^{N_L}\to \R^{N_L} $ is derived from \eqref{eqn:numerics:discrete_scheme}-\eqref{eqn:numerics:discrete_scheme:2}. 

To this fixed point scheme we  applied a nonlinear fixed-point 
solver using the Python programming language. 
We present our results in Fig.~\ref{3:fig:numerics:k23} using a colormap representation, i.e., to each value of the numerical  speed $c$ we assign a color. The darker the color the more distinct it is from zero.

Since numerical computations never provide exact values, it is not straightforward to determine when the speed of the wave is exactly equal to $0$. Nevertheless, as the value of $a$ increases from $0$ to $1$, keeping $d$ fixed, one can observe that at some $a=a_-$ a harsh jump occurs between the values $|c|\gg0$, and $c\approx 0$. That is, the absolute value of the speed does not follow a smooth path but suddenly drops from values of the order $10^{-2}$ to values of the order $10^{-6}$ or even lower.  Similarly, for some $a=a_+$ the numerical speed suddenly rises from the low-order values back to the smooth trajectory. 
In view of the fact that $c$ is a smooth, monotonic function with respect to the detuning parameter $a$ whenever $c\neq 0$, we 
simply set $c=0$ in this region $[a_-, a_+]$. In Fig.~\ref{3:fig:numerics:k23}, the numerical pinning region is depicted in white. 
We observe that the `cone' in which $c=0$ becomes smaller as we increase $k$, which is in line with our theoretical results. 
\end{example}
\begin{figure}[t!]
\centering
    \begin{subfigure}{0.49\textwidth}
         \centering
         \includegraphics[width=\textwidth]{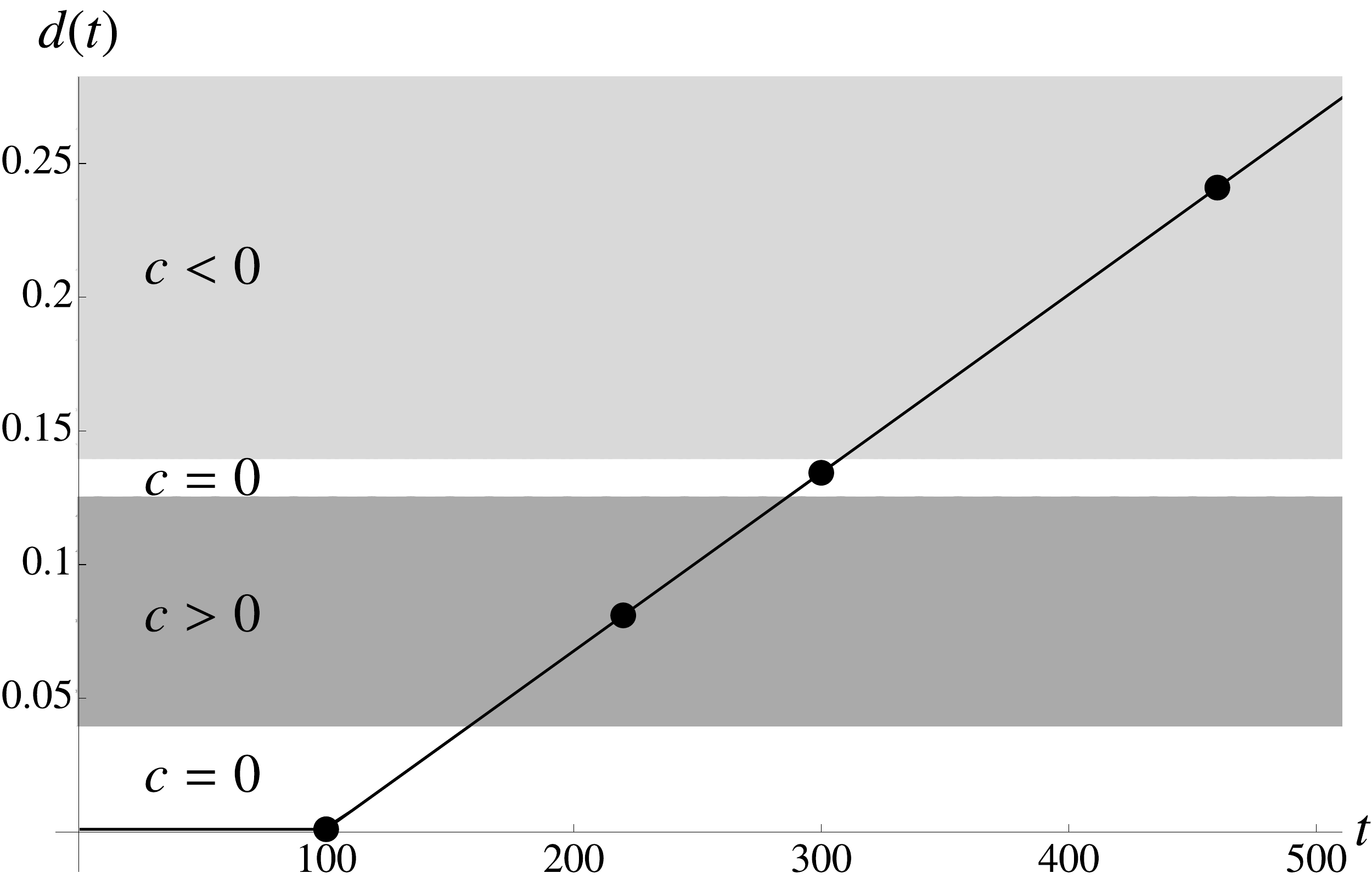}
     \end{subfigure}
     \begin{subfigure}{0.49\textwidth}
         \centering
         \includegraphics[width=\textwidth]{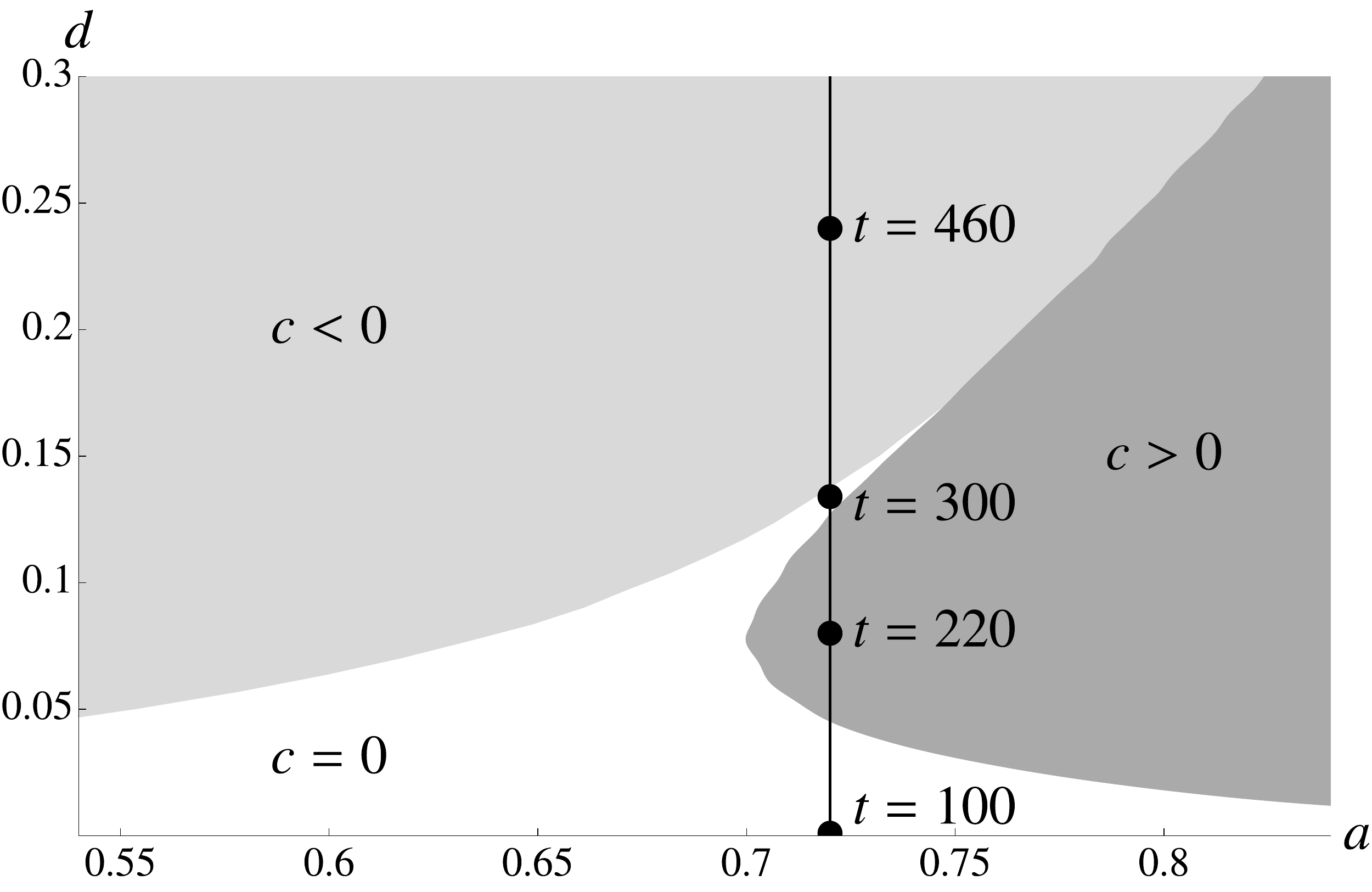}
    \end{subfigure}
    \caption{Illustration of the diffusion-driven propagation reversal discussed in  Ex.~\ref{ex:propagation:reversal}. The nonconstant time-dependent diffusion $d(t)$ given by \eqref{eqn:ex:reversal:d(t)} (left) and its trajectory through the $(a,d)$ plane (right). The profiles in highlighted time entries $t=100$, $220$ $300$ and $460$ are visualised in Fig.~\ref{fig:ex:propagation:reversal}.}
    \label{fig:ex:propagation:reversal:diagrams}
\end{figure}

In the second numerical experiment we highlight the phenomenon of the diffusion-driven propagation reversal. For this purpose vary the diffusion parameter $d$ and fix the branching parameter $k$ and the detuning parameter $a\approx 1$. This allows us to illustrate the diffusion-driven propagation reversal.

\begin{figure}[t!]
\centering
    % \fcolorbox{lightgray}{white}{
    % \begin{subfigure}{0.45\textwidth}
    %      \centering
    %      \includegraphics[width=\textwidth]{figVS/frkwvs-fig-trees-propreversal-u-100-bw.pdf}
    %      \includegraphics[width=\textwidth]{figVS/frkwvs-fig-trees-propreversal-G-100-bw.pdf}
    %  \end{subfigure}
    % }\hspace{2mm}
    % \fcolorbox{lightgray}{white}{
    %  \begin{subfigure}{0.45\textwidth}
    %      \centering
    %     \includegraphics[width=\textwidth]{figVS/frkwvs-fig-trees-propreversal-u-220-bw.pdf}
    %     \includegraphics[width=\textwidth]{figVS/frkwvs-fig-trees-propreversal-G-220-bw.pdf}
    %  \end{subfigure}
    %  }\\[2mm]
    % \fcolorbox{lightgray}{white}{
    % \begin{subfigure}{0.45\textwidth}
    %      \centering
    %     \includegraphics[width=\textwidth]{figVS/frkwvs-fig-trees-propreversal-u-300-bw.pdf}
    %     \includegraphics[width=\textwidth]{figVS/frkwvs-fig-trees-propreversal-G-300-bw.pdf}
    %  \end{subfigure}
    % }\hspace{2mm}
    % \fcolorbox{lightgray}{white}{
    %  \begin{subfigure}{0.45\textwidth}
    %      \centering
    %     \includegraphics[width=\textwidth]{figVS/frkwvs-fig-trees-propreversal-u-460-bw.pdf}
    %     \includegraphics[width=\textwidth]{figVS/frkwvs-fig-trees-propreversal-G-460-bw.pdf}
    %  \end{subfigure}
    % }
    
      \fcolorbox{lightgray}{white}{
    \begin{subfigure}{0.45\textwidth}
         \centering
         \includegraphics[width=\textwidth]{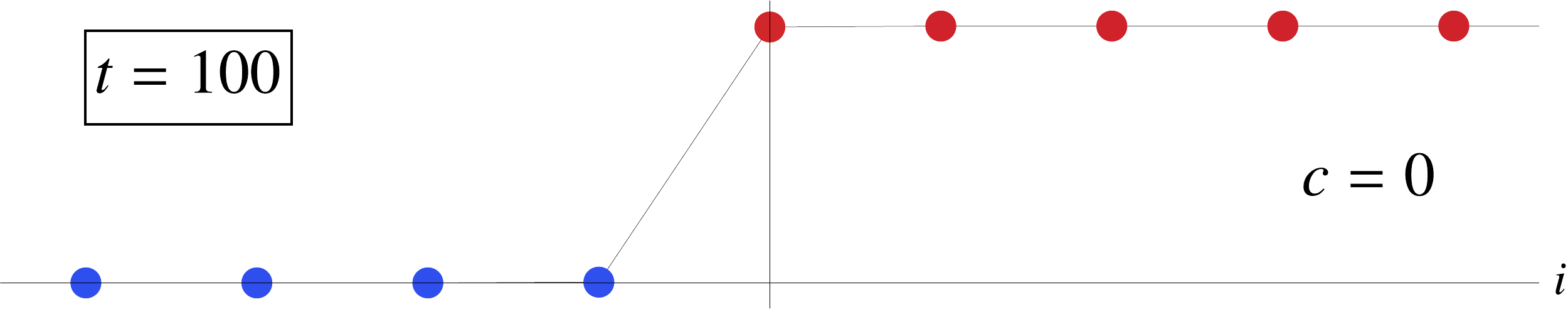}
         \includegraphics[width=\textwidth]{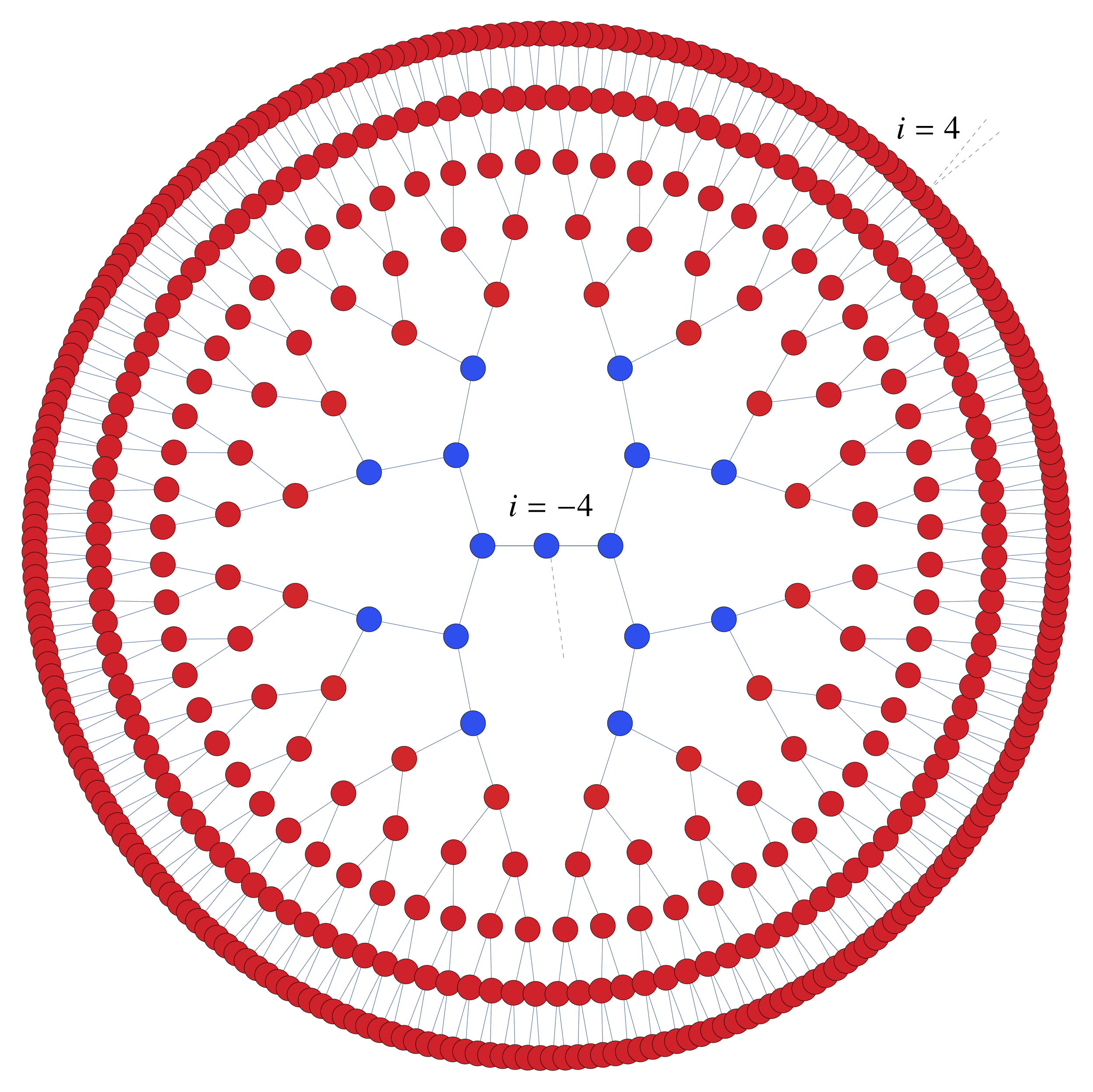}
     \end{subfigure}
    }\hspace{2mm}
    \fcolorbox{lightgray}{white}{
     \begin{subfigure}{0.45\textwidth}
         \centering
        \includegraphics[width=\textwidth]{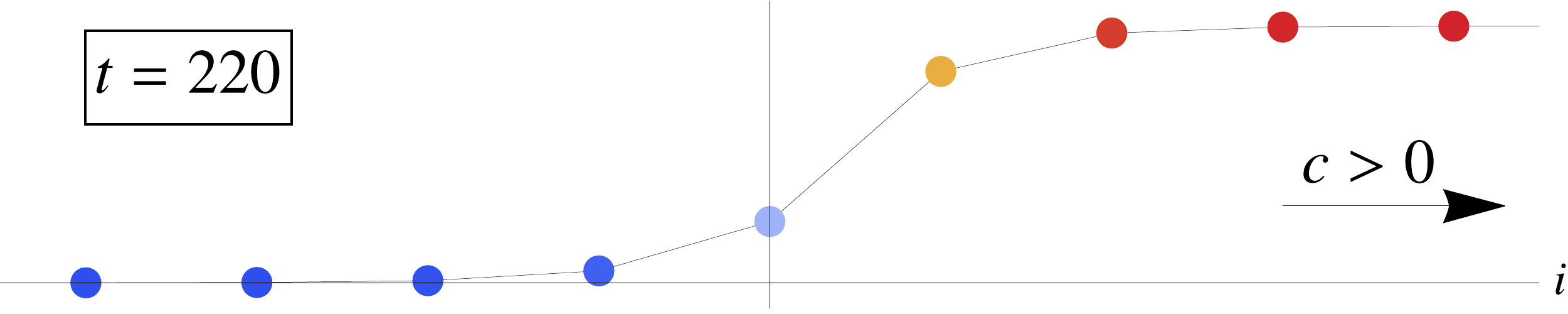}
        \includegraphics[width=\textwidth]{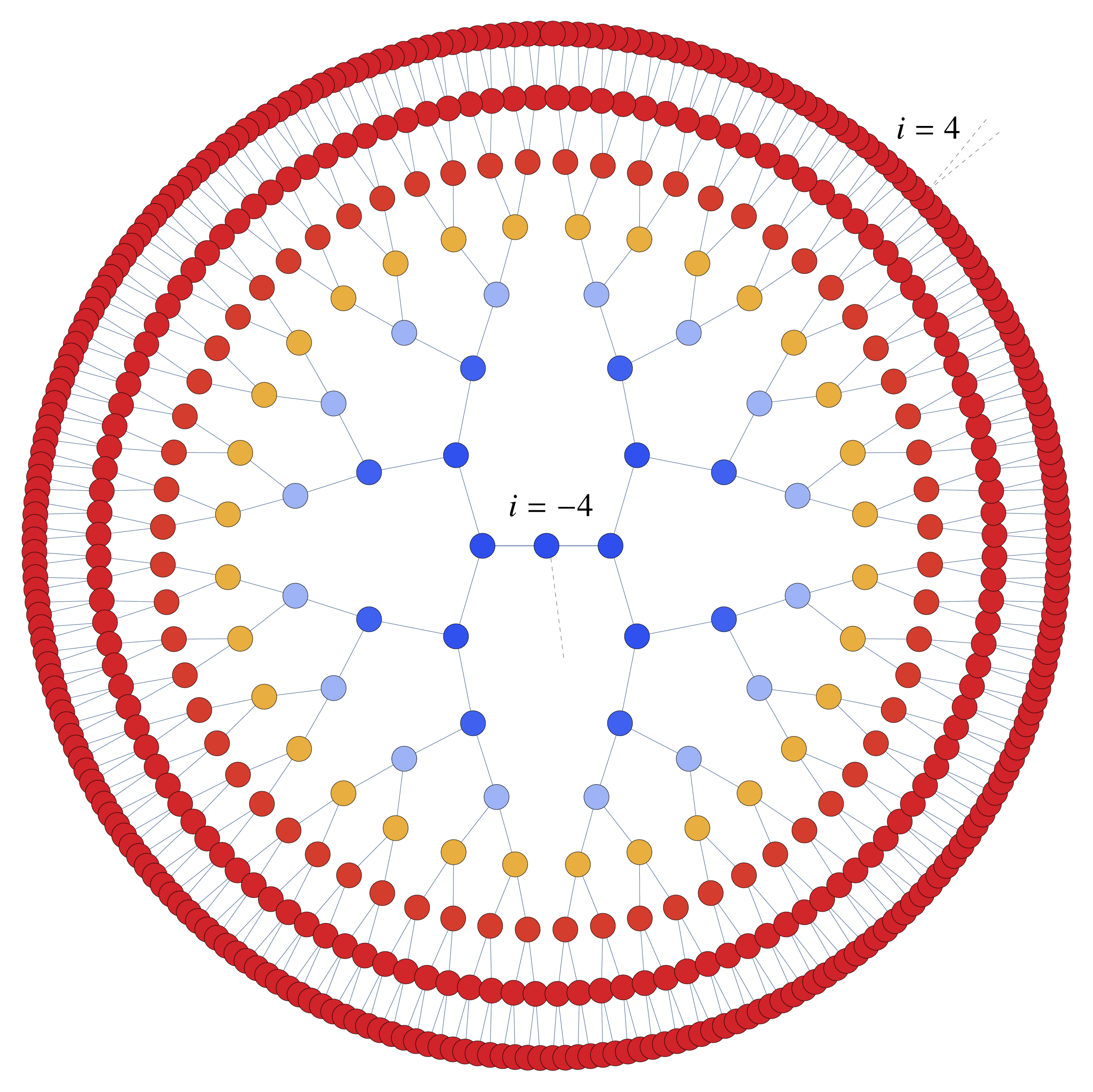}
     \end{subfigure}
     }\\[2mm]
    \fcolorbox{lightgray}{white}{
    \begin{subfigure}{0.45\textwidth}
         \centering
        \includegraphics[width=\textwidth]{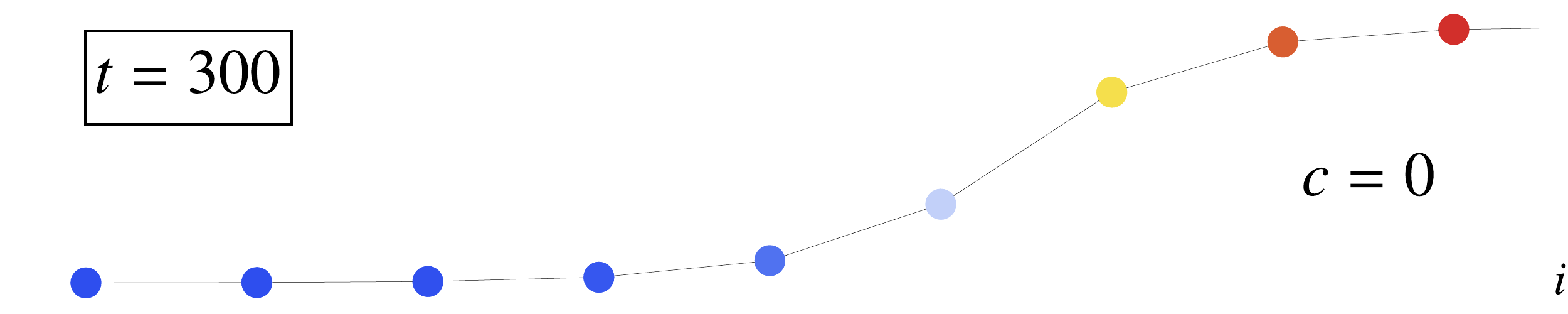}
        \includegraphics[width=\textwidth]{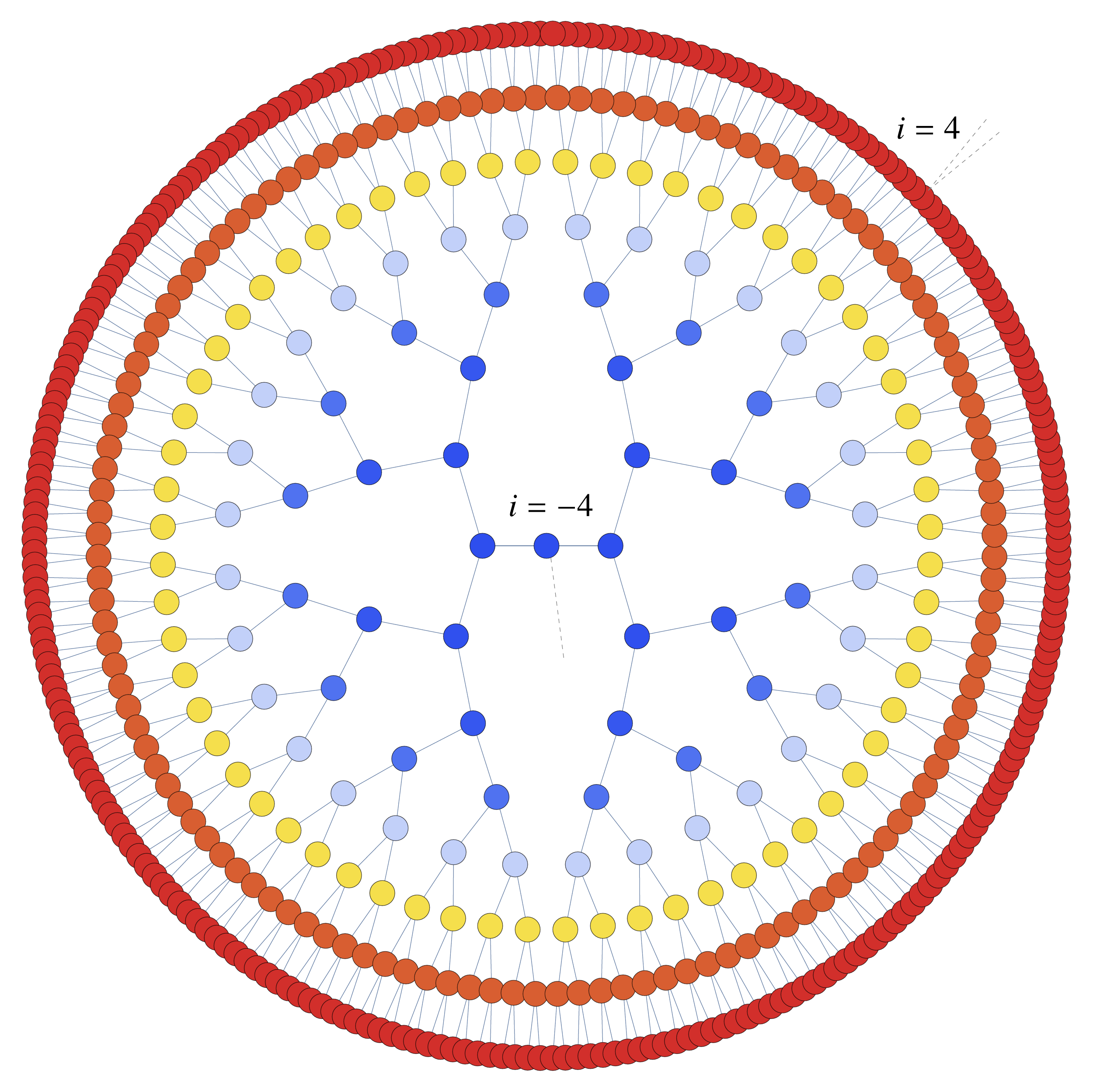}
     \end{subfigure}
    }\hspace{2mm}
    \fcolorbox{lightgray}{white}{
     \begin{subfigure}{0.45\textwidth}
         \centering
        \includegraphics[width=\textwidth]{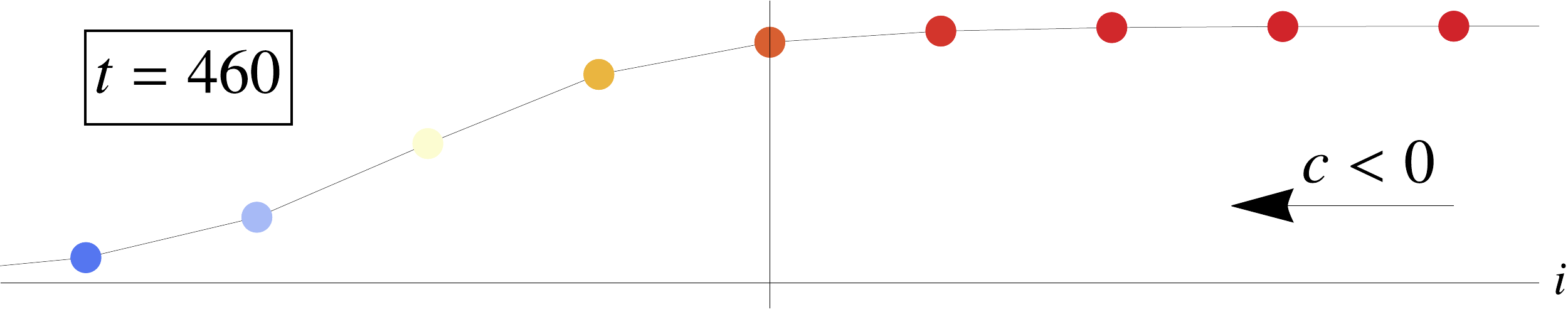}
        \includegraphics[width=\textwidth]{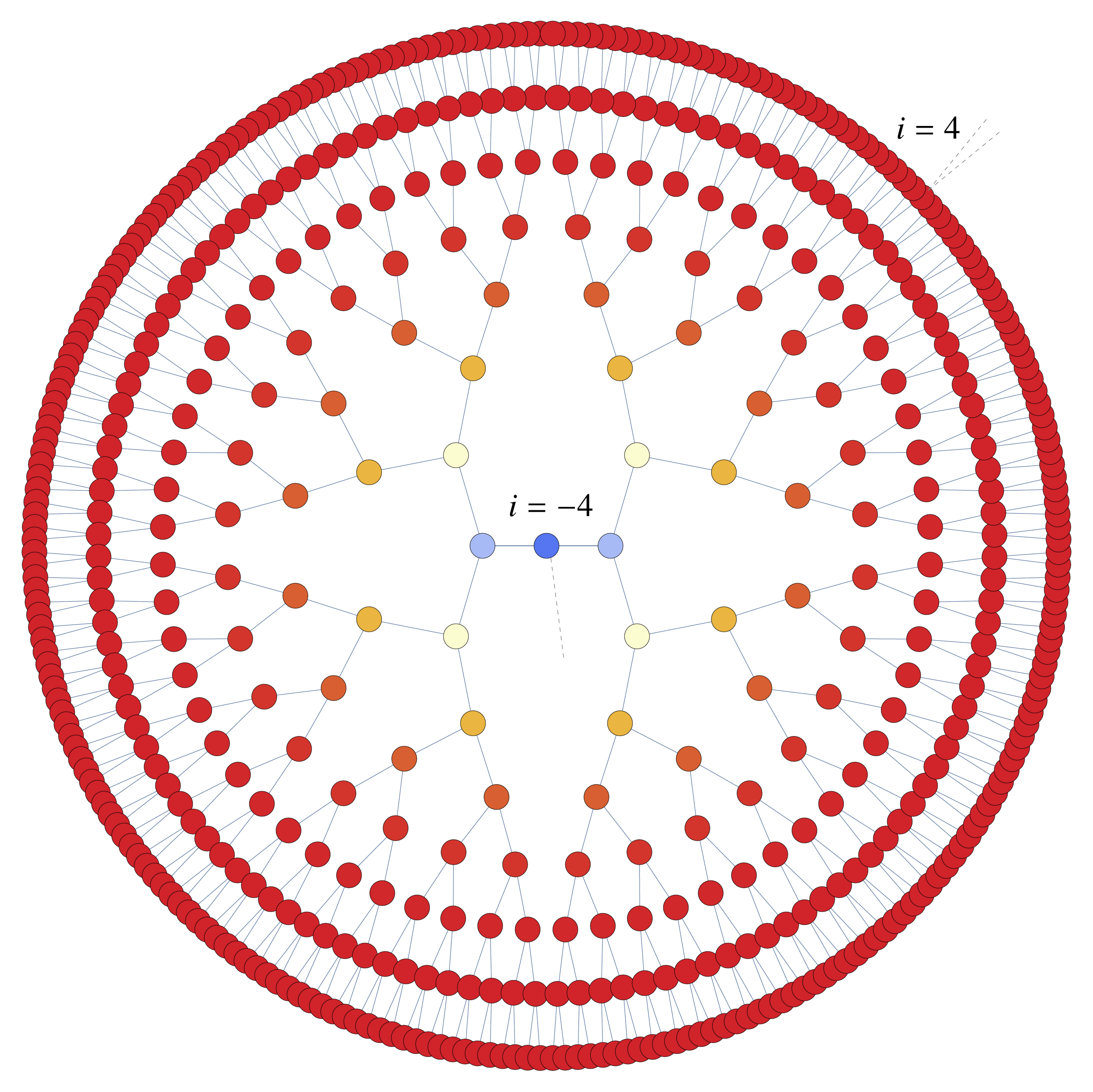}
     \end{subfigure}
    }
    
     \caption{
     Illustration of the diffusion-driven propagation reversal discussed in  Ex.~\ref{ex:propagation:reversal}. %The top panels show the nonconstant time-dependent diffusion $d(t)$ given by \eqref{eqn:ex:reversal:d(t)} (left) and its trajectory through the $(a,d)$ plane (right). 
     The top panels in each frame display the solutions $u_i(t)$ of \eqref{eqn:ex:reversal:eq} at $t=100$, $220$ $300$ and $460$, see Fig.~\ref{fig:ex:propagation:reversal:diagrams}. The bottom panels in each frame  depict the corresponding solutions of equation~\eqref{eqn:intro:trees} on the binary tree $\mathcal{T}_2$. Only the layers with $i=-4,-3,\ldots,3,4$ are visualised. To see this figure in color, please go online.}\label{fig:ex:propagation:reversal}
    \end{figure}

\begin{example}[Propagation reversal]\label{ex:propagation:reversal}
In this example we illustrate the diffusion-driven propagation reversal. Let the non-constant-diffusion be given by
\begin{equation}\label{eqn:ex:reversal:d(t)}
d(t)=\begin{cases}
 .001 & t\leq 100, \\
 .001+\frac{1}{1500}(t-100) & t>100,
\end{cases}
\end{equation}
illustrated in the left panel of Fig.~\ref{fig:ex:propagation:reversal:diagrams}. We consider the bistable differential equation \eqref{eqn:intro:kLDE} on the binary tree $\mathcal{T}_2$
\begin{equation}\label{eqn:ex:reversal:eq}
\begin{cases}
\dot{u}_i(t) = d(t) \left(2u_{i+1}(t) - 3u_i(t) + u_{i-1}(t)\right) + g(u_i(t);.72)\\
{u}_i(0) = \begin{cases} 0 & i<0\\ 1 & i\geq 0. \end{cases}
\end{cases}
\end{equation}
with the cubic bistability \eqref{eqn:intro:cubic} with fixed $a=.72$.

In particular, as we increase the diffusion parameter $d$ we expect the wave to go through four phases. This
is numerically confirmed by the results in Figs.~\ref{fig:ex:propagation:reversal:diagrams} and \ref{fig:ex:propagation:reversal}.
Indeed, for $d>0$ sufficiently small, the wave is pinned $(c=0)$. As we increase $d$ the wave moves to the right ($c>0$, or outwards in the circular depiction of $\mathcal{T}_2$), then it is pinned again $c=0$ and once the diffusion is sufficiently strong %it crosses a threshold $d^*(.72,2)$
it propagates to the left ($c<0$, or inwards in the circular depiction of $\mathcal{T}_2$).
We note that the transition boundaries for $d$ correspond well with the numerical results from Example \ref{ex:speed_sign}.

%pinning for small $d$, i.e., $c=0$; spreading through the $k$-ary tree ($c>0$);  pinning again; and finally wave retreat, i.e., $c<0$ for large $d$. Our expectations are numerically confirmed and we show our results in the bottom four panels of Fig.~\ref{3:fig:ex:propagation:reversal}. We note that the wave direction aligns with the results from the numerical investigations from Example \ref{ex:speed_sign}. 

%Note that for $d>0$ sufficiently small, the wave is pinned $c=0$. As we increase $d$ the wave moves right $c>0$ (or outwards in the circular depiction of $\mathcal{T}_2$, then it is pinned again $c=0$ and once it crosses a threshold $d^*(.72,2)$ it propagates left $c<0$ (or inwards in the circular depiction of $\mathcal{T}_2$).
\end{example}

%\todo[color=green]{VS: I like these pictures a lot (Figure 12). The hint of a continuation of the tree helps (in my eyes). Maybe putting some frame around the pictures would help the reader to easily assign the corresponding graphs and profiles...}

% \clearpage

% \section{Discussion}
% \section{Notes too big in volume to be included in the comments}
% \subsection{Connection of advection-reaction-diffusion equation and reaction diffusion equation}
% Let us have an advection-reaction-diffusion equation given by
% \[
% u_t(x,t) = \alpha u_{xx}(x,t) + \beta u_x(x,t) + g(u(x,t);a), \quad x\in\mathbb{R}, \quad t>0
% \]
% with some parameters $\alpha, \beta \in \mathbb{R} \setminus \{0 \}$. We use a substitution which in general tilts the coordinate system such that the new space variable $\xi$ is now following the direction of the characteristics of the problem
% \[
% u_t(x,t) = \beta u_x(x,t),
% \]
% i.e.
% \[
% \xi := x+\beta t, \qquad \tau := t, \qquad v(\xi,\tau) := u(x,t)
% \]
% Using the derivatives of a compound function, one has
% \[
% u_t = \frac{\partial}{\partial t} v(\beta t + x,t) = \beta v_\xi + v_\tau, \quad u_x = v_\xi, \quad u_{xx}= v_{\xi\xi}
% \]
% in which $v_\xi$ and $v_\tau$ are partial derivatives of $v$ in the first and  the second variable, respectively. We can now rewrite the advection-reaction-diffusion equation without the advection term
% \[
% v_\tau(\xi,\tau) = \alpha v_{\xi\xi}(\xi,\tau) + g(v(\xi,\tau);a), \quad \xi\in\mathbb{R}, \quad \tau>0. 
% \]

\paragraph{Acknowledgements} HJH and MJ acknowledge support from the Netherlands Organization for Scientific Research (NWO) (grant 639.032.612). PS and V\v{S} gratefully acknowledge the support by the Czech Science Foundation grant no. GA22-18261S.

% \clearpage

\bibliographystyle{klunumHJ}
\bibliography{ref}

%\printbibliography

\end{document}